\documentclass[11pt]{article}
\usepackage{graphicx} 
\usepackage[english]{babel}
\usepackage[utf8]{inputenc}
\usepackage[letterpaper, margin=1in]{geometry}
\usepackage{authblk}
\usepackage{csquotes}

\usepackage{fancyhdr}

\usepackage{amsmath}
\usepackage{amsfonts}
\usepackage{amssymb}
\usepackage{amsthm}
\usepackage[only,llbracket,rrbracket]{stmaryrd}

\usepackage{hyperref}
\usepackage[capitalise, noabbrev]{cleveref}

\usepackage{tikzit}

\tikzstyle{basic}=[fill=black, draw=black, shape=circle, scale=0.5]
\tikzstyle{new style 0}=[fill={rgb,255: red,16; green,112; blue,255}, draw={rgb,255: red,16; green,112; blue,255}, shape=circle, scale=0.5]
\tikzstyle{new style 1}=[fill={rgb,255: red,191; green,191; blue,191}, draw={rgb,255: red,191; green,191; blue,191}, shape=circle]
\tikzstyle{new style 2}=[fill={rgb,255: red,182; green,227; blue,160}, draw=black, shape=circle]
\tikzstyle{new style 3}=[fill={rgb,255: red,255; green,143; blue,121}, draw=black, shape=circle]
\tikzstyle{new style 4}=[fill={rgb,255: red,184; green,255; blue,255}, draw=black, shape=circle]
\tikzstyle{new style 5}=[fill={rgb,255: red,255; green,248; blue,190}, draw=black, shape=circle]
\tikzstyle{new style 6}=[fill={rgb,255: red,16; green,112; blue,255}, draw=black, shape=circle]

\tikzstyle{new edge style 0}=[-, draw={rgb,255: red,191; green,0; blue,64}]
\tikzstyle{new edge style 1}=[-, draw={rgb,255: red,17; green,0; blue,255}]
\tikzstyle{new edge style 2}=[-, draw={rgb,255: red,20; green,153; blue,0}]
\tikzstyle{new edge style 3}=[-, dashed]
\tikzstyle{new edge style 4}=[-, fill={rgb,255: red,198; green,255; blue,255}]
\tikzstyle{new edge style 5}=[-, fill=white]
\tikzstyle{new edge style 6}=[-, fill={rgb,255: red,255; green,191; blue,191}]
\tikzstyle{new edge style 7}=[-, fill={rgb,255: red,182; green,227; blue,160}, draw=black]
\tikzstyle{new edge style 8}=[-, draw=none]
\tikzstyle{new edge style 9}=[-, fill={rgb,255: red,255; green,191; blue,191}]
\tikzstyle{new edge style 10}=[-, draw={rgb,255: red,16; green,112; blue,255}, fill={rgb,255: red,198; green,255; blue,255}]
\tikzstyle{new edge style 11}=[-, fill={rgb,255: red,191; green,191; blue,191}]
\tikzstyle{new edge style 12}=[-, draw=white, fill=white]
\tikzstyle{new edge style 13}=[-, fill={rgb,255: red,255; green,191; blue,191}, draw=white]
\tikzstyle{new edge style 14}=[-, fill={rgb,255: red,198; green,255; blue,255}, draw=white]
\tikzstyle{new edge style 15}=[-, fill={rgb,255: red,182; green,227; blue,160}, draw={rgb,255: red,16; green,112; blue,255}]
\tikzstyle{new edge style 16}=[-, draw={rgb,255: red,191; green,191; blue,191}]
\tikzstyle{new edge style 17}=[-, draw={rgb,255: red,198; green,255; blue,255}]
\tikzstyle{new edge style 19}=[->]

\usepackage{subcaption}
\usepackage{float}

\makeatletter
\newtheorem*{rep@theorem}{\rep@title}
\newcommand{\newreptheorem}[2]{%
\newenvironment{rep#1}[1]{%
 \def\rep@title{#2 \ref{##1}}%
 \begin{rep@theorem}}%
 {\end{rep@theorem}}}
\makeatother

\theoremstyle{plain}
\newtheorem{theo}{Theorem}[section]
\newreptheorem{theo}{Theorem}
\newtheorem{lem}[theo]{Lemma}
\newtheorem{sublem}{Lemma}[theo]

\newtheorem{prop}[theo]{Proposition}
\newreptheorem{prop}{Proposition}
\newtheorem{cor}[theo]{Corollary}
\newreptheorem{cor}{Corollary}
\newtheorem*{prob}{Open problem}
\newtheorem{claim}{Claim}[theo]

\theoremstyle{definition}
\newtheorem{defi}{Definition}[section]

\crefname{theo}{Theorem}{Theorems}
\Crefname{theo}{Theorem}{Theorems}
\crefname{lem}{Lemma}{Lemmas}
\Crefname{lem}{Lemma}{Lemmas}
\crefname{sublem}{Lemma}{Lemmas}
\Crefname{sublem}{Lemma}{Lemmas}
\crefname{conj}{Conjecture}{Conjectures}
\Crefname{conj}{Conjecture}{Conjectures}
\crefname{prop}{Proposition}{Propositions}
\Crefname{prop}{Proposition}{Propositions}
\crefname{cor}{Corollary}{Corollaries}
\Crefname{cor}{Corollary}{Corollaries}
\crefname{prob}{Problem}{Problems}
\Crefname{prob}{Problem}{Problems}
\crefname{claim}{Claim}{Claims}
\Crefname{claim}{Claim}{Claims}
\crefname{subclaim}{Claim}{Claims}
\Crefname{subclaim}{Claim}{Claims}
\crefname{defi}{Definition}{Definitions}
\Crefname{defi}{Definition}{Definitions}
\crefname{subsection}{Subsection}{Subsections}
\Crefname{subsection}{Subsection}{Subsections}

\newenvironment{proof_claim}{\noindent\textit{Proof of the claim.}}{\hfill$\blacksquare$}

\definecolor{Tolpurple}{HTML}{AA3377}

\title{An almost quadratic bound for the minimal excluded minors for a surface}
\author{}
\author{Sarah Houdaigoui\\
\href{mailto:shoudaigoui@nii.ac.jp}{shoudaigoui@nii.ac.jp}\\ 
National Institute of Informatics, Tokyo, Japan\\
Graduate university for advanced studies (SOKENDAI), Hayama, Japan
\and Ken-ichi Kawarabayashi\footnote{Research supported by JSPS Kakenhi 22H05001 and by JST ASPIRE JPMJAP2302.}\\
\href{mailto:k_keniti@nii.ac.jp}{k\_keniti@nii.ac.jp}\\
National Institute of Informatics, Tokyo, Japan\\
The University of Tokyo, Japan
}
\date{}

\begin{document}

\maketitle

\begin{abstract}
As part of the graph minor project, Robertson and Seymour showed in 1990 that the class of graphs embeddable in a given surface can be characterized by a finite set of minimal excluded minors \cite{GM8}. However, the proof is purely existential and therefore provides no explicit information about these excluded minors. In 1993, Seymour established the first general upper bound on the order of such minimal excluded minors \cite{seymour}. Recently, Houdaigoui and Kawarabayashi improved this result by deriving a quasi-polynomial upper bound \cite{HK2026}. Despite this advance, the gap between this bound and the known linear lower bound $\Omega(g)$ (where $g$ denotes the genus) remains substantial. 
In particular, they conjectured that a polynomial upper bound should hold.

In this paper, we confirm this conjecture by showing that the order of the minimal excluded minors for a surface of genus $g$ is $g^{2+o(1)}$.
This result significantly narrows the gap between the known lower and upper bounds, bringing the asymptotic behavior much closer to the conjectured optimum.

Our approach relies on a new structural property of minimal excluded minors. Let $G$ be a minimal excluded minor for a surface of Euler genus $g$. Houdaigoui and Kawarabayashi \cite{HK2026} showed that $G$ contains $O(\log g)$ pairwise disjoint cycles that are contractible and nested in some embedding of $G$. We strengthen this result by proving a separator-based variant: for any contractible subgraph $H \subseteq G$ with a separator of size $s$ (with $H$ contained entirely in one side), the subgraph $H$ contains $O(\log s)$ disjoint cycles that are contractible and nested in some embedding of $G$. This allows us to replace a genus-dependent bound with a separator-dependent one, which is the main new ingredient in deriving our polynomial bound.
\end{abstract}

\newpage
\section{Introduction}

The graph minor theorem \cite{GM20}, proved by Robertson and Seymour in 2004, states that every class of graphs closed under minors can be characterized by a finite list of minimal excluded minors  (which we will denote by \textit{excluded minors} from now on). Understanding the structure of the excluded minors for various minor-closed graph classes is a very active area of research. In particular, exploring the link between minor-monotone graph parameters and the presence of specific graphs as minors has been a very fruitful area of research.
The most iconic example of such a result is the grid minor theorem \cite{GM5}, which states that there exists a function $f$ such that every graph of treewidth at least $f(|V(H)|)$ contains a given grid $H$ as a minor. Extensive work has been devoted to finding lower and upper bounds on $f$; see \cite{Chekuri_chuzhoy, Chuzhoy_Tan} for the current state of the art. Beyond the function $f$ itself, the structure of the excluded minors characterizing bounded treewidth has also been studied \cite{APC1990,Ramachandramurthi,Lagergren1998,Chlebikova}. Analogous results to the grid minor theorem are known for pathwidth \cite{Kinnersley, TUK} and for treedepth \cite{Nesetril_POM2012,Kawarabayashi_Rossman} (see \cite{Nesetril_POM2012} for a definition of the latter parameter). Recently, Rambaud \cite{RambaudExcludingRectangularGrid2025} strengthened both the grid minor theorem and the excluded-minor characterization of bounded treedepth, obtaining an excluded-minor characterization for infinitely many parameters interpolating between treedepth and treewidth. We refer the reader to \cite{IST2023} for an excellent survey of excluded minors for various parameters and graph classes, and of their interplay with parameterized complexity.

We now turn to another minor-monotone graph parameter that has long attracted considerable attention: the genus. The study of graph embeddings on surfaces predates the graph minor project by several decades. However, the graph minor project profoundly renewed this area by reframing the study of surfaces through the perspective of excluded minors.

In 1930, Kuratowski \cite{Kuratowski1930} showed that planar graphs are precisely those that do not contain a subdivision of $K_5$ or $K_{3,3}$ as a subgraph. Subsequently, in 1937, Wagner showed that planar graphs are precisely those that do not contain $K_5$ or $K_{3,3}$ as a minor. These two results can easily be shown to be equivalent.
Finding analogues of Kuratowski's and Wagner's theorems for higher genus surfaces was for a long time a central open problem: in 1978, Glover and Huneke showed that there are finitely many minimal graphs excluded as subdivisions (thereafter minimal excluded subgraphs) for the projective plane \cite{Glover_Huneke}; then, in 1979, Glover, Huneke and Wang presented a list of 103 minimal excluded subgraphs \cite{Glover_Huneke_Wang}; finally, Archdeacon proved in 1980 that this list is complete \cite{Archdeacon_1980, Archdeacon_1981}. 
It is easy to see that a class of graphs is characterized by a finite set of excluded minors if and only if it is characterized by a finite set of minimal excluded subgraphs (see, e.g., \cite[Proposition 6.1.1]{graphs_on_surfaces}).
In 1989, Archdeacon and Huneke extended the result from the projective plane to every non-orientable surface by showing that every non-orientable surface has only a finite number of minimal excluded subgraphs. Robertson and Seymour finally settled the question by showing that every surface has only finitely many excluded minors in a preliminary result of the celebrated graph minor theorem \cite{GM8}. 
Simpler proofs of the graph minor theorem for surfaces have been found by Mohar \cite{Mohar_2001} and Thomassen \cite{Thomassen}.

However, the proof of the graph minor theorem is purely existential and hence provides no information about how to determine excluded minors for a surface. Therefore, except for the plane and the projective plane, the exact list of excluded minors for each surface is unknown. Considerable effort has been devoted to the torus: several characterizations of excluded minors and excluded topological minors satisfying additional structural conditions have been obtained; see \cite{Duke_Haggard, Decker_1978, Bodendiek_Wagner, Neufeld1993, Juvan_1995, NM1997, Hlavacek_1997, Chambers2002, Woodcock2006, GMC2009, Mohar_Skoda2014, MW2018}. 
In particular, encouraged by the access to more powerful computers, recent work has relied heavily on computational methods to compute excluded minors for the torus.
However, to this day, no complete list is known. Currently, at least 17,523 excluded minors for the torus are known \cite{MW2018}. It is, therefore, unrealistic in the foreseeable future to hope that a characterization of excluded minors will be explicitly found for surfaces of Euler genus exceeding that of the torus. 

Recently, Mohar and {\v{S}}koda \cite{Mohar_Skoda2024_I, Mohar_Skoda2024_II} found the complete list of excluded minors of connectivity $2$ for the Klein bottle. Along similar lines, Robertson and Seymour recently published results, originally proved in the 1990s, that establish a link between excluding Kuratowski graphs and being almost embeddable in a surface \cite{RS2026,RS2026_2}.

Consequently, the research community has been trying to bound the order of excluded minors on a fixed surface by a function of the surface's Euler genus. In 1993, Seymour gave the first and until quite recently only known upper bound on the order of an excluded minor for a surface \cite{seymour}.

\begin{theo}[\cite{seymour}]
    Let $S$ be a given surface of Euler genus $g$. Every excluded minor for $S$ has at most $2^{2^k}$ vertices where $k = (3g+9)^9$. 
\end{theo}

The only known lower bound on the maximum order of excluded minors for a surface is $\Omega(g)$, which is obtained by $g+1$ copies of $K_{3,3}$.
Therefore, the gap between known lower and upper bounds is huge. Recently, Houdaigoui and Kawarabayashi \cite{HK2026} finally improved the bound to a quasi-polynomial in the Euler genus $g$ of the surface:

\begin{theo}[\cite{HK2026}]
    \label{main_SODA}
    Let $S$ be a given surface of Euler genus $g$. Every excluded minor for $S$ is of order at most $g^{O(\log^3 g)}$.
\end{theo}

As mentioned in \cite{HK2026}, this bound is most likely far from optimal. Indeed, they conjectured that a polynomial bound should be possible. 
In this paper, we confirm this conjecture.
The main result of this paper is the following:

\begin{theo}
    \label{main}
    Let $S$ be a given surface of Euler genus $g$. Every excluded minor for $S$ is of order at most $U(g) = g^2 \times (\log g)^{O\left((\log\log g)^3\right)} = g^{2+o(1)}$.
\end{theo}

\paragraph{Technical contributions over \cite{HK2026}.}

Let us highlight the technical contributions of this paper. 
Let $G$ be an excluded minor for a surface of genus $g$.

Our approach is based on a new structural theorem for a minimal excluded minor $G$ for a surface of Euler genus $g$. Houdaigoui and Kawarabayashi \cite{HK2026} proved that $G$ contains $O(\log g)$ pairwise disjoint cycles that are contractible and nested in some embedding of $G$. We refine this result by establishing a separator-based variant: for any contractible subgraph $H$ of $G$ with a separator of size $s$, that is, any planar induced subgraph $H$ that is separated from the rest of the graph $G$ by a separator on its outer face boundary of size $s$, $H$ contains $O(\log s)$ disjoint cycles that are contractible and nested in some embedding of $G$. This allows us to shift from a genus-dependent bound to a separator-dependent bound. 

At a high level, the proof of this result builds on the argument developed in \cite{HK2026}, showing that $G$ contains $O(\log g)$ pairwise disjoint cycles that are contractible and nested in some embedding of $G$. However, the conceptual shift underlying our approach is essential to establishing our main result: it allows us to overcome a barrier that has so far prevented previous approaches from obtaining a polynomial bound. Indeed, by identifying a contractible subgraph of $G$ that admits a separator of size asymptotically smaller than $g$, we can focus on a region of the graph in which the number of disjoint contractible nested cycles is smaller than the $O(\log g)$ bound that applies to the entire graph and that is believed to be asymptotically tight, as shown in \cite{HK2026}. In the rest of our proof, we isolate a contractible subgraph of $G$ with a separator of size $O(\log^2 g)$, which therefore contains $O(\log \log g)$ pairwise disjoint contractible nested cycles. Consequently, we can ``forbid'' $\Omega(\log \log g)$ nested cycles, rather than $\Omega(\log g)$ nested cycles, which allows us to obtain the polynomial bound in terms of the genus $g$ (we can bound the number of vertices of such a contractible subgraph, which is one of the key ideas of the proof). 
This constitutes the paper's main new ingredient.

In the remainder of the proof, we build on the techniques developed in \cite{HK2026} and use the celebrated balanced separator theorem from Robertson and Seymour \cite{GM2}, to obtain our polynomial bound.

\paragraph{Algorithmic implications.}

Let us highlight the algorithmic implications of our results.
As a consequence of the graph minor theorem, Robertson and Seymour \cite{GM13} showed in 1995 that, for every minor-closed class of graphs $\mathcal{G}$, there exists a cubic-time algorithm deciding membership in $\mathcal{G}$, later improved to quadratic time by Kawarabayashi, Kobayashi, and Reed \cite{KKR} in 2012, and to quasi-linear time by Korhonen, Pilipczuk, and Stamoulis \cite{KPS} in 2024. However, this algorithm is only explicit in the sense that its existence and running time are guaranteed: actually instantiating it for a specific class $\mathcal{G}$ requires holding, in hand, the finite set of excluded minors that characterizes $\mathcal{G}$, a set which is, in general, unknown. Determining the excluded minors of a given graph class is therefore not a side question, but the missing ingredient that turns the abstract existence of a fast algorithm into an actual one.

This is made precise by a result of Fellows and Langston \cite{Fellows_Langston}, who showed in 1989 that no algorithm can, given a Turing machine implementing a membership test for a minor-closed class $\mathcal{G}$, compute the excluded minors of $\mathcal{G}$: the task is undecidable in general. Fellows and Langston showed, however, that this obstacle disappears as soon as an explicit upper bound on the order of the excluded minors of $\mathcal{G}$ is known, in which case the excluded minors can be computed by testing minor-minimality among the finitely many candidate graphs of at most that order, see also Adler, Grohe, and Kreutzer \cite{AGK} for an algorithm along the same lines. Crucially, the running time of this procedure grows exponentially in the bound on the order of the excluded minors: for surfaces, the only previously known bound was Seymour's bound of $2^{2^{(3g+9)^9}}$ vertices established in 1993 \cite{seymour}, later improved to a quasi-polynomial bound by Houdaigoui and Kawarabayashi \cite{HK2026}. 

Our \cref{main}, which shows an almost-quadratic bound of $g^{2+o(1)}$ on the order of the excluded minors for a surface of Euler genus $g$, therefore turns an astronomically large, if not entirely impractical, computation into an asymptotically dramatically more tractable one, although the resulting computation remains enormous in absolute terms. It also yields correspondingly better explicit constants for the uniform membership-testing algorithms of \cite{GM13,KKR,KPS} once specialized to embeddability in a fixed surface.


\vspace{1em}

We also wish to emphasize the algorithmic significance of obtaining a small upper bound on the order of excluded minors for surfaces. Bounds of this type have repeatedly served as a fundamental ingredient in parameterized algorithms on graphs embedded in surfaces, where they enable algorithms whose dependence on the input size is linear or quadratic while isolating the complexity in the surface parameter.

For instance, Kawarabayashi, Mohar, and Reed \cite{KMR} presented at FOCS 2008 a linear-time FPT algorithm for embedding graphs into an arbitrary surface, parameterized by its genus, whose correctness and parameter dependence rely on an upper bound on the size of the excluded minors for the surface. Likewise, Grohe, Kawarabayashi, and Reed \cite{GKR} used such a bound in their quadratic-time algorithm for computing graph minor decompositions.

Although replacing Seymour's double-exponential bound by an almost quadratic one does not immediately improve the asymptotic running times of these algorithms, it significantly strengthens one of their key structural ingredients. More broadly, recent progress in graph minor algorithms has repeatedly shown that sharper structural bounds eventually translate into substantially more efficient algorithms. We therefore expect our result to play a similar role in the development of algorithms for bounded-genus graphs.

Moreover, we hope that the techniques and results developed in this paper will prove useful to computational approaches, such as those of \cite{Neufeld1993, NM1997, Chambers2002, MW2018}, aimed at finding the excluded minors of small-genus surfaces such as the torus and the Klein bottle. In particular, they may help reduce the search space involved in such approaches, thereby enabling more efficient algorithms and, ultimately, a more complete understanding of the excluded minors for these two surfaces.

\vspace{1em}

Beyond graphs on surfaces, a major line of research in graph minor theory has been to turn the qualitative structural results of Robertson and Seymour into quantitative ones with polynomial bounds. Such improvements are driven by algorithmic applications, since polynomial dependencies are essential for obtaining efficient parameterized algorithms and practical structural decompositions.

This line of research has seen remarkable progress in recent years. Chekuri and Chuzhoy \cite{Chekuri_chuzhoy} established the first polynomial excluded grid theorem. More recently, building on \cite{quickly}, Gorsky, Seweryn, and Wiederrecht \cite{polynomialGM} proved that all quantitative bounds in the graph minor structure theorem can be made polynomial in the size of the excluded minor.

Our result places excluded minors for surfaces within this broader program. By replacing Seymour's double-exponential bound with an almost quadratic one, we provide the first polynomial bound for this fundamental parameter and thereby strengthen one of the remaining structural ingredients underlying algorithms for bounded-genus graphs.

\section{Overview of the proof and the paper} \label{proof_strategy}

Let $S, S'$ be surfaces with $S'$ of Euler genus $g$ and $S$ of Euler genus $g+1$ or $g+2$. Let $G$ be an excluded minor for the surface $S'$ and suppose that $G$ can be embedded in the surface $S$ with embedding $\Pi$. We define a \textit{piece} as a vertex or a face of $(G, \Pi)$.

To prove that $G$ is of order bounded by an almost quadratic function of $g$, we proceed in four steps. The main idea is to repeatedly use balanced separator decompositions to isolate a contractible region of $G$ in which our separator-sensitive bound on nested contractible cycles can be applied.

\begin{itemize}
\item \textbf{We first isolate a large subgraph $G_0$ whose bridges with respect to a small separator are all $\Pi$-contractible (\cref{sec:main_proof}, \cref{almost_main}).}

We start with a tree decomposition of $G$ of width $O(g\log g)$, as given by \cite{HK2026}. We apply a recursive balanced separator decomposition to this tree decomposition. The decomposition produces $O(g\log^2 g)$ parts of comparable size, while the intersection of any one part with the other parts has size  $O(g\log^2 g)$. Each part corresponds to a subgraph of $G$ and we consider the components of these subgraphs after deleting their boundaries.

The key point is that the components of one of these parts must all be $\Pi$-contractible. Indeed, suppose that every part contained a $\Pi$-noncontractible component. From each such component, we can extract a $\Pi$-noncontractible cycle meeting the boundary in at most one vertex. The resulting cycles are almost disjoint, and their number is $O(g\log^2 g)$. This contradicts the bound of \cref{nb_non_contractible_cycles} on the number of almost disjoint $\Pi$-noncontractible cycles. Hence, there is a subgraph $G_0$ separated from the rest of $G$ by a set $A$ of size $O(g\log^2 g)$, such that every bridge on $A$ contained in $G_0$ is $\Pi$-contractible.

This reduction is important for two reasons. First, since every bridge of $G_0$ is contractible, each such bridge can be treated as a planar graph inside a disk of the embedding $\Pi$. Second, the balanced decomposition guarantees that $G_0$ contains a large fraction of the vertices of $G$, up to a factor $Q(g)=O(g\log^2 g)$. Thus, it is enough to obtain a sufficiently strong bound on the size of $G_0$.

\item \textbf{We then reduce each bridge on $A$ to a $2$-connected contractible subgraph with a logarithmic-size separator (\cref{sec:main_proof}, \cref{main_planar}).}

Consider a bridge $B$ on $A$ in $G_0$, and let $A_B \subseteq A$ be a separator of $B$ of size between $a$ and $2a-1$, for some $a\geq 3$. Since $B$ is $\Pi$-contractible, the global bound of \cref{good_square_cor_general} on nested contractible cycles implies that $B$ cannot contain a large grid minor: a sufficiently large grid minor would contain too many disjoint nested cycles. By the planar grid--treewidth relation (see {\cite[(6.2)]{RST1994}}), this gives $tw(B)=O(\log g)$.

We can therefore find a balanced separator decomposition of $B$ with $O(g \log^2 g)$ parts. This decomposition, together with the fact that $B$ is attached to the rest of $G$ through $A_B$, allows us to isolate a $2$-connected subgraph $G_1$ of $B$ whose separator from the rest of $B$ has size $O(\log g)$. More precisely, $G_1$ can be chosen to be one of the $2$-connected components obtained after cutting the decomposition at an appropriate balanced separator. The remaining pieces of $B$ can then be controlled in terms of $g$, $a$, and $|V(G_1)|$.

The crucial gain is that the separator of $G_1$ is logarithmic in $g$, rather than of order $g$. This is what allows us to apply the separator-sensitive nested-cycle result of \cref{good_square_cor_attachments} with $s=O(\log g)$.

\item \textbf{We prove that such a subgraph $G_1$ has sub-polynomial order in $g$ (\cref{sec:bound_on_planar_subgraph_logarithmic_separator}, \cref{order_planar_subgraph_logarithmic_separator}).}

The separator-sensitive result of \cref{good_square_cor_attachments} implies that $G_1$ contains only $O(\log\log g)$ pairwise disjoint $\Pi$-contractible cycles that are nested in $\Pi$. This is the decisive improvement over the $O(\log g)$ bound that applies to the whole graph.

We use this stronger restriction to control the local structure of $G_1$. First, if a vertex of $G_1$ had very large degree, or if a face of $(G_1,\Pi)$ had very large size, we could find a large $\Pi$-contractible fan based at that piece. The fan can then be iteratively extended to produce many $\Pi$-contractible nested cycles, contradicting the $O(\log\log g)$ bound. Consequently, both the maximum degree of $G_1$ and the maximum size of a face of $(G_1,\Pi)$ are bounded by a sub-polynomial function of $g$ (see \cref{subsec:max_degree}).

We then use the planar structure of $G_1$ to propagate these local bounds through successive layers of nested cycles (see \cref{subsec:bound_on_planar_subgraph_logarithmic_separator}). Starting from an outer cycle, the vertices and bridges between consecutive layers grow by at most a sub-polynomial factor at each step. Since the number of nested layers that can occur in the relevant configuration is controlled by the $O(\log\log g)$ bound above, this yields a sub-polynomial bound $|V(G_1)|=g^{o(1)}$ (more precisely, $|V(G_1)|$ is bounded by the function $P(g)$ defined in \cref{sec:bound_on_planar_subgraph_logarithmic_separator}).

Thus, the large-scale structure of each bridge is reduced to a small $2$-connected core $G_1$, while the dependence on the size $a$ of its attachment to $A$ is kept explicit.

\item \textbf{We finally combine the bounds according to the size of the separators of the bridges (\cref{sec:main_proof}, \cref{main_2_connected}).}

We partition the bridges on $A$ according to the size of their attachment separator. For
$a_i=3\cdot 2^i$, the $i$-th class consists of bridges whose separator has size between $a_i$ and $2a_i-1$. There are only $O(\log g)$ such classes, since the size of $A$ is $O(g\log^2 g)$.

For a fixed class, the number of bridges is bounded using the fact that each bridge has at least $a_i$ edges incident with its separator, while the total number of edges in the contractible part is controlled by its number of vertices and the genus of the ambient embedding. This gives a bound on the number of bridges in the class of order $O\left(\frac{g\log^2 g}{a_i}\right)$.

On the other hand, \cref{main_planar} bounds the size of each bridge by a function that is essentially linear in $a_i$ and sub-polynomial in $g$. The factor $a_i$ therefore cancels with the inverse dependence in the number of bridges in each class. Summing over the $O(\log g)$ classes gives a bound on $|V(G_0)|$ that is almost linear in $g$.

Finally, \cref{almost_main} gives $|V(G)|\leq Q(g)|V(G_0)|$
where $Q(g)$ is nearly linear in $g$, and combining the two bounds yields
\[|V(G)|\leq g^2(\log g)^{O((\log\log g)^3)} =g^{2+o(1)}\]
The argument above is first carried out for $2$-connected excluded minors. The general case then follows from the decomposition into $2$-connected components given by \cref{G_2_connected}.
\end{itemize}

In this paper, we prove three main structural results that describe forbidden structures in $(G,\Pi)$ and that will be used repeatedly to find a contradiction in the main proof:

\begin{itemize}
    \item There exist $O(\log g)$ disjoint cycles that are $\Pi$-contractible and nested in $\Pi$ (\cref{subsec:contractible_nested_cycles}, \cref{good_square_cor_general}).
    \item There exist $O(\log g)$ disjoint cycles that are $\Pi$-noncontractible homotopic (\cref{subsec:homotopic_nested_cycles}, \cref{good_square_variant}).
    \item In a subgraph $H$ of $G$ with a separator of size $s$, there are $O(\log s)$ disjoint cycles that are $\Pi$-contractible and nested in $\Pi$ (\cref{subsec:contractible_nested_cycles}, \cref{good_square_cor_attachments}).
\end{itemize}

These three results improve on a result of Seymour {\cite[(2.2)]{seymour}} and, more recently, on results in \cite[Section 5]{HK2026}. They are central to obtaining the main result of this paper.

\vspace{1em}

This paper is organized as follows. \cref{sec:definitions} contains the basic definitions and notation for graphs and surfaces, which are used throughout the paper. In \cref{sec:preliminary_results}, we introduce basic results for graphs on surfaces and preliminary results on the excluded minors for a surface. In \cref{sec:nested_cycles}, we prove the three main structural results mentioned above. In \cref{sec:bound_on_planar_subgraph_logarithmic_separator}, we show that a $\Pi$-contractible $2$-connected subgraph of $G$ with a separator of size logarithmic in $g$ is of order sub-polynomial in $g$.
In \cref{sec:main_proof}, we detail the proof of our main result. \cref{sec:conclusion} presents the remaining open questions regarding the order of excluded minors for surfaces and our perspective on how they may be approached.

\newpage

\section{Definitions and Notation} \label{sec:definitions}

The definitions and notation presented here are the same as in \cite{HK2026}.

\subsubsection*{Basics}

We consider simple graphs. Let $G$ be a graph.
A \textit{walk} in $G$ is a finite sequence of vertices in which each pair of consecutive vertices is adjacent. A \textit{path} is a walk whose vertices are distinct, except possibly its two endpoints. The \textit{interior} of a path $P$ is the set of vertices of $P$ obtained after removing its two endpoints. A \textit{circuit} (or \textit{closed walk}) is a walk whose first and last vertices coincide. A \textit{cycle} is a circuit whose vertices are all distinct except for its two endpoints.

The \textit{degree} of a vertex $v$ of $G$ is the number of edges adjacent to $v$; we denote it by $d_G(v)$ (or $d(v)$ if clear in the context). The \textit{maximum degree} of $G$ is the maximum degree of a vertex of $G$ and is denoted $\Delta(G)$.

A \textit{minor} $H$ of $G$ is a graph obtained from $G$ by successive operations of deletion of a vertex, deletion of an edge, and contraction of an edge. A \textit{proper minor} of $G$ is a minor of $G$ different from $G$.
We denote by $G - v$ (resp. $G-e$) the graph obtained from $G$ after the deletion of a vertex $v$ (resp. an edge $e$), and we denote $G - X$ the graph obtained from $G$ after the deletion of a set $X$ of vertices or edges. We denote by $G / e$ the graph obtained from $G$ after the contraction of an edge $e$. We denote $G / X$ the graph obtained from $G$ after the contraction of a set $X$ of edges (remark that the order in which the edges in $X$ are contracted is not important).

\subsubsection*{Trees and tree decomposition}

A \textit{tree} $T$ is a connected graph without any cycle. We can distinguish a vertex $v$ of $T$, called the \textit{root}. In that case, $T$ is said to be \textit{rooted} in $v$. We define the \textit{height} of a tree as the order of its longest path.
A \textit{spanning tree} $T$ of $G$ is a subgraph of $G$ that is a tree and contains all the vertices of $G$.
A \textit{tree decomposition} of $G$ is a pair $(T, (V_t)_{t \in V(T)})$ with $T$ a tree and, for every $t \in V(T)$, $V_t \subseteq V(G)$ with the following properties:
\begin{itemize}
    \item $\bigcup_{t \in V(T)} V_t = V(G)$,
    \item for every $e = uv \in E(G)$, there exists $t \in V(T)$ so that $u,v \in V_t$,
    \item for $t,t',t'' \in V(T)$ so that $t'$ is on the path between $t$ and $t''$ in $T$, $V_t \cap V_{t''} \subseteq V_{t'}$.
\end{itemize}
The \textit{width} of a tree decomposition $(T, (V_t)_{t \in V(T)})$ of $G$ is $\max_{t \in V(T)} |V_t| -1$ and the \textit{treewidth} of $G$ is the minimal width of its tree decompositions.

We say that a tree decomposition $(T, (V_t)_{t \in V(T)})$ of $G$, possibly subject to additional constraints, is \textit{minimal}, if $|V(T)|$ is chosen minimum among all tree decompositions of $G$ satisfying these constraints.

\subsubsection*{Connectivity}

A graph $G$ is \textit{connected} if, for every pair of vertices $u,v \in V(G)$, there is a path in $G$ between $u$ and $v$. Let $k \in \mathbb{N}^*$, $G$ is \textit{$k$-connected} if, for every set $V \subseteq V(G)$ of size $k-1$, $G - V$ is still connected. Let $k \geq 1$, a \textit{$k$-separator} of $G$ is a set $V \subseteq V(G)$ of $k$ vertices so that $G - V$ is not connected. If $G$ contains a $1$-separator $\{v\}$, $v$ is called a \textit{cutvertex} of $G$. Remark that, for $k \geq 1$, a graph contains no $k-1$-separator if and only if it is $k$-connected.
A \textit{separation} of $G$ is a pair $(A, B)$ of the subgraph of $G$ such that $A \cup B = G$ and $A \cap B$ contains no edge. Remark that $V(A \cap B)$ is a separator of $G$. For $k \in \mathbb{N}$, we say that $(A,B)$ is a \textit{$k$-separation} of $G$ if $(A,B)$ is a separation of $G$ and $|V(A \cap B)| = k$.
A \textit{($2$-connected) block} of $G$ is the subgraph of $G$ induced by an equivalence class of the following equivalence relation on $E(G)$: $e_1 \sim e_2$ if $e_1 = e_2$, or there is in $G$ a cycle that contains $e_1$ and $e_2$. Remark that $H$ is a block of $G$ if and only if $H$ is a maximal $2$-connected subgraph of $G$.

\subsubsection*{Bridges}

Let $H$ be a graph and $H_0$ be a subgraph of $H$. A \textit{bridge} $B$ of $H$ on $H_0$ is either an edge with both ends in $H_0$ (and its ends do not belong to the bridge) or a maximal connected subgraph of $H - V(H_0)$ together with all the edges which have one end in this component and the other end in $H_0$ (the ends in $H_0$ does not belong to the bridge). 
Remark that the bridges of $H$ on $H_0$ partition $E(H) - E(H_0)$.
We say that $w \in H_0$ is an \textit{attachment} of a bridge $B$ on $H_0$ if there exists $v \in V(B)$ such that $e = vw \in E(B)$.

\vspace{0.7cm}

To describe graphs on surfaces, we follow the definitions and notation of Mohar and Thomassen's book \textit{Graphs on Surfaces} \cite{graphs_on_surfaces}, with the sole exception that we use the term genus to denote the Euler genus. Thereafter, we give a condensed summary of the definitions and notation used in this paper.

\subsubsection*{Surface and disk}

We define a \textit{surface} as a connected compact Hausdorff topological space $S$ so that each point of $S$ has an open neighborhood homeomorphic to the open unit disk in $\mathbb{R}^2$. A \textit{simple closed curve} in a surface $S$ is a continuous 1-1 function $f : [0,1] \mapsto S$ with $f(0) = f(1)$. A \textit{disk} on $S$ is a simple closed curve that can be continuously deformed into a single point.

\subsubsection*{Embedding and facial walk}

Let $G$ be a graph and $S$ a surface, we define an \textit{embedding} $\Pi$ of $G$ in $S$ to be a pair $\Pi = (\pi, \lambda)$ where $\pi = \{ \pi_v | v \in V(G) \}$ associates to each vertex $v \in V(G)$ a cyclic permutation of the edges incident to $v$ and $\lambda : E(G) \mapsto \{-1,1\}$ associates to each edge $e \in E(G)$ a sign (either $-1$ or $1$) called its \textit{signature}. 

A \textit{local change} of an embedding $\Pi = (\pi, \lambda)$ of $G$ changes the clockwise ordering to anticlockwise at some vertex $v \in V(G)$, i.e. $\pi_v$ is replaced by its inverse $\pi_v^{-1}$, and $\lambda(e)$ is replaced by $-\lambda(e)$ for every edge $e$ that is incident with $v$. Two embeddings of $G$ are \textit{equivalent} if one can be obtained from the other by a sequence of local changes.

The \textit{face traversal procedure} is the following: We start with a vertex $v \in V(G)$ and an edge $e = vw \in E(G)$. Let's traverse the edge from $v$ to $w$, if the signature of $e$ is $1$, then we continue the walk along the edge $e' = \pi_w(e)$, otherwise ($\lambda(e) = -1$), we change $\pi_u$ by $\pi_u^{-1}$ for every $u \in V(G)$ and we continue the walk along the edge $e' = \pi_w(e)$. And so forth. The walk is completed when the initial edge $e$ is encountered in the same direction (from $v$ to $w$) and with the same orientation ($\pi$ is the same as before we began the walk). A \textit{$\Pi$-facial walk} (or, if clear in the context, \textit{face}) is a closed walk in $G$ determined by the face traversal procedure. We define the \textit{size} of a face to be the number of edges that the walk contains. We define the \textit{maximum face degree} of $(G, \Pi)$ to be $\Delta_F(G, \Pi) = \max \{ |f|, f \text{ is a face in } (G, \Pi) \}$. 

Let $C$ be a cycle of $G$, we define the \textit{signature} of $C$ in $\Pi$ to be $\lambda(C) = \prod_{e \in C} \lambda(e)$. We say that $C$ is \textit{two-sided} if $\lambda(C) = 1$ and \textit{one-sided} otherwise. We say that an embedding $\Pi$ of $G$ is \textit{orientable} if every cycle $C$ of $G$ is two-sided; otherwise, we say that it is \textit{non orientable}.

\subsubsection*{Genus and embeddability}

In this paper, the \textit{genus} of a surface always refers to the Euler genus, that is to say, twice its orientable genus if the surface is orientable and its non orientable genus if the surface is non orientable. The sphere is the only surface of genus $0$, the projective plane is the only surface of genus $1$ and is a non orientable surface.
Let $F(G, \Pi)$ be the $\Pi$-facial walks of $(G, \Pi)$. We define the Euler characteristic of $\Pi$ to be $\chi(\Pi) = |V(G)| - |E(G)| + |F(G, \Pi)|$.
For an embedding $\Pi$ and a graph $G$, $g(\Pi) = 2 - \chi(\Pi)$ denotes the \textit{genus} of $\Pi$, and $g(G)$ denotes the genus of an embedding of $G$ with minimal genus. The formula \[ \chi(\Pi) = 2 - g(\Pi)\] is called the \textit{Euler formula}.
We say that $G$ is \textit{embeddable} in a surface $S$, if there exists an embedding $\Pi$ of $G$ so that $g(\Pi) = g(S)$ and $\Pi$ is \textit{orientable} if and only if $S$ is orientable.
Let $G'$ be a minor of $G$; it is easy to prove that $g(G') \leq g(G)$. We say that $G$ is a \textit{(minimal) excluded minor} for a surface $S$ if $G$ is not embeddable in $S$ but every proper minor of $G$ is embeddable in $S$.

\subsubsection*{Separating cycles}

Let $C = v_0 e_1 v_1 e_2 ... v_{l-1}e_lv_0$ be a $\Pi$-twosided cycle of a $\Pi$-embedded graph $G$. Suppose that the signature of each edge of $C$ is positive in $\Pi$. We define the \textit{left graph} and the \textit{right graph} of $C$ as follows: for $1 \leq i \leq l$, if $e_{i+1} = \pi_{v_i}^{k_i}(e_i)$, then all the edges $\pi_{v_i}(e_i), \pi_{v_i}^2(e_i), ..., \pi_{v_i}^{k_i-1}(e_i)$ are said to be on the left side of $C$ and the left graph of $C$, denoted by $G_l(C, \Pi)$, is defined as the union of all bridges on $C$ that attach to $C$ by at least one edge on the left side of $C$. The right graph $G_r(C, \Pi)$ is defined analogously.
A cycle $C$ of a $\Pi$-embedded graph $G$ is \textit{$\Pi$-separating} if $C$ is two-sided and $G_l(C, \Pi)$ and $G_r(C, \Pi)$ have no edges in common. If $\Pi(G_l(C, \Pi))$ or $\Pi(G_r(C, \Pi))$ is an induced embedding of genus $0$, then we say that $C$ is \textit{$\Pi$-contractible}. Suppose without loss of generality that $\Pi(G_l(C, \Pi))$ is an induced embedding of genus $0$, we define $\text{int}(C, \Pi_H)$ (resp. $\text{ext}(C, \Pi_H)$) to be the bridges onto $C$ in $G_l(C, \Pi)$ (resp. $G_r(C, \Pi)$). Moreover, we define $\text{Int}(C, \Pi_H) = \text{int}(C, \Pi_H) \cup C$ and $\text{Ext}(C, \Pi_H) = \text{ext}(C, \Pi_H) \cup C$. If it is clear in the context, the mention of the embedding is removed, and we write $\text{Int}(C)$, $\text{int}(C)$, $\text{Ext}(C)$, $\text{ext}(C)$.

\subsubsection*{Planar graphs}

A \textit{planar graph} is a graph embeddable in the sphere (surface of genus $0$). Let $G$ be a planar graph. A \textit{planar embedding} $\Pi$ of $G$ is an embedding of $G$ in the sphere with a distinguished face called the \textit{outer face}. Let $C$ be a cycle of $G$, then $C$ is two-sided and $\Pi$-separating. Let $G_l(C, \Pi)$ and $G_r(C, \Pi)$ be the left and right graphs of $G$ and suppose without loss of generality that the outer face of $\Pi$ is in $G_r(C, \Pi)$. Then, we define $\text{Int}(C, \Pi)$, $\text{int}(C, \Pi)$, $\text{Ext}(C, \Pi)$, $\text{ext}(C, \Pi)$ as above, by making sure that the outer face in the exterior of $C$.

\subsubsection*{Cutting along a cycle}

Let $G$ be a $\Pi$-embedded graph. Let $C$ be a $\Pi$-separating cycle, then cutting along $C$ gives rise to two graphs $G_l(C, \Pi) \cup C$ and $G_r(C, \Pi) \cup C$ and their induced embedding $\Pi(G_l(C, \Pi))$ and $\Pi(G_r(C,\Pi))$. Let $C$ be a two-sided $\Pi$-nonseparating cycle. Let $\overline{G}$ be the graph obtained from $G$ by replacing $C$ with two copies of $C$ such that all the edges of $C$ on the left side of $C$ are incident on one copy of $C$ and all the edges on the right side of $C$ are incident with the other copy of $C$. We say that $\overline{G}$ is the graph obtained by \textit{cutting along $C$}, and we call the induced embedding $\overline{\Pi}$. Let $C= v_0 e_1 v_1 e_2 ... v_{l-1}e_lv_0$ be a one-sided $\Pi$-nonseparating cycle. Let's first define edges on the left side of $G$ and edges on the right side of $G$ at each vertex of $C$ in a similar way as in the case of two-sided cycles: suppose first that the signature $\lambda$ of $\Pi$ satisfies $\lambda(e_i) = 1$ for $1 \leq i < l$ and $\lambda(e_l) = -1$. Then, we use the pairs of consecutive edges $e_i, e_{i+1}$ on $C$ ($1 \leq i < l$)  to define edges on the left side of $C$ incident with the vertex $v_i$. Then we construct $\overline{G}$ by replacing $C$ in $G$ by the cycle $\overline{C} = v_0 e_1 ... e_l \overline{v}_0 \overline{e}_1 ... \overline{e}_l v_0$. The edges on the left side of $C$ are adjacent to $v_0, ..., v_{l-1}$ and the edges on the right side of $C$ are incident to $\overline{v}_0, ..., \overline{v}_{l-1}$. We extend $\lambda$ by putting $\lambda(\overline{e}_0) = ... = \lambda(\overline{e}_{l-1}) = 1$ and $\lambda(\overline{e}_l) = -1$ and hence obtain an embedding $\overline{\Pi}$ of $\overline{G}$.
We say that $\overline{G}$ is the graph obtained by \textit{cutting along $C$}. 

\subsubsection*{Homotopic cycles}

Let $C$ and $C'$ be two-sided cycles of a $\Pi$-embedded graph $G$. Suppose that $C$ and $C'$ are either disjoint or share a path (that might be reduced to a vertex). We say that $C$ and $C'$ are \textit{$\Pi$-homotopic} if cutting along $C$ and $C'$ results in a graph which has a component $D$ which contains precisely one copy of $C$ and one copy of $C'$ and the induced genus of the embedding $\Pi(D)$ is $0$. We then write $D = \text{Int}(C \cup C', \Pi)$. 

\subsubsection*{Cycles bounding a disk or cylinder}

Let $G$ be a graph $\Pi$-embedded in a surface $S$. Let $C$ be a $\Pi$-contractible cycle in $G$. Then, we say that $C$ is the \textit{boundary} of $\text{Int}(C, \Pi)$. Moreover, we say that $C$ \textit{bounds a disk} if $\text{int}(C, \Pi)$ is empty. Let $C$ and $C'$ be two cycles that are disjoint and $\Pi$-homotopic. We say that $C$ and $C'$ are the \textit{boundaries} of $\text{Int}(C \cup C', \Pi)$. Moreover, we say that $C$ and $C'$ \textit{bound an cylinder} if $\text{int}(C \cup C', \Pi)$ is empty.

\section{Preliminary results}\label{sec:preliminary_results}

We begin by presenting basic results on graphs on surfaces and several simple structural results on excluded minors for surfaces. Almost all the results presented here come either from the book \textit{Graphs on Surfaces} of Mohar and Thomassen \cite{graphs_on_surfaces} or from \cite{HK2026}. We cite all results here for completeness, but we do not recall their proofs, which are available in \cite{graphs_on_surfaces, HK2026}.

\vspace{1em}

Let's define $G, S, S'$ for the rest of the paper as follows: Let $S'$ be a surface of genus $g$ and $G$ be a minimal excluded minor for $S'$. Remark that $G$ has genus either $g+1$ or $g+2$: for any edge $e \in E(G)$, $G-e$ is embeddable in $S'$, and it is easy to show that any extension of an embedding $\Pi_e$ of $G-e$ in $S'$ to an embedding of $G$ has genus at most $g+2$. Therefore $g < g(G) \leq g+2$. Finally, let $S$ be a surface of genus $g(G)$ in which $G$ can be embedded with embedding $\Pi$.

\subsection{Basic results on graphs on surfaces}

First, we recall classical results on graphs on surfaces that will be useful later.

\begin{defi}[Flipping]
    Let $H$ be a 2-connected planar graph with embedding $\Pi_H$ in the plane. Let $C$ be a cycle of $H$ such that only two vertices $v$ and $w$ of $C$ have incident edges in $\text{ext}(C, \Pi_H)$. Then, we define a \textit{flipping} of $H$ with respect to $C$ as a reembedding of $H$ such that the embedding in $\text{ext}(C, \Pi_H)$ is unchanged. The embedding of $H' = H \cap \text{Int}(C, \Pi_H)$ is changed so that the new embedding of $H'$ is equivalent to the original one, but the clockwise orientations of all the facial cycles are reversed.
\end{defi}

\begin{prop}[Whitney's Theorem \cite{Whitney}]

    \label{Whitney_planar_graph_flipping}
    Let $H$ be a 2-connected plane graph and $\Pi_H$ be an embedding of $H$ in the plane. Then, any embedding of $H$ in the plane can be obtained from $\Pi_H$ by a sequence of flippings.
\end{prop} 

\begin{prop}[{\cite[Proposition 4.2.7]{graphs_on_surfaces}}]
    \label{homotopic_cycles}
    Let $H$ be a $\Pi_H$-embedded graph and $a$, $b$ vertices of $H$ (possibly $a = b$). If $P_0, ..., P_k$ are pairwise internally disjoint paths (or cycles) from $a$ to $b$ such that no two of them are $\Pi$-homotopic, then

    \[ k \leq
    \left\{
    \begin{array}{lr}
        g(\Pi_H) &  \text{ if } g(\Pi_H) \leq 1\\
        3 g(\Pi_H) - 3 & \text{ if } g(\Pi_H) \geq 2 \\ 
    \end{array}\right.\]
\end{prop}

\vspace{2em}
We then present two variants of \cref{homotopic_cycles} that extend its scope of application in two different directions.

\begin{defi}[Almost disjoint cycles]
    Let $k \in \mathbb{N}$. Let $C_0,..., C_k$ be cycles of a graph $H$. We say that $C_0, ..., C_k$ are \textit{almost disjoint cycles} if each cycle $C_i$ ($0 \leq i \leq k$) has at most one vertex in common with the subgraph $\cup_{j \neq i} C_j$.

    See \cref{fig:almost_disjoint_cycles} for an illustration of almost disjoint cycles.
\end{defi}

\begin{figure}[h!]
    \centering
    \begin{tikzpicture}[scale=1.5]
	\begin{pgfonlayer}{nodelayer}
		\node [style=basic] (0) at (-4, 3) {};
		\node [style=basic] (1) at (0, 1) {};
		\node [style=basic] (2) at (4, 2) {};
		\node [style=none] (3) at (-6.5, 4.25) {};
		\node [style=none] (4) at (-5, 0.75) {};
		\node [style=none] (5) at (-2.5, 1.5) {};
		\node [style=none] (6) at (0, 3) {};
		\node [style=none] (7) at (0, -1) {};
		\node [style=none] (8) at (2.5, 0.5) {};
		\node [style=none] (9) at (4, 4) {};
		\node [style=none] (10) at (6, 1) {};
	\end{pgfonlayer}
	\begin{pgfonlayer}{edgelayer}
		\draw [in=105, out=60, looseness=1.50] (3.center) to (0);
		\draw [in=150, out=-135, looseness=1.25] (0) to (4.center);
		\draw [in=-150, out=-120, looseness=1.50] (3.center) to (0);
		\draw [in=-15, out=-90] (0) to (4.center);
		\draw [in=135, out=75, looseness=1.25] (5.center) to (1);
		\draw [in=0, out=45] (1) to (6.center);
		\draw [in=135, out=-180, looseness=1.25] (6.center) to (1);
		\draw [in=-150, out=-90] (5.center) to (1);
		\draw [in=180, out=-135, looseness=1.50] (1) to (7.center);
		\draw [in=0, out=-45, looseness=1.50] (1) to (7.center);
		\draw [in=135, out=180, looseness=1.75] (9.center) to (2);
		\draw [in=45, out=0, looseness=1.50] (9.center) to (2);
		\draw [in=120, out=-165, looseness=1.25] (2) to (8.center);
		\draw [in=-30, out=-90, looseness=1.25] (2) to (8.center);
		\draw [in=60, out=15, looseness=1.25] (2) to (10.center);
		\draw [in=-135, out=-60, looseness=1.50] (2) to (10.center);
	\end{pgfonlayer}
\end{tikzpicture}
    \caption{Almost disjoint cycles. The almost disjoint cycles are depicted in solid black lines. The black dots are vertices shared by several almost disjoint cycles.}
    \label{fig:almost_disjoint_cycles}
\end{figure}
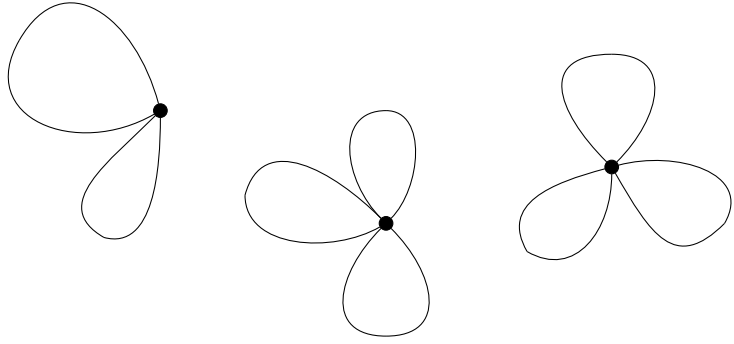

\begin{prop}[Variant of \cref{homotopic_cycles}, {\cite[Proposition 4.8]{HK2026}}]
    \label{homotopic_cycles_variant1}
    Let $H$ be a $\Pi_H$-embedded connected graph. If $C_1, ..., C_k$ are cycles of $H$ that are almost disjoint, $\Pi_H$-noncontractible and such that no two of them are $\Pi_H$-homotopic, then

    \[ k \leq
    \left\{
    \begin{array}{lr}
        g(\Pi_H) &  \text{ if } g(\Pi_H) \leq 1\\
        3 g(\Pi_H) - 3 & \text{ if } g(\Pi_H) \geq 2 \\ 
    \end{array}\right.\]
\end{prop}

\begin{defi}[Cycles on a spanning tree] \label{def:free_cycles}
     Let $H$ be a graph and $a$ a vertex of $H$. Let $k \geq 1$ and let $C_0, ..., C_k$ be subgraphs of $H$.
     We say that $C_0, ..., C_k$ are \textit{cycles on a spanning tree} rooted in $a$ if:
     \begin{itemize}
         \item there exists a spanning tree $T$ of $C_0 \cup ... \cup C_k$ rooted in $a$ such that $(C_0 \cup ... \cup C_k) - T$ consists of $k+1$ edges $\{e_0, ..., e_k\}$ and, for $0 \leq i \leq k$, $e_i$ belongs solely to $C_i$;
         \item let $C'_i$ be the unique cycle induced by $T$ and $e_i$, then $C_i$ is the subgraph consisting of $C'_i$ together with the unique path in $T$ from $a$ to $C'_i$ (if $C'_i$ contains $a$ then $C_i = C'_i$).
     \end{itemize}

     \cref{fig:free_cycles} shows an example of cycles on a spanning tree.
\end{defi}

\begin{figure}[h!]
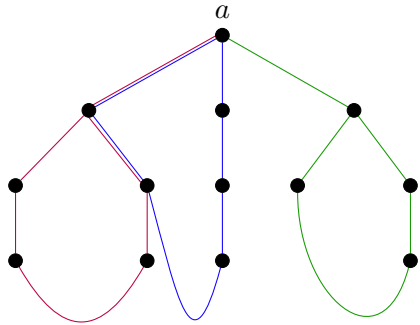

    \centering
    \ctikzfig{images/free_cycles}
    \caption{Cycles on a spanning tree rooted in $a$. The three colored subgraphs are cycles on a spanning tree rooted in $a$.}
    \label{fig:free_cycles}
\end{figure}

\begin{prop}[Variant of \cref{homotopic_cycles}, {\cite[Proposition 4.10]{HK2026}}]
    \label{homotopic_cycles_variant2}
    Let $H$ be a $\Pi_H$-embedded graph and $a$ a vertex of $H$. Let $C_0, ..., C_k$ be cycles on a spanning tree rooted in $a$ that are $\Pi_H$-noncontractible and pairwise $\Pi_H$-nonhomotopic.
    Then,
    \[ k \leq
    \left\{
    \begin{array}{lr}
        g(\Pi_H) &  \text{ if } g(\Pi_H) \leq 1\\
        3 g(\Pi_H) - 3 & \text{ if } g(\Pi_H) \geq 2 \\ 
    \end{array}\right.\]
\end{prop}

\subsection{Structure of \texorpdfstring{$G-e$}{G-e}}

For $e \in E(G)$, we define $G_e = G - e$. As $G_e$ is a proper minor of $G$, it can be embedded in $S'$. We define $\Pi_e$ as an embedding of $G_e$ in $S'$. Let's prove the following result, which compares the embeddings $(G,\Pi)$ and $(G_e, \Pi_e)$.

\begin{lem}[{\cite[Lemma 4.11]{HK2026}}]
    \label{2_connected_faces_cycles}
    Let $H$ be a $2$-connected $\Pi_H$-embedded graph. Let $C$ be a $\Pi_H$-contractible cycle. Then, every face in $\text{Int}(C, \Pi_H)$ is a cycle and, therefore, every edge in $\text{int}(C, \Pi_H)$ is contained in two distinct faces.
\end{lem}

\begin{lem}[{\cite[Lemma 4.17]{HK2026}}]
    \label{C_e_non_contractible}
    Let $C$ be a $\Pi$-contractible cycle and let $e$ be an edge with both endpoints in $\text{int}(C,\Pi)$. Let $f,f'$ be the faces in $\Pi$ that contain $e$ and let $C_e = f \cup f' -e$. Then, $C_e$ is a contractible cycle in $(G,\Pi)$ such that $\text{Int}(C_e,\Pi) = C_e + e$. 
    Moreover, $C_e$ is $\Pi_e$-nonseparating (and thus $\Pi_e$-noncontractible).
\end{lem}

\begin{figure}[h!]
    \centering
    \begin{tikzpicture}[scale=1.2]
	\begin{pgfonlayer}{nodelayer}
		\node [style=none] (0) at (0, 6) {};
		\node [style=none] (1) at (-4, 4.5) {};
		\node [style=none] (2) at (-7.25, 2.75) {};
		\node [style=none] (3) at (-7, 0) {};
		\node [style=none] (4) at (-5, -4) {};
		\node [style=none] (5) at (0, -6) {};
		\node [style=none] (6) at (5, -4) {};
		\node [style=none] (7) at (7, -2) {};
		\node [style=none] (8) at (8, 0) {};
		\node [style=none] (9) at (6, 3.25) {};
		\node [style=none] (10) at (4, 4.75) {};
		\node [style=none] (11) at (-6.25, -1.75) {};
		\node [style=none] (12) at (-4.75, -1.75) {};
		\node [style=none] (13) at (-4.25, 2) {};
		\node [style=none] (14) at (-1.5, 0) {};
		\node [style=none] (15) at (-2.5, -1.75) {};
		\node [style=none] (16) at (-2.25, 1.5) {};
		\node [style=none] (17) at (-5, 2) {};
		\node [style=none] (18) at (-3.5, 1.75) {};
		\node [style=none] (19) at (-4, 3) {};
		\node [style=none] (20) at (-5.5, -1.75) {};
		\node [style=none] (21) at (-3.5, -1.75) {};
		\node [style=none] (22) at (-4.05, -1.75) {};
		\node [style=none] (23) at (-5, -2.5) {};
		\node [style=none] (24) at (-3.75, -2.5) {};
		\node [style=none] (25) at (-3.5, -3) {};
		\node [style=none] (26) at (-0.25, -2.25) {};
		\node [style=none] (27) at (-1, -3.75) {};
		\node [style=none] (28) at (5.25, -1) {};
		\node [style=none] (29) at (1, 0.25) {};
		\node [style=none] (30) at (0.5, 2) {};
		\node [style=none] (31) at (-1, 3.75) {};
		\node [style=none] (32) at (1.75, 3) {};
		\node [style=none] (33) at (3, 0.5) {};
		\node [style=none] (34) at (2, -2.25) {};
		\node [style=none] (35) at (4.75, 0.5) {};
		\node [style=none] (36) at (1.5, -3.75) {};
		\node [style=none] (37) at (-1, -3.75) {};
		\node [style=none] (38) at (0.25, 0.75) {};
		\node [style=none] (39) at (-1, -7) {};
		\node [style=none] (40) at (1, -7) {};
		\node [style=none] (41) at (5.25, -5.25) {};
		\node [style=none] (42) at (5.75, -5) {};
		\node [style=none] (43) at (6.25, -4.5) {};
		\node [style=none] (44) at (7.75, -2.75) {};
		\node [style=none] (45) at (9.25, -0.5) {};
		\node [style=none] (46) at (9.25, 0.5) {};
		\node [style=none] (47) at (7, 3.75) {};
		\node [style=none] (48) at (4.75, 5.75) {};
		\node [style=none] (49) at (-1.25, 6.75) {};
		\node [style=none] (50) at (-0.5, 7) {};
		\node [style=none] (51) at (0.75, 7) {};
		\node [style=none] (52) at (1.25, 6.75) {};
		\node [style=none] (53) at (-4.75, 5.5) {};
		\node [style=none] (54) at (-8.25, 2.5) {};
		\node [style=none] (55) at (-8.25, 3.25) {};
		\node [style=none] (56) at (-8.5, 0) {};
		\node [style=none] (57) at (-8.25, -0.75) {};
		\node [style=none] (58) at (-7.5, -2) {};
		\node [style=none] (59) at (-6.25, -4.5) {};
		\node [style=none] (60) at (-5.95, 1.725) {};
		\node [style=none] (61) at (-6.2, 0) {};
		\node [style=none] (62) at (-6.2, -1.75) {};
		\node [style=none, label={$e$}] (63) at (-4.15, -0.5) {};
		\node [style=none, label={$C_e$}] (64) at (-2.45, -1.2) {};
		\node [style=none, label={$C$}] (65) at (-2.25, 5.2) {};
		\node [style=none] (66) at (-4.5, 3.25) {};
		\node [style=none] (67) at (-3.5, 3) {};
		\node [style=none] (68) at (-3, -3.25) {};
		\node [style=none] (69) at (-4.5, -3.25) {};
	\end{pgfonlayer}
	\begin{pgfonlayer}{edgelayer}
		\draw [style=new edge style 14] (68.center)
			 to [bend left, looseness=0.75] (69.center)
			 to [in=-90, out=-180, looseness=0.75] (20.center)
			 to [bend right] (22.center)
			 to [bend right=45, looseness=2.50] (21.center)
			 to [in=0, out=-45, looseness=0.75] cycle;
		\draw [style=new edge style 14] (67.center)
			 to [in=75, out=-15, looseness=0.50] (18.center)
			 to [bend right=45, looseness=0.75] (13.center)
			 to [bend left=315, looseness=0.75] (17.center)
			 to [in=-165, out=60, looseness=0.50] (66.center)
			 to [bend left, looseness=0.75] cycle;
		\draw [style=new edge style 1] (2.center)
			 to (3.center)
			 to (11.center)
			 to (4.center)
			 to (5.center)
			 to (6.center)
			 to (7.center)
			 to (8.center)
			 to (9.center)
			 to (10.center)
			 to (0.center)
			 to (1.center)
			 to cycle;
		\draw [style=new edge style 2] (16.center) to (14.center);
		\draw [style=new edge style 2] (14.center) to (15.center);
		\draw [style=new edge style 0] (13.center) to (12.center);
		\draw [style=new edge style 2] (17.center) to (13.center);
		\draw [style=new edge style 2] (13.center) to (18.center);
		\draw [style=new edge style 2] (18.center) to (16.center);
		\draw (17.center) to (19.center);
		\draw (13.center) to (19.center);
		\draw (18.center) to (19.center);
		\draw [style=new edge style 2] (11.center) to (20.center);
		\draw [style=new edge style 2] (20.center) to (12.center);
		\draw [style=new edge style 2] (21.center) to (15.center);
		\draw [style=new edge style 2] (12.center) to (22.center);
		\draw [style=new edge style 2] (22.center) to (21.center);
		\draw (22.center) to (24.center);
		\draw (24.center) to (21.center);
		\draw (20.center) to (23.center);
		\draw (23.center) to (22.center);
		\draw (23.center) to (24.center);
		\draw (21.center) to (25.center);
		\draw (24.center) to (25.center);
		\draw (23.center) to (25.center);
		\draw (27.center) to (5.center);
		\draw (4.center) to (37.center);
		\draw (37.center) to (15.center);
		\draw (15.center) to (26.center);
		\draw (26.center) to (36.center);
		\draw (36.center) to (6.center);
		\draw (37.center) to (26.center);
		\draw (26.center) to (34.center);
		\draw (34.center) to (36.center);
		\draw (34.center) to (28.center);
		\draw (28.center) to (6.center);
		\draw (28.center) to (7.center);
		\draw (28.center) to (35.center);
		\draw (28.center) to (8.center);
		\draw (8.center) to (35.center);
		\draw (35.center) to (33.center);
		\draw (33.center) to (34.center);
		\draw (33.center) to (29.center);
		\draw (29.center) to (14.center);
		\draw (29.center) to (38.center);
		\draw (38.center) to (30.center);
		\draw (38.center) to (16.center);
		\draw (30.center) to (16.center);
		\draw (16.center) to (31.center);
		\draw (31.center) to (1.center);
		\draw (31.center) to (30.center);
		\draw (30.center) to (32.center);
		\draw (32.center) to (10.center);
		\draw (32.center) to (0.center);
		\draw (32.center) to (9.center);
		\draw [style=new edge style 3] (4.center) to (59.center);
		\draw [style=new edge style 3] (5.center) to (39.center);
		\draw [style=new edge style 3] (5.center) to (40.center);
		\draw [style=new edge style 3] (6.center) to (41.center);
		\draw [style=new edge style 3] (6.center) to (42.center);
		\draw [style=new edge style 3] (6.center) to (43.center);
		\draw [style=new edge style 3] (7.center) to (44.center);
		\draw [style=new edge style 3] (8.center) to (45.center);
		\draw [style=new edge style 3] (8.center) to (46.center);
		\draw [style=new edge style 3] (9.center) to (47.center);
		\draw [style=new edge style 3] (10.center) to (48.center);
		\draw [style=new edge style 3] (0.center) to (52.center);
		\draw [style=new edge style 3] (0.center) to (51.center);
		\draw [style=new edge style 3] (0.center) to (50.center);
		\draw [style=new edge style 3] (0.center) to (49.center);
		\draw [style=new edge style 3] (1.center) to (53.center);
		\draw [style=new edge style 3] (2.center) to (55.center);
		\draw [style=new edge style 3] (2.center) to (54.center);
		\draw [style=new edge style 3] (3.center) to (56.center);
		\draw [style=new edge style 3] (3.center) to (57.center);
		\draw [style=new edge style 3] (11.center) to (58.center);
		\draw [style=new edge style 2] (60.center) to (61.center);
		\draw [style=new edge style 2] (61.center) to (62.center);
		\draw [style=new edge style 2] (60.center) to (17.center);
		\draw (60.center) to (2.center);
		\draw (61.center) to (2.center);
		\draw (61.center) to (3.center);
		\draw (17.center) to (2.center);
	\end{pgfonlayer}
\end{tikzpicture}
    \caption{The subgraph $\text{Int}(C,\Pi)$ of $G$ is shown embedded in $\Pi$. The cycles $C$ and $C_e$ are represented in blue and green, respectively, and the edge $e$ is represented in red. The bridges of $\mathcal{B} - B_C$ are highlighted in pale blue, while the bridge $B_C$ consists of the remaining part of $G- (C_e \cup e)$.}
    \label{fig:lemma_C_e_non_contractible}
\end{figure}
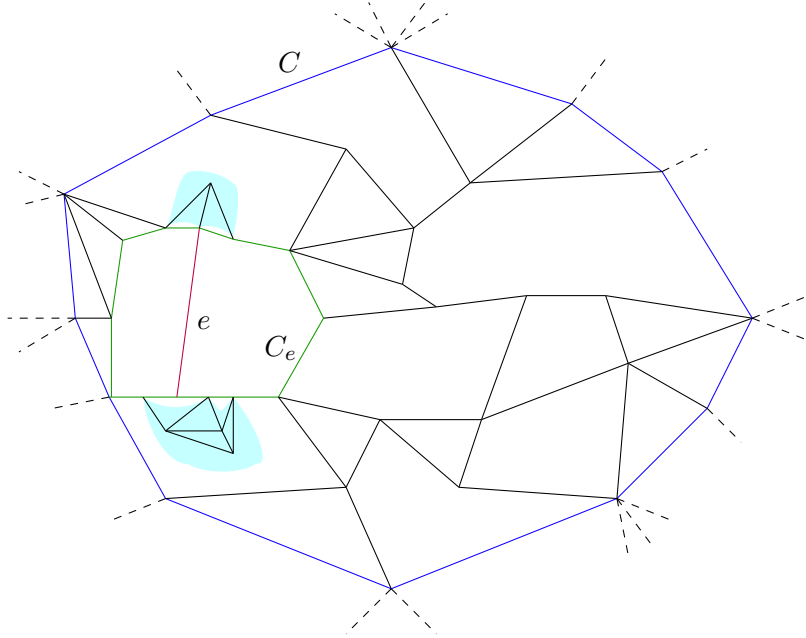

\begin{defi}[Same relative orientation]
    Let $H$ be a graph that is $\Pi_H$-embedded in a surface $S_H$ and $\Pi_H'$-embedded in a surface $S_H'$. Let $C$ and $C'$ be two almost disjoint cycles in $H$ that are $\Pi_H$-contractible, both embedded in some disk or cylinder $S_0$ on $S_H$, and $\Pi_H'$-noncontractible homotopic.

    Let $\overline{\Pi_H'}$ be the embedding in the sphere obtained from $\Pi_H'$ by cutting along $C$ and $C'$. We fix the clockwise orientation in both $S_0$ and $\overline{\Pi_H'}$ to be the orientation of some walk around $C$.

    We say that $C$ and $C'$ \textit{have the same relative orientation} in $\Pi_H$ and $\Pi_H'$ if the clockwise walk of $C'$ in $S_0$ is the same as its clockwise walk in $\overline{\Pi_H'}$.

    \vspace{1em}

    See \cref{fig:same_relative_orientation} for an illustration.
\end{defi}

\begin{figure}[h!]
    \centering
    \begin{subfloat}[Embedding of a subgraph of $H$ in $\Pi_H$]{\begin{tikzpicture}[scale=1.5]
	\begin{pgfonlayer}{nodelayer}
		\node [style=none] (0) at (0, 4) {};
		\node [style=none] (1) at (-4, 0) {};
		\node [style=none] (2) at (0, -4) {};
		\node [style=none] (3) at (4, 0) {};
		\node [style=none] (4) at (-2, 0.5) {};
		\node [style=none] (5) at (-1, 0.75) {};
		\node [style=none] (6) at (-2, -0.75) {};
		\node [style=none] (7) at (-1, -1.25) {};
		\node [style=none] (8) at (-0.5, -0.5) {};
		\node [style=none] (9) at (0.5, 0) {};
		\node [style=none] (10) at (1.25, 0.25) {};
		\node [style=none] (11) at (-0.25, 1.25) {};
		\node [style=none] (12) at (1, 1.5) {};
		\node [style=none] (13) at (2, -0.25) {};
		\node [style=none] (14) at (2.75, -1) {};
		\node [style=none] (15) at (2.25, 1.25) {};
		\node [style=none] (16) at (3.25, 0.5) {};
		\node [style=none] (17) at (-0.25, 1.25) {};
		\node [style=none] (18) at (-2.7, 1) {};
		\node [style=none] (19) at (-2.975, -1.125) {};
		\node [style=none] (20) at (-1, -2.1) {};
		\node [style=none] (21) at (-2.175, -1.75) {};
		\node [style=none] (23) at (0.5, -0.9) {};
		\node [style=none] (24) at (1, -0.55) {};
		\node [style=none] (25) at (-1.6, 1.275) {};
		\node [style=none] (26) at (-1.25, 1.575) {};
		\node [style=none] (27) at (-0.675, 2.375) {};
		\node [style=none] (28) at (-0.15, 2.025) {};
		\node [style=none] (29) at (1.725, 2.075) {};
		\node [style=none] (32) at (0.9, 2.45) {};
		\node [style=none] (33) at (3.675, 1.075) {};
		\node [style=none] (34) at (2.175, -1.55) {};
		\node [style=none] (35) at (2.75, -1.85) {};
		\node [style=none] (36) at (3.25, -1.675) {};
		\node [style=none] (37) at (3.575, -1) {};
		\node [style=none] (39) at (1.475, -0.8) {};
		\node [style=none, label={$C$}] (40) at (-1.6, -2) {};
		\node [style=none, label={$C'$}] (42) at (3.4, -0.9) {};
		\node [style=none] (43) at (2.725, 2.975) {};
		\node [style=none] (44) at (-1.75, -0.25) {};
		\node [style=none] (45) at (-1.25, 0.25) {};
		\node [style=none] (46) at (2.25, 0.25) {};
		\node [style=none] (47) at (2.825, 0.5) {};
	\end{pgfonlayer}
	\begin{pgfonlayer}{edgelayer}
		\draw [bend left=45] (1.center) to (0.center);
		\draw [bend left=45] (0.center) to (3.center);
		\draw [bend left=45] (3.center) to (2.center);
		\draw [bend left=45] (2.center) to (1.center);
		\draw [style=new edge style 4] (8.center)
			 to (5.center)
			 to (4.center)
			 to (6.center)
			 to (7.center)
			 to cycle;
		\draw (5.center) to (11.center);
		\draw [style=new edge style 4] (11.center) to (9.center);
		\draw [style=new edge style 4, in=195, out=15] (9.center) to (10.center);
		\draw [style=new edge style 4] (10.center) to (12.center);
		\draw [style=new edge style 4] (12.center) to (17.center);
		\draw (9.center) to (8.center);
		\draw (12.center) to (15.center);
		\draw [style=new edge style 4] (13.center)
			 to (15.center)
			 to (16.center)
			 to (14.center)
			 to cycle;
		\draw (13.center) to (10.center);
		\draw [style=new edge style 3] (4.center) to (18.center);
		\draw [style=new edge style 3] (6.center) to (19.center);
		\draw [style=new edge style 3] (6.center) to (21.center);
		\draw [style=new edge style 3] (7.center) to (20.center);
		\draw [style=new edge style 3] (5.center) to (25.center);
		\draw [style=new edge style 3] (5.center) to (26.center);
		\draw (15.center) to (29.center);
		\draw [style=new edge style 3] (16.center) to (33.center);
		\draw [style=new edge style 3] (14.center) to (34.center);
		\draw [style=new edge style 3] (14.center) to (35.center);
		\draw [style=new edge style 3] (14.center) to (36.center);
		\draw [style=new edge style 3] (14.center) to (37.center);
		\draw (13.center) to (39.center);
		\draw (10.center) to (24.center);
		\draw (9.center) to (23.center);
		\draw (23.center) to (24.center);
		\draw (24.center) to (39.center);
		\draw (39.center) to (23.center);
		\draw (23.center) to (8.center);
		\draw (12.center) to (9.center);
		\draw (27.center) to (17.center);
		\draw (27.center) to (28.center);
		\draw (28.center) to (17.center);
		\draw (28.center) to (32.center);
		\draw (32.center) to (12.center);
		\draw (12.center) to (29.center);
		\draw (29.center) to (32.center);
		\draw (27.center) to (32.center);
		\draw [style=new edge style 19, bend left=60] (44.center) to (45.center);
		\draw [style=new edge style 19, in=120, out=105, looseness=1.75] (46.center) to (47.center);
	\end{pgfonlayer}
\end{tikzpicture}}
    \end{subfloat}
    \hspace{0.5cm}
    \begin{subfloat}[Embedding of a subgraph of $H$ in $\Pi_H'$]{\begin{tikzpicture}[scale=1.35]
	\begin{pgfonlayer}{nodelayer}
		\node [style=none] (0) at (6.3, 4) {};
		\node [style=none] (1) at (6.325, -4) {};
		\node [style=none] (2) at (2.1, 0) {};
		\node [style=none] (3) at (3.35, 4) {};
		\node [style=none] (4) at (4.6, 0) {};
		\node [style=none] (5) at (3.35, -4) {};
		\node [style=none] (7) at (-0.15, 4) {};
		\node [style=none] (9) at (-0.15, -4) {};
		\node [style=none] (10) at (-3.4, 4) {};
		\node [style=none] (11) at (-2.15, 0) {};
		\node [style=none] (12) at (-3.4, -4) {};
		\node [style=none] (13) at (-6.25, 4) {};
		\node [style=none] (14) at (-6.225, -4) {};
		\node [style=none] (15) at (-2.65, 3) {};
		\node [style=none] (16) at (-2.4, 1.75) {};
		\node [style=none] (17) at (-2.575, -2.75) {};
		\node [style=none, label={right:$C'$}] (19) at (4.1, 2.75) {};
		\node [style=none, label={right:$C$}] (20) at (-2.65, 2.75) {};
		\node [style=none] (26) at (4.35, 1.75) {};
		\node [style=none] (27) at (4.425, -1.775) {};
		\node [style=none] (28) at (-2.4, 1.75) {};
		\node [style=none] (29) at (-2.575, -2.75) {};
		\node [style=none] (30) at (-0.925, 0.25) {};
		\node [style=none] (31) at (0.825, 0.25) {};
		\node [style=none] (32) at (-0.675, 1.25) {};
		\node [style=none] (33) at (0.825, 1.25) {};
		\node [style=none] (34) at (4.425, -1.775) {};
		\node [style=none] (35) at (4.35, 1.75) {};
		\node [style=none] (36) at (-0.675, 1.25) {};
		\node [style=none] (37) at (-0.425, -1.4) {};
		\node [style=none] (38) at (0.575, -0.55) {};
		\node [style=none] (39) at (-1.1, 2.375) {};
		\node [style=none] (40) at (-0.075, 1.775) {};
		\node [style=none] (41) at (1.3, 2.075) {};
		\node [style=none] (42) at (0.475, 2.45) {};
		\node [style=none] (43) at (1.3, -1.3) {};
		\node [style=none] (44) at (-4, 2.25) {};
		\node [style=none] (45) at (-3.15, 3) {};
		\node [style=none] (46) at (3.75, 2.75) {};
		\node [style=none] (47) at (3, 3.25) {};
	\end{pgfonlayer}
	\begin{pgfonlayer}{edgelayer}
		\draw [style=new edge style 6] (7.center) to (3.center);
		\draw [in=90, out=-15, looseness=0.50] (3.center) to (4.center);
		\draw [in=0, out=-90, looseness=0.50] (4.center) to (5.center);
		\draw [style=new edge style 6] (5.center) to (9.center);
		\draw [style=new edge style 3, in=165, out=-90, looseness=0.50] (2.center) to (5.center);
		\draw [style=new edge style 3, in=180, out=90, looseness=0.50] (2.center) to (3.center);
		\draw [style=new edge style 4] (13.center) to (10.center);
		\draw [in=90, out=-15, looseness=0.50] (10.center) to (11.center);
		\draw [in=0, out=-90, looseness=0.50] (11.center) to (12.center);
		\draw [style=new edge style 4] (12.center) to (14.center);
		\draw (10.center) to (7.center);
		\draw (3.center) to (0.center);
		\draw (5.center) to (1.center);
		\draw (12.center) to (9.center);
		\draw [style=new edge style 3, bend left=75, looseness=0.50] (12.center) to (10.center);
		\draw (28.center) to (32.center);
		\draw [style=new edge style 4] (32.center) to (30.center);
		\draw [style=new edge style 4] (30.center) to (31.center);
		\draw [style=new edge style 4] (31.center) to (33.center);
		\draw [style=new edge style 4] (33.center) to (36.center);
		\draw (30.center) to (29.center);
		\draw (33.center) to (35.center);
		\draw (34.center) to (31.center);
		\draw (35.center) to (41.center);
		\draw (34.center) to (43.center);
		\draw (31.center) to (38.center);
		\draw (30.center) to (37.center);
		\draw (37.center) to (38.center);
		\draw (38.center) to (43.center);
		\draw (43.center) to (37.center);
		\draw (37.center) to (29.center);
		\draw (33.center) to (30.center);
		\draw (39.center) to (36.center);
		\draw (39.center) to (40.center);
		\draw (40.center) to (36.center);
		\draw (40.center) to (42.center);
		\draw (42.center) to (33.center);
		\draw (33.center) to (41.center);
		\draw (41.center) to (42.center);
		\draw (39.center) to (42.center);
		\draw [style=new edge style 19, in=105, out=90, looseness=2.25] (44.center) to (45.center);
		\draw [style=new edge style 19, in=60, out=90, looseness=1.75] (46.center) to (47.center);
	\end{pgfonlayer}
\end{tikzpicture}}
    \end{subfloat}

    \caption{Two cycles that have the same relative orientation. The cycles $C$ and $C'$ are $\Pi_H$-contractible (see Subfigure (a)) and $\Pi_H'$-noncontractible homotopic (see Subfigure (b)). They have the same relative orientation in $\Pi_H$ and $\Pi_H'$.}
    \label{fig:same_relative_orientation}
\end{figure}
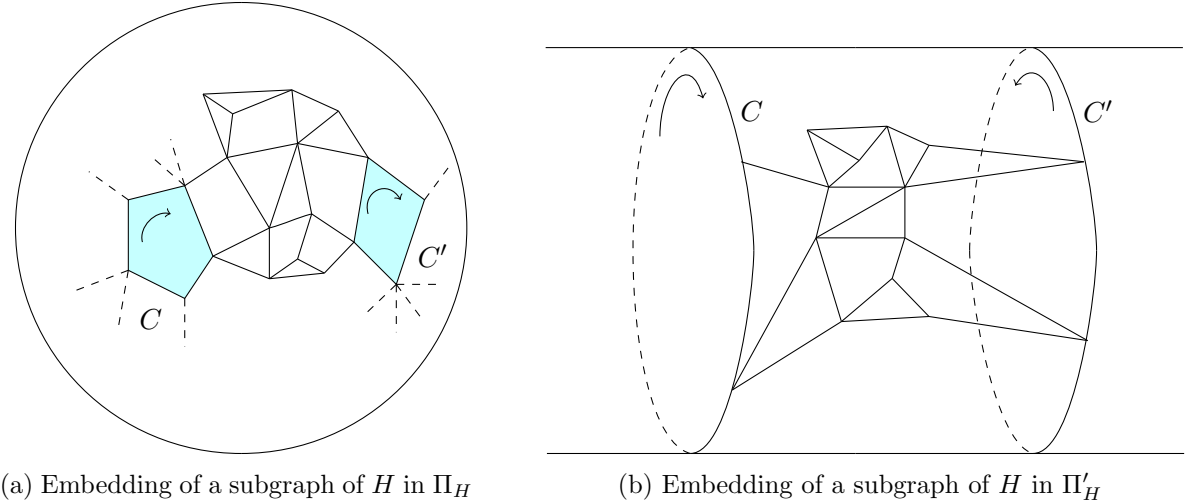

\begin{lem}[{\cite[Lemma 4.13]{HK2026}}]
    \label{cylinder_reembedding}
    Let $H$ be a graph that is $\Pi_H$-embedded in a surface $S_H$ and $\Pi_{H'}$-embedded in a surface $S_H'$. Let $C$ and $C'$ be two almost disjoint cycles in $H$ that are $\Pi_H$-contractible and $\Pi_H'$-noncontractible homotopic. Let $(A,B)$ be the separation of $H$ such that $V(C \cup C') = V(A \cap B)$ and $\text{Int}(C \cup C', \Pi_{H'}) = A$. Suppose moreover that $C \cup C' \cup A$ is a $\Pi_H$-contractible subgraph of $H$ embedded in some disk or cylinder on $S_H$, and that $C$ and $C'$ have the same relative orientation in $\Pi_H$ and in $\Pi_{H'}$.
    
    Then there exists an embedding $\Pi'_H$ of $H$ in $S_H'$ so that $\Pi_H(A)$ and $\Pi'_H(A)$ are equivalent.
\end{lem}

\subsection{Separators in \texorpdfstring{$G$}{G}}

We analyze the connectivity of $G$. We show that there is no $2$-separation between a $\Pi$-contractible subgraph $B$ and the rest of $G$, unless $B$ is reduced to an edge. Then we prove that it suffices to consider an excluded minor $G$ that is $2$-connected, since the other cases can be reduced to this case.

\begin{lem}[{\cite[Lemma 4.14]{HK2026}}]
    \label{2_separated_subgraph_is_edge}
    Let $(A,B)$ be a $2$-separation of $G$ and let $a,b \in V(G)$ so that $A \cap B = \{a,b\}$. Then, either $B$ is an edge or is not contained in a disk in $\Pi$. 
\end{lem}

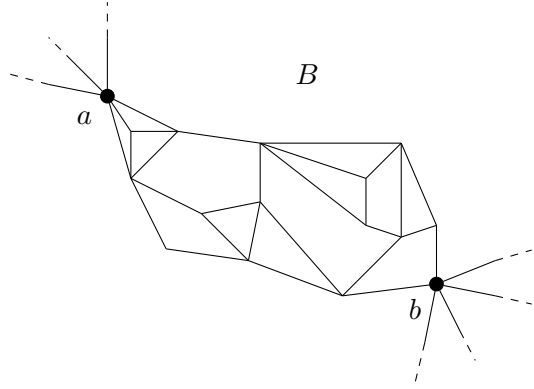
\begin{figure}[h!]
    \centering
    \begin{tikzpicture}[scale=1.25]
	\begin{pgfonlayer}{nodelayer}
		\node [style=basic, label={below left:$a$}] (0) at (-4, 3) {};
		\node [style=basic, label={below left:$b$}] (1) at (3, -1) {};
		\node [style=none] (2) at (-5.25, 3.25) {};
		\node [style=none] (3) at (-4.75, 3.75) {};
		\node [style=none] (4) at (-4, 4.25) {};
		\node [style=none] (5) at (4.25, -0.5) {};
		\node [style=none] (6) at (4.25, -1.25) {};
		\node [style=none] (7) at (3.5, -2) {};
		\node [style=none] (8) at (2.75, -2.25) {};
		\node [style=none] (9) at (-4, 5) {};
		\node [style=none] (10) at (-5.25, 4.25) {};
		\node [style=none] (11) at (-6.1, 3.475) {};
		\node [style=none] (12) at (2.55, -3.075) {};
		\node [style=none] (13) at (3.825, -2.625) {};
		\node [style=none] (14) at (5.025, -1.425) {};
		\node [style=none] (15) at (5.025, -0.275) {};
		\node [style=none] (16) at (-3.5, 1.25) {};
		\node [style=none] (17) at (-2.5, 2.25) {};
		\node [style=none] (18) at (-3.5, 2.25) {};
		\node [style=none] (19) at (-0.75, 2) {};
		\node [style=none] (20) at (-2.75, -0.25) {};
		\node [style=none] (21) at (-2, 0.5) {};
		\node [style=none] (22) at (-0.75, 0.75) {};
		\node [style=none] (23) at (-1, -0.5) {};
		\node [style=none] (24) at (1.5, 0.25) {};
		\node [style=none] (25) at (1, -1.25) {};
		\node [style=none] (26) at (2.25, 0) {};
		\node [style=none] (27) at (1.5, 1.25) {};
		\node [style=none] (28) at (2.25, 2) {};
		\node [style=none] (29) at (3, 0.25) {};
		\node [style=none, label={$B$}] (30) at (0.25, 3) {};
	\end{pgfonlayer}
	\begin{pgfonlayer}{edgelayer}
		\draw (1) to (8.center);
		\draw (1) to (7.center);
		\draw (1) to (6.center);
		\draw (1) to (5.center);
		\draw (0) to (2.center);
		\draw (0) to (3.center);
		\draw (0) to (4.center);
		\draw [style=new edge style 3] (5.center) to (15.center);
		\draw [style=new edge style 3] (6.center) to (14.center);
		\draw [style=new edge style 3] (7.center) to (13.center);
		\draw [style=new edge style 3] (8.center) to (12.center);
		\draw [style=new edge style 3] (4.center) to (9.center);
		\draw [style=new edge style 3] (3.center) to (10.center);
		\draw [style=new edge style 3] (2.center) to (11.center);
		\draw (0) to (18.center);
		\draw (0) to (16.center);
		\draw (0) to (17.center);
		\draw (18.center) to (17.center);
		\draw (16.center) to (17.center);
		\draw (18.center) to (16.center);
		\draw (16.center) to (20.center);
		\draw (16.center) to (21.center);
		\draw (17.center) to (19.center);
		\draw (19.center) to (22.center);
		\draw (22.center) to (21.center);
		\draw (20.center) to (23.center);
		\draw (23.center) to (25.center);
		\draw (25.center) to (1);
		\draw (22.center) to (23.center);
		\draw (22.center) to (25.center);
		\draw (21.center) to (23.center);
		\draw (19.center) to (28.center);
		\draw (19.center) to (27.center);
		\draw (19.center) to (24.center);
		\draw (24.center) to (26.center);
		\draw (26.center) to (25.center);
		\draw (24.center) to (27.center);
		\draw (27.center) to (28.center);
		\draw (28.center) to (26.center);
		\draw (28.center) to (29.center);
		\draw (29.center) to (1);
		\draw (29.center) to (26.center);
	\end{pgfonlayer}
\end{tikzpicture}
    \caption{The subgraph $B$ is 2-separated from the rest of the graph by $\{a, b\}$ and is contained in a disk in $\Pi$.}
    \label{fig:2_separation}
\end{figure}

\begin{prop}[{\cite[Theorem 4.4.2]{graphs_on_surfaces}}]
    \label{4.4.2}
    Let $H$ be a connected graph and $H_1, ..., H_p$ ($p \geq 1$) be its 2-connected blocks. Then, \[ g(H) = g(H_1) + ... + g(H_p)\]
\end{prop}

\begin{lem}[{\cite[Lemma 4.15]{HK2026}}]
    \label{G_blocks_excluded_minors}
    Let $G_1, ..., G_p$ ($p \geq 1$) be the $2$-connected blocks of $G$. Then, for $1 \leq i \leq p$, $G_i$ is an excluded minor for some surface $S_i$.
\end{lem}

The following result is a refinement of a result from \cite[Theorem 4.16]{HK2026}. By adding basic constraints on the bounding function $N$, we can extend the bound $N(g)$ on the order of the $2$-connected excluded minors for a surface to a bound for all its excluded minors with the same function $N(g)$. This improves the result from \cite[Theorem 4.16]{HK2026}, which had an added factor $g$.

\begin{cor}
    \label{G_2_connected}
    Suppose that, for any graph $H$ that is an excluded minor for some surface $S_H$ and that is $2$-connected, $|V(H)| \leq N(g(S_H))$ with $N$ having the two following properties:
    \begin{itemize}
        \item $N$ is an increasing function
        \item $N(g_1 + g_2) \geq N(g_1) + N(g_2)$.
    \end{itemize}

    Then, $|V(G)| \leq N(g)$.
\end{cor}

\begin{proof}
    If $G$ is $2$-connected, the result follows immediately.

    Suppose not: let $G_1, ..., G_p$ ($p \geq 2$) be the $2$-connected blocks of $G$. Then, by \cref{G_blocks_excluded_minors}, for $1 \leq i \leq p$, $G_i$ is an excluded minor for some surface $S_i$. 
    Remark that, by \cref{4.4.2}, $g(G_1)+...+g(G_p) = g(G)$.
    
    Finally,
    \begin{align*}
        |V(G)| &\leq \sum_{i=1}^p |V(G_i)| &&& \text{by \cref{4.4.2}} \\
        & \leq \sum_{i=1}^p N(g(S_i)) &&& \text{by hypothesis}\\
    \end{align*}
    \begin{align*}
        \hspace{3cm} & \leq \sum_{i=1}^p N(g(G_i)-1) &&& \text{because } g(G_i) > g(S_i)\\
        & \leq N\left(\sum_{i=1}^p (g(G_i)-1)\right) &&& \text{by using the second property of } N \\
        & \leq N(g(G)-p) &&& \text{because } \sum_{i=1}^p g(G_i) = g(G)\\
        & \leq N(g) &&& \text{because }  g(G) \leq g+2 \text{ and } p \geq 2
    \end{align*} 
\end{proof}

\cref{G_2_connected} implies that bounding the order of the $2$-connected excluded minors for $S'$ yields the same bound on the order of all the excluded minors for $S'$. Therefore, from this point on, we assume that $G$ is $2$-connected.

\section{Structural tools for the main proof} \label{sec:nested_cycles}

In this section, we prove three structural results on $G$: \cref{good_square_cor_general,good_square_cor_attachments,good_square_variant}. These results describe forbidden structures in $(G,\Pi)$ and will be used repeatedly to find contradictions in the rest of this paper. They constitute one of the main ingredients to obtain the main result of this paper.

\vspace{1em}

These results are improvements over the results of Seymour \cite[(2.2)]{seymour} and, more recently, over the results in \cite[Section 5]{HK2026}. These results indeed generalize three results in \cite{HK2026}: the two types of forbidden structures, isolated paths and nested/homotopic squares, from \cite{HK2026} are in this paper unified into the same forbidden structure, called nested/homotopic cycles.
The proofs in this section are largely inspired by the proofs in \cite[Section 5.2]{HK2026}.

\vspace{1em}

This unification allows us to find a better bound for the structure formerly called isolated paths: in \cite{HK2026}, it is proved that $(G,\Pi)$ contains $O(g)$ isolated paths, whereas in this paper we prove that $(G, \Pi)$ contains $O(\log g)$ isolated paths. Apart from the extension of the scope of the structural results, the bound for nested/homotopic cycles remains asymptotically the same as in \cite{HK2026}.

While \cref{good_square_cor_general,good_square_variant} are direct extensions of results proven in \cite{HK2026}, \cref{good_square_cor_attachments} is a variation of \cref{good_square_cor_general} which finds a bound on the number of nested cycles that depends on the size of a separator for the considered subgraph. This dependency on the size of a separator will be taken advantage of in the rest of the paper: by finding a contractible subgraph with a separator of logarithmic size in $g$, we can prove that it contains $O(\log \log g)$ nested cycles (see \cref{loglog_bound_G0}), which is a crucial ingredient to establish the polynomial bound on the order of an excluded minor for a surface of genus $g$.

\subsection{Contractible nested cycles} \label{subsec:contractible_nested_cycles}

Let $C$ be a $\Pi$-contractible cycle of $G$. Then, let $\mathcal{F}(C, \Pi)$ be the set of all faces inside of $C$ in $\Pi$. Moreover, for a subgraph $G_0$ of $G$ that is $\Pi$-contractible, let $\mathcal{F}(G_0, \Pi) = \bigcup_{C \subseteq G_0} \mathcal{F}(C, \Pi)$.

\begin{defi}[Nested cycles]
    Let $H$ be a graph with embedding $\Pi_H$. Let $C$ and $C'$ be two $\Pi_H$-contractible cycles of $H$ such that $C' \subseteq \text{Int}(C, \Pi_H)$. We say that $C'$ is \textit{$\Pi_H$-nested} (or \textit{nested} if it is clear in the context) in $C$.
    
    Let $k \geq 1$ and let $C_0, ..., C_k$ be $\Pi_H$-contractible cycles of $H$.
    We say that $C_0, ..., C_k$ are \textit{$\Pi_H$-nested} if for $0 \leq i \leq k-1$, $C_i$ is $\Pi_H$-nested in $C_{i+1}$.

    For instance, in \cref{fig:contractible_square}, the cycles $C, C', C''$ are nested in the reversed order in all three figures.
\end{defi}

For two $\Pi$-contractible nested cycles $C$ and $C'$, we denote by $\text{Int}(C \cup C', \Pi)$ the subgraph of $G$ inside the surface bounded by $C$ and $C'$ in $\Pi$, together with $C$ and $C'$. Therefore, $\text{Int}(C \cup C', \Pi ) = \text{Int}(C, \Pi) - \text{int}(C', \Pi)$.

\begin{defi}[Well nested]
    \label{def:well_nested_cycles}
    Let $H$ be a graph with embedding $\Pi_H$. Let $C$ and $C'$ be two $\Pi_H$-contractible cycles of $H$ such that $C'$ is $\Pi_H$-nested in $C$. We say that $C'$ is \textit{$\Pi_H$-well-nested} (or \textit{well nested}, if it is clear in the context) in $C$ if one of the following condition holds:
    \begin{itemize}
        \item $C$ and $C'$ are disjoint (\textit{(free,free)-well-nested} or \textit{freely well nested}), or 
        \item $C$ and $C'$ intersect at a vertex $v$ (\textit{($v$,free)-well-nested} or \textit{well nested pinched on the vertex $v$}), or 
        \item $C$ and $C'$ both intersect the same face $f$ of $(G, \Pi)$ and if $C$ intersects $f$ in a subpath $P$ (of length at least $3$) then $C'$ intersect $f$ in a subpath that is in the interior of $P$ (\textit{($f$,free)-well-nested} or \textit{well nested pinched on the face $f$}), or
        \item $C$ and $C'$ intersect at two disjoint vertices $v$ and $v'$ (\textit{($v$,$v'$)-well-nested} or \textit{well nested pinched on two vertices $v$ and $v'$}), or
        \item $C$ and $C'$ both intersect the same two faces $f$ and $f'$ of $(G, \Pi)$ and if $C$ intersects $f$ (resp. $f'$) in a subpath $P$ (resp. $P'$), of length at least $3$, then $C'$ intersect $f$ (resp. $f'$) in a subpath that is in the interior of $P$ (resp. $P'$) (\textit{($f$,$f'$)-well-nested} or \textit{well nested pinched on the faces $f$ and $f'$}), or
        \item $C$ and $C'$ both intersect at a vertex $v$ and a face $f$ of $(G, \Pi)$ and if $C$ intersects $f$ in a subpath $P$ (of length at least $3$) then $C'$ intersect $f$ in a subpath that is in the interior of $P$ (\textit{($v$,$f$)-well-nested} or \textit{well nested pinched on the vertex $v$ and the face $f$}).
    \end{itemize}

    In the second and third cases, we say that $C$ and $C'$ are \textit{well nested pinched on a piece $p$}. In the last three cases, we say that $C$ and $C'$ are \textit{well nested pinched on two pieces $p$ and $p'$}.

    \begin{figure}[H]
    \centering
    \begin{subfloat}[Full contractible square]{\input{images/contractible_square.tikz}}
    \end{subfloat}
    \hspace{0.5cm}
    \begin{subfloat}[Contractible square pinched on a vertex]{\input{images/contractible_square_pinched_on_vertex.tikz}}
    \end{subfloat}
    
    \vspace{0.5cm}
    \begin{subfloat}[Contractible square pinched on a face]{\begin{tikzpicture}[scale=0.63]
	\begin{pgfonlayer}{nodelayer}
		\node [style=none] (6) at (0, -5.5) {};
		\node [style=none] (7) at (-4.875, -1) {};
		\node [style=none] (8) at (-7.875, -1) {};
		\node [style=none] (10) at (7.85, -1.575) {};
		\node [style=none] (11) at (0, -8.5) {};
		\node [style=none] (12) at (-7.9, 0) {};
		\node [style=none] (14) at (-4.85, 0) {};
		\node [style=none] (15) at (-7.85, -1.575) {};
		\node [style=none] (16) at (-7.8, -2.1) {};
		\node [style=none] (18) at (-4.875, -1.5) {};
		\node [style=none] (19) at (-7.65, -2.825) {};
		\node [style=none] (20) at (-7.4, -3.525) {};
		\node [style=none] (21) at (-6.175, -5.525) {};
		\node [style=none] (22) at (-7.05, -4.175) {};
		\node [style=none] (23) at (-6.7, -4.775) {};
		\node [style=none] (26) at (-5.15, -6.575) {};
		\node [style=none] (29) at (-2.675, -4.825) {};
		\node [style=none] (30) at (-1.625, -5.35) {};
		\node [style=none] (33) at (-3.625, -7.6) {};
		\node [style=none] (34) at (-2.65, -8.025) {};
		\node [style=none] (35) at (-1.25, -8.425) {};
		\node [style=none] (36) at (0.925, -5.425) {};
		\node [style=none] (37) at (1.55, -5.275) {};
		\node [style=none] (38) at (1.825, -8.25) {};
		\node [style=none] (39) at (2.6, -8.025) {};
		\node [style=none] (40) at (3.05, -7.825) {};
		\node [style=none] (41) at (3.55, -7.6) {};
		\node [style=none] (42) at (4.05, -7.35) {};
		\node [style=none] (43) at (4.55, -7) {};
		\node [style=none] (47) at (5.3, -6.35) {};
		\node [style=none] (48) at (5.775, -5.925) {};
		\node [style=none] (51) at (7.25, -3.775) {};
		\node [style=none] (52) at (6.9, -4.45) {};
		\node [style=none, label={right:$C$}] (53) at (6.45, -5.175) {};
		\node [style=none] (54) at (7.5, -3.25) {};
		\node [style=none] (55) at (7.675, -2.725) {};
		\node [style=none] (56) at (7.775, -2.125) {};
		\node [style=none] (57) at (7.875, -1.025) {};
		\node [style=none] (58) at (7.9, -0.45) {};
		\node [style=none] (62) at (7.925, 0.025) {};
		\node [style=none] (1) at (3, 0) {};
		\node [style=none] (2) at (0, -2.75) {};
		\node [style=none] (3) at (-3, 0) {};
		\node [style=none] (17) at (-6.875, -1) {};
		\node [style=none] (24) at (-5.875, -3.5) {};
		\node [style=none] (25) at (-5.375, -4.25) {};
		\node [style=none] (27) at (-5.875, -1.975) {};
		\node [style=none] (28) at (-4.55, -4.2) {};
		\node [style=none] (31) at (-3.675, -6.375) {};
		\node [style=none] (32) at (-2.9, -6.125) {};
		\node [style=none] (44) at (3.675, -6.6) {};
		\node [style=none] (45) at (2.8, -6.5) {};
		\node [style=none] (46) at (3.85, -5.6) {};
		\node [style=none] (49) at (1.7, -7.25) {};
		\node [style=none] (59) at (6.55, -1.25) {};
		\node [style=none] (60) at (6.05, -1.25) {};
		\node [style=none] (61) at (7.225, -1.45) {};
		\node [style=none] (63) at (7.325, -2.225) {};
		\node [style=none] (64) at (7.225, -0.725) {};
		\node [style=none] (66) at (6.55, -2.475) {};
		\node [style=none] (5) at (4.925, -1.125) {};
		\node [style=none, label={left:$C'$}] (50) at (4.3, -3.2) {};
		\node [style=none] (82) at (4.9, 0) {};
		\node [style=none] (83) at (5.425, 0) {};
		\node [style=none] (84) at (5.4, 0) {};
		\node [style=none] (65) at (5.65, -2.025) {};
		\node [style=none] (67) at (5.8, -3.5) {};
		\node [style=none] (86) at (8.375, 1.25) {};
		\node [style=none] (87) at (8, 2.75) {};
		\node [style=none] (90) at (6.75, 4) {};
		\node [style=none] (92) at (5, 5) {};
		\node [style=none] (94) at (3, 5.6) {};
		\node [style=none] (96) at (1, 5.95) {};
		\node [style=none] (97) at (0, 6.025) {};
		\node [style=none] (98) at (-8.375, 1.225) {};
		\node [style=none] (99) at (-8, 2.725) {};
		\node [style=none] (100) at (-6.75, 3.975) {};
		\node [style=none] (101) at (-5, 4.975) {};
		\node [style=none] (102) at (-3, 5.6) {};
		\node [style=none] (103) at (-1, 5.925) {};
		\node [style=none, label={left:$C''$}] (104) at (2.4, -1.7) {};
	\end{pgfonlayer}
	\begin{pgfonlayer}{edgelayer}
		\draw [style=new edge style 4] (41.center)
			 to (42.center)
			 to (43.center)
			 to (47.center)
			 to (48.center)
			 to (53.center)
			 to (52.center)
			 to (51.center)
			 to (54.center)
			 to (55.center)
			 to (56.center)
			 to (10.center)
			 to (57.center)
			 to (58.center)
			 to (62.center)
			 to (84.center)
			 to [bend left=360] (83.center)
			 to (82.center)
			 to [bend left=15, looseness=0.75] (5.center)
			 to [bend left=15, looseness=0.75] (50.center)
			 to [bend left, looseness=0.75] (37.center)
			 to (36.center)
			 to (6.center)
			 to [bend left=10] (30.center)
			 to [bend left=15, looseness=0.75] (29.center)
			 to [bend left, looseness=0.75] (18.center)
			 to (7.center)
			 to (14.center)
			 to (12.center)
			 to (8.center)
			 to (15.center)
			 to (16.center)
			 to (19.center)
			 to (20.center)
			 to (22.center)
			 to (23.center)
			 to (21.center)
			 to (26.center)
			 to (33.center)
			 to (34.center)
			 to (35.center)
			 to (11.center)
			 to [bend right=10] (38.center)
			 to (39.center)
			 to (40.center)
			 to cycle;
		\draw [style=new edge style 6] (2.center)
			 to [bend left=45] (3.center)
			 to (1.center)
			 to [bend left=45] cycle;
		\draw [style=new edge style 4] (15.center) to (8.center);
		\draw [style=new edge style 5] (18.center) to (27.center);
		\draw [style=new edge style 5] (28.center) to (29.center);
		\draw [style=new edge style 4] (7.center) to (12.center);
		\draw [style=new edge style 4] (8.center) to (17.center);
		\draw [style=new edge style 4, bend left=15] (17.center) to (16.center);
		\draw [style=new edge style 4, bend right=15] (15.center) to (17.center);
		\draw [style=new edge style 4] (17.center) to (7.center);
		\draw [style=new edge style 4, bend left, looseness=1.25] (20.center) to (24.center);
		\draw [style=new edge style 5] (28.center)
			 to (24.center)
			 to [bend left] (25.center)
			 to cycle;
		\draw [style=new edge style 4, bend left] (25.center) to (21.center);
		\draw [style=new edge style 4] (22.center) to (24.center);
		\draw [style=new edge style 4] (23.center) to (25.center);
		\draw [style=new edge style 4] (31.center) to (33.center);
		\draw [style=new edge style 4] (32.center) to (34.center);
		\draw [style=new edge style 4] (30.center) to (35.center);
		\draw [style=new edge style 4] (30.center) to (11.center);
		\draw [style=new edge style 4] (6.center) to (11.center);
		\draw [style=new edge style 4] (36.center) to (11.center);
		\draw [style=new edge style 4, bend right=60, looseness=1.50] (48.center) to (49.center);
		\draw [style=new edge style 4] (52.center) to (50.center);
		\draw [style=new edge style 4] (53.center) to (50.center);
		\draw [style=new edge style 4] (10.center) to (61.center);
		\draw [style=new edge style 4] (56.center) to (61.center);
		\draw [style=new edge style 4] (57.center) to (61.center);
		\draw [style=new edge style 4, bend right=15] (62.center) to (60.center);
		\draw [style=new edge style 4] (59.center) to (60.center);
		\draw [style=new edge style 5] (61.center)
			 to (64.center)
			 to [bend right=15] (59.center)
			 to [bend right=15] (63.center)
			 to cycle;
		\draw [style=new edge style 4, bend left=15] (66.center) to (65.center);
		\draw [style=new edge style 5] (50.center) to (67.center);
		\draw [style=new edge style 4, bend left=15] (65.center) to (84.center);
		\draw [style=new edge style 4, bend right=15] (60.center) to (66.center);
		\draw [style=new edge style 4, bend right=15] (66.center) to (54.center);
		\draw [style=new edge style 4] (67.center) to (51.center);
		\draw [style=new edge style 4] (37.center) to (49.center);
		\draw [style=new edge style 4] (49.center) to (38.center);
		\draw [style=new edge style 4] (30.center) to (32.center);
		\draw [style=new edge style 4] (32.center) to (31.center);
		\draw [style=new edge style 4] (31.center) to (26.center);
		\draw [style=new edge style 4, bend left=15] (28.center) to (26.center);
		\draw [style=new edge style 4] (27.center) to (19.center);
		\draw [style=new edge style 5] (58.center) to (64.center);
		\draw [style=new edge style 5] (55.center) to (63.center);
		\draw [style=new edge style 4] (43.center) to (44.center);
		\draw [style=new edge style 4] (42.center) to (44.center);
		\draw [style=new edge style 4] (41.center) to (44.center);
		\draw [style=new edge style 4, bend right] (47.center) to (46.center);
		\draw [style=new edge style 4, bend right=15, looseness=1.25] (46.center) to (45.center);
		\draw [style=new edge style 4, bend right=15] (45.center) to (39.center);
		\draw [style=new edge style 4, bend right=15] (45.center) to (40.center);
		\draw (14.center) to (3.center);
		\draw (1.center) to (82.center);
		\draw [style=new edge style 5, bend left=360] (84.center) to (83.center);
		\draw [style=new edge style 5] (83.center) to (82.center);
		\draw [style=new edge style 5, bend left=15, looseness=0.75] (82.center) to (5.center);
		\draw [style=new edge style 5, bend left=15, looseness=0.75] (5.center) to (50.center);
		\draw [style=new edge style 5] (50.center) to (67.center);
		\draw [style=new edge style 3] (12.center)
			 to [bend left, looseness=0.75] (98.center)
			 to [bend left=15, looseness=0.75] (99.center)
			 to [bend left, looseness=0.25] (100.center)
			 to [bend left=15, looseness=0.50] (101.center)
			 to [bend left=15, looseness=0.50] (102.center)
			 to [bend left=15, looseness=0.50] (103.center)
			 to [bend left=15, looseness=0.50] (97.center)
			 to [bend left=15, looseness=0.50] (96.center)
			 to [bend left=15, looseness=0.50] (94.center)
			 to [bend left=15, looseness=0.50] (92.center)
			 to [bend left=15, looseness=0.50] (90.center)
			 to [bend left, looseness=0.25] (87.center)
			 to [bend left=15, looseness=0.75] (86.center)
			 to [bend left, looseness=0.75] (62.center);
	\end{pgfonlayer}
\end{tikzpicture}}
    \end{subfloat}
    \hspace{0.5cm}
    \begin{subfloat}[Contractible square pinched on two vertices]{\begin{tikzpicture}[scale=1]
	\begin{pgfonlayer}{nodelayer}
		\node [style=none] (6) at (0, -8.5) {};
		\node [style=none] (7) at (-1.325, -1) {};
		\node [style=none] (8) at (-3.4, -1) {};
		\node [style=none] (10) at (4.1, -1.575) {};
		\node [style=none] (11) at (0, -8.5) {};
		\node [style=none] (12) at (0, 0) {};
		\node [style=none] (14) at (0, 0) {};
		\node [style=none] (15) at (-4.075, -1.575) {};
		\node [style=none] (16) at (-4.475, -2.1) {};
		\node [style=none] (18) at (-2, -1.75) {};
		\node [style=none] (19) at (-4.825, -2.825) {};
		\node [style=none] (20) at (-5, -3.275) {};
		\node [style=none] (21) at (-4.925, -5.525) {};
		\node [style=none] (22) at (-5.1, -4.175) {};
		\node [style=none] (23) at (-5.075, -4.775) {};
		\node [style=none] (26) at (-4.35, -6.575) {};
		\node [style=none] (29) at (-2.65, -5.075) {};
		\node [style=none] (30) at (-2.375, -6) {};
		\node [style=none] (33) at (-3.5, -7.45) {};
		\node [style=none] (34) at (-2.65, -7.95) {};
		\node [style=none] (35) at (-1.25, -8.425) {};
		\node [style=none] (36) at (1.425, -7.45) {};
		\node [style=none] (37) at (2.1, -6.55) {};
		\node [style=none] (38) at (1.825, -8.25) {};
		\node [style=none] (39) at (2.6, -8.025) {};
		\node [style=none] (40) at (3.05, -7.825) {};
		\node [style=none] (41) at (3.45, -7.6) {};
		\node [style=none] (42) at (3.725, -7.35) {};
		\node [style=none] (43) at (4.075, -7) {};
		\node [style=none] (47) at (4.575, -6.35) {};
		\node [style=none] (48) at (4.8, -5.925) {};
		\node [style=none] (51) at (5.15, -3.775) {};
		\node [style=none] (52) at (5.175, -4.45) {};
		\node [style=none, label={right:$C$}] (53) at (5.075, -5.175) {};
		\node [style=none] (54) at (5.025, -3.25) {};
		\node [style=none] (55) at (4.875, -2.725) {};
		\node [style=none] (56) at (4.525, -2.125) {};
		\node [style=none] (57) at (3.55, -1.025) {};
		\node [style=none] (58) at (2.5, -0.45) {};
		\node [style=none] (62) at (0, 0) {};
		\node [style=none] (1) at (0, 0) {};
		\node [style=none] (2) at (0, -8.5) {};
		\node [style=none] (3) at (0, 0) {};
		\node [style=none] (17) at (-2.5, -1) {};
		\node [style=none] (24) at (-4.3, -3.5) {};
		\node [style=none] (25) at (-4.125, -4.25) {};
		\node [style=none] (27) at (-2.875, -2.225) {};
		\node [style=none] (28) at (-3.3, -4.2) {};
		\node [style=none] (31) at (-2.925, -6.375) {};
		\node [style=none] (32) at (-2.65, -6.125) {};
		\node [style=none] (44) at (3.425, -6.85) {};
		\node [style=none] (45) at (3.3, -6.5) {};
		\node [style=none] (46) at (4.1, -6.1) {};
		\node [style=none] (49) at (2.7, -6.75) {};
		\node [style=none] (59) at (2.55, -1.75) {};
		\node [style=none] (60) at (2.05, -1.25) {};
		\node [style=none] (61) at (3.475, -1.7) {};
		\node [style=none] (63) at (3.325, -2.225) {};
		\node [style=none] (64) at (2.825, -0.975) {};
		\node [style=none] (66) at (3.05, -2.725) {};
		\node [style=none] (5) at (1.45, -1.125) {};
		\node [style=none, label={left:$C'$}] (50) at (2.625, -3.2) {};
		\node [style=none] (82) at (0, 0) {};
		\node [style=none] (83) at (0.025, 0) {};
		\node [style=none] (84) at (0, 0) {};
		\node [style=none] (67) at (4.05, -3.5) {};
		\node [style=none, label={left:$C''$}] (104) at (1, -1.7) {};
		\node [style=none] (105) at (0.149999, -8.5) {};
		\node [style=none] (106) at (0.149999, 0) {};
	\end{pgfonlayer}
	\begin{pgfonlayer}{edgelayer}
		\draw [style=new edge style 4] (41.center)
			 to (42.center)
			 to (43.center)
			 to (47.center)
			 to (48.center)
			 to (53.center)
			 to (52.center)
			 to (51.center)
			 to (54.center)
			 to (55.center)
			 to (56.center)
			 to (10.center)
			 to (57.center)
			 to [bend right=15, looseness=0.50] (58.center)
			 to [bend right=15, looseness=0.75] (62.center)
			 to (84.center)
			 to [bend left=360] (83.center)
			 to (82.center)
			 to (5.center)
			 to [bend left=15, looseness=0.50] (50.center)
			 to [bend left, looseness=0.75] (37.center)
			 to [bend left, looseness=0.25] (36.center)
			 to [bend left=15, looseness=0.75] (6.center)
			 to [bend left=15] (30.center)
			 to [bend left=15, looseness=0.75] (29.center)
			 to [bend left, looseness=0.75] (18.center)
			 to [bend left=120, looseness=0.25] (7.center)
			 to [bend left=45, looseness=0.25] (14.center)
			 to (12.center)
			 to [bend right=15] (8.center)
			 to [bend right=150, looseness=0.25] (15.center)
			 to (16.center)
			 to (19.center)
			 to (20.center)
			 to (22.center)
			 to (23.center)
			 to (21.center)
			 to (26.center)
			 to (33.center)
			 to (34.center)
			 to (35.center)
			 to (11.center)
			 to [bend right=10] (38.center)
			 to (39.center)
			 to (40.center)
			 to cycle;
		\draw [style=new edge style 6] (2.center)
			 to [bend left] (3.center)
			 to (1.center)
			 to [bend left] cycle;
		\draw [style=new edge style 5] (18.center) to (27.center);
		\draw [style=new edge style 5] (28.center) to (29.center);
		\draw [style=new edge style 4] (8.center) to (17.center);
		\draw [style=new edge style 4, bend left=15] (17.center) to (16.center);
		\draw [style=new edge style 4, bend right=15] (15.center) to (17.center);
		\draw [style=new edge style 4] (17.center) to (7.center);
		\draw [style=new edge style 4, bend left, looseness=1.25] (20.center) to (24.center);
		\draw [style=new edge style 5] (28.center)
			 to (24.center)
			 to [bend left] (25.center)
			 to cycle;
		\draw [style=new edge style 4, bend left] (25.center) to (21.center);
		\draw [style=new edge style 4] (22.center) to (24.center);
		\draw [style=new edge style 4] (23.center) to (25.center);
		\draw [style=new edge style 4] (31.center) to (33.center);
		\draw [style=new edge style 4] (32.center) to (34.center);
		\draw [style=new edge style 4] (30.center) to (35.center);
		\draw [style=new edge style 4] (6.center) to (11.center);
		\draw [style=new edge style 4, in=60, out=-210, looseness=0.75] (48.center) to (49.center);
		\draw [style=new edge style 4] (52.center) to (50.center);
		\draw [style=new edge style 4] (53.center) to (50.center);
		\draw [style=new edge style 4] (10.center) to (61.center);
		\draw [style=new edge style 4] (56.center) to (61.center);
		\draw [style=new edge style 4] (57.center) to (61.center);
		\draw [style=new edge style 4, bend left=15, looseness=0.75] (62.center) to (60.center);
		\draw [style=new edge style 4] (59.center) to (60.center);
		\draw [style=new edge style 5] (61.center)
			 to (64.center)
			 to [bend right=15] (59.center)
			 to [bend right=15] (63.center)
			 to cycle;
		\draw [style=new edge style 5] (50.center) to (67.center);
		\draw [style=new edge style 4, bend right=15] (60.center) to (66.center);
		\draw [style=new edge style 4, bend right=15, looseness=0.75] (66.center) to (54.center);
		\draw [style=new edge style 4] (67.center) to (51.center);
		\draw [style=new edge style 4] (37.center) to (49.center);
		\draw [style=new edge style 4, bend right=15, looseness=0.75] (49.center) to (38.center);
		\draw [style=new edge style 4] (30.center) to (32.center);
		\draw [style=new edge style 4] (32.center) to (31.center);
		\draw [style=new edge style 4] (31.center) to (26.center);
		\draw [style=new edge style 4, bend left=15] (28.center) to (26.center);
		\draw [style=new edge style 4] (27.center) to (19.center);
		\draw [style=new edge style 5] (58.center) to (64.center);
		\draw [style=new edge style 5] (55.center) to (63.center);
		\draw [style=new edge style 4] (43.center) to (44.center);
		\draw [style=new edge style 4] (42.center) to (44.center);
		\draw [style=new edge style 4] (41.center) to (44.center);
		\draw [style=new edge style 4, bend right, looseness=0.75] (47.center) to (46.center);
		\draw [style=new edge style 4, bend right=15, looseness=1.25] (46.center) to (45.center);
		\draw [style=new edge style 4, bend right=15] (45.center) to (39.center);
		\draw [style=new edge style 4, bend right=15] (45.center) to (40.center);
		\draw (14.center) to (3.center);
		\draw (1.center) to (82.center);
		\draw [style=new edge style 5, bend left=360] (84.center) to (83.center);
		\draw [style=new edge style 5] (83.center) to (82.center);
		\draw [style=new edge style 5] (50.center) to (67.center);
	\end{pgfonlayer}
\end{tikzpicture}}
    \end{subfloat}

    \vspace{0.5cm}
    \begin{subfloat}[Contractible square pinched on two faces]{\input{images/contractible_square_pinched_on_face_face.tikz}}
    \end{subfloat}
    \hspace{0.5cm}
    \begin{subfloat}[Contractible square pinched on a vertex and a face]{\begin{tikzpicture}[scale=0.63]
	\begin{pgfonlayer}{nodelayer}
		\node [style=none] (6) at (0, 5) {};
		\node [style=none] (7) at (-4.825, 12.5) {};
		\node [style=none] (8) at (-7.875, 12.5) {};
		\node [style=none] (10) at (7.85, 11.925) {};
		\node [style=none] (11) at (0, 5) {};
		\node [style=none] (12) at (-7.9, 13.5) {};
		\node [style=none] (14) at (-4.85, 13.5) {};
		\node [style=none] (15) at (-7.85, 11.925) {};
		\node [style=none] (16) at (-7.8, 11.4) {};
		\node [style=none] (18) at (-4.75, 12) {};
		\node [style=none] (19) at (-7.65, 10.675) {};
		\node [style=none] (20) at (-7.4, 9.975) {};
		\node [style=none] (21) at (-6.175, 7.975) {};
		\node [style=none] (22) at (-7.05, 9.325) {};
		\node [style=none] (23) at (-6.7, 8.725) {};
		\node [style=none] (26) at (-5.15, 6.925) {};
		\node [style=none] (29) at (-3.6, 8.425) {};
		\node [style=none] (30) at (-2.975, 7.5) {};
		\node [style=none] (33) at (-3.625, 5.9) {};
		\node [style=none] (34) at (-2.65, 5.475) {};
		\node [style=none] (35) at (-1.25, 5.075) {};
		\node [style=none] (36) at (1.575, 6.05) {};
		\node [style=none] (37) at (2.55, 6.975) {};
		\node [style=none] (38) at (1.825, 5.25) {};
		\node [style=none] (39) at (2.6, 5.475) {};
		\node [style=none] (40) at (3.05, 5.675) {};
		\node [style=none] (41) at (3.55, 5.9) {};
		\node [style=none] (42) at (4.05, 6.15) {};
		\node [style=none] (43) at (4.55, 6.5) {};
		\node [style=none] (47) at (5.3, 7.15) {};
		\node [style=none] (48) at (5.775, 7.575) {};
		\node [style=none] (51) at (7.25, 9.725) {};
		\node [style=none] (52) at (6.9, 9.05) {};
		\node [style=none, label={right:$C$}] (53) at (6.45, 8.325) {};
		\node [style=none] (54) at (7.5, 10.25) {};
		\node [style=none] (55) at (7.675, 10.775) {};
		\node [style=none] (56) at (7.775, 11.375) {};
		\node [style=none] (57) at (7.875, 12.475) {};
		\node [style=none] (58) at (7.9, 13.05) {};
		\node [style=none] (62) at (7.925, 13.525) {};
		\node [style=none] (1) at (3, 13.5) {};
		\node [style=none] (2) at (0, 5) {};
		\node [style=none] (3) at (-3, 13.5) {};
		\node [style=none] (17) at (-6.875, 12.5) {};
		\node [style=none] (24) at (-5.875, 10) {};
		\node [style=none] (25) at (-5.375, 9.25) {};
		\node [style=none] (27) at (-5.875, 11.525) {};
		\node [style=none] (28) at (-4.55, 9.3) {};
		\node [style=none] (31) at (-4.175, 7.125) {};
		\node [style=none] (32) at (-3.4, 7.375) {};
		\node [style=none] (44) at (3.925, 6.65) {};
		\node [style=none] (45) at (3.3, 7) {};
		\node [style=none] (46) at (4.1, 7.4) {};
		\node [style=none] (49) at (2.7, 6.75) {};
		\node [style=none] (59) at (6.55, 12.25) {};
		\node [style=none] (60) at (6.05, 12.25) {};
		\node [style=none] (61) at (7.225, 12.05) {};
		\node [style=none] (63) at (7.325, 11.275) {};
		\node [style=none] (64) at (7.225, 12.775) {};
		\node [style=none] (66) at (6.55, 11.025) {};
		\node [style=none] (5) at (4.825, 12.375) {};
		\node [style=none, label={left:$C'$}] (50) at (4.5, 10.3) {};
		\node [style=none] (82) at (4.9, 13.5) {};
		\node [style=none] (83) at (5.425, 13.5) {};
		\node [style=none] (84) at (5.4, 13.5) {};
		\node [style=none] (65) at (5.65, 11.475) {};
		\node [style=none] (67) at (5.8, 10) {};
		\node [style=none] (86) at (8.375, 14.75) {};
		\node [style=none] (87) at (8, 16.25) {};
		\node [style=none] (90) at (6.75, 17.5) {};
		\node [style=none] (92) at (5, 18.5) {};
		\node [style=none] (94) at (3, 19.1) {};
		\node [style=none] (96) at (1, 19.45) {};
		\node [style=none] (97) at (0, 19.525) {};
		\node [style=none] (98) at (-8.375, 14.725) {};
		\node [style=none] (99) at (-8, 16.225) {};
		\node [style=none] (100) at (-6.75, 17.475) {};
		\node [style=none] (101) at (-5, 18.475) {};
		\node [style=none] (102) at (-3, 19.1) {};
		\node [style=none] (103) at (-1, 19.425) {};
		\node [style=none, label={left:$C''$}] (104) at (3.25, 11.8) {};
	\end{pgfonlayer}
	\begin{pgfonlayer}{edgelayer}
		\draw [style=new edge style 4] (41.center)
			 to (42.center)
			 to (43.center)
			 to (47.center)
			 to (48.center)
			 to (53.center)
			 to (52.center)
			 to (51.center)
			 to (54.center)
			 to (55.center)
			 to (56.center)
			 to (10.center)
			 to (57.center)
			 to (58.center)
			 to (62.center)
			 to (84.center)
			 to [bend left=360] (83.center)
			 to (82.center)
			 to (5.center)
			 to [bend left=15, looseness=0.50] (50.center)
			 to [bend left=15, looseness=0.75] (37.center)
			 to (36.center)
			 to (6.center)
			 to [bend left=10] (30.center)
			 to [bend left=15, looseness=0.75] (29.center)
			 to [bend left=15, looseness=0.75] (18.center)
			 to (7.center)
			 to (14.center)
			 to (12.center)
			 to (8.center)
			 to (15.center)
			 to (16.center)
			 to (19.center)
			 to (20.center)
			 to (22.center)
			 to (23.center)
			 to (21.center)
			 to (26.center)
			 to (33.center)
			 to (34.center)
			 to (35.center)
			 to (11.center)
			 to [bend right=10] (38.center)
			 to (39.center)
			 to (40.center)
			 to cycle;
		\draw [style=new edge style 6] (2.center)
			 to [bend left] (3.center)
			 to (1.center)
			 to [bend left] cycle;
		\draw [style=new edge style 4] (15.center) to (8.center);
		\draw [style=new edge style 5] (18.center) to (27.center);
		\draw [style=new edge style 5] (28.center) to (29.center);
		\draw [style=new edge style 4] (7.center) to (12.center);
		\draw [style=new edge style 4] (8.center) to (17.center);
		\draw [style=new edge style 4, bend left=15] (17.center) to (16.center);
		\draw [style=new edge style 4, bend right=15] (15.center) to (17.center);
		\draw [style=new edge style 4] (17.center) to (7.center);
		\draw [style=new edge style 4, bend left, looseness=1.25] (20.center) to (24.center);
		\draw [style=new edge style 5] (28.center)
			 to (24.center)
			 to [bend left] (25.center)
			 to cycle;
		\draw [style=new edge style 4, bend left] (25.center) to (21.center);
		\draw [style=new edge style 4] (22.center) to (24.center);
		\draw [style=new edge style 4] (23.center) to (25.center);
		\draw [style=new edge style 4] (31.center) to (33.center);
		\draw [style=new edge style 4] (32.center) to (34.center);
		\draw [style=new edge style 4] (30.center) to (35.center);
		\draw [style=new edge style 4] (6.center) to (11.center);
		\draw [style=new edge style 4] (36.center) to (11.center);
		\draw [style=new edge style 4, in=60, out=-210, looseness=0.75] (48.center) to (49.center);
		\draw [style=new edge style 4] (52.center) to (50.center);
		\draw [style=new edge style 4] (53.center) to (50.center);
		\draw [style=new edge style 4] (10.center) to (61.center);
		\draw [style=new edge style 4] (56.center) to (61.center);
		\draw [style=new edge style 4] (57.center) to (61.center);
		\draw [style=new edge style 4, bend right=15] (62.center) to (60.center);
		\draw [style=new edge style 4] (59.center) to (60.center);
		\draw [style=new edge style 5] (61.center)
			 to (64.center)
			 to [bend right=15] (59.center)
			 to [bend right=15] (63.center)
			 to cycle;
		\draw [style=new edge style 4, bend left=15] (66.center) to (65.center);
		\draw [style=new edge style 5] (50.center) to (67.center);
		\draw [style=new edge style 4, bend left=15] (65.center) to (84.center);
		\draw [style=new edge style 4, bend right=15] (60.center) to (66.center);
		\draw [style=new edge style 4, bend right=15] (66.center) to (54.center);
		\draw [style=new edge style 4] (67.center) to (51.center);
		\draw [style=new edge style 4] (37.center) to (49.center);
		\draw [style=new edge style 4, bend right=15, looseness=0.75] (49.center) to (38.center);
		\draw [style=new edge style 4] (30.center) to (32.center);
		\draw [style=new edge style 4] (32.center) to (31.center);
		\draw [style=new edge style 4] (31.center) to (26.center);
		\draw [style=new edge style 4, bend left=15] (28.center) to (26.center);
		\draw [style=new edge style 4] (27.center) to (19.center);
		\draw [style=new edge style 5] (58.center) to (64.center);
		\draw [style=new edge style 5] (55.center) to (63.center);
		\draw [style=new edge style 4] (43.center) to (44.center);
		\draw [style=new edge style 4] (42.center) to (44.center);
		\draw [style=new edge style 4] (41.center) to (44.center);
		\draw [style=new edge style 4, bend right, looseness=0.75] (47.center) to (46.center);
		\draw [style=new edge style 4, bend right=15, looseness=1.25] (46.center) to (45.center);
		\draw [style=new edge style 4, bend right=15] (45.center) to (39.center);
		\draw [style=new edge style 4, bend right=15] (45.center) to (40.center);
		\draw (14.center) to (3.center);
		\draw (1.center) to (82.center);
		\draw [style=new edge style 5, bend left=360] (84.center) to (83.center);
		\draw [style=new edge style 5] (83.center) to (82.center);
		\draw [style=new edge style 5] (50.center) to (67.center);
		\draw [style=new edge style 3] (12.center)
			 to [bend left, looseness=0.75] (98.center)
			 to [bend left=15, looseness=0.75] (99.center)
			 to [bend left, looseness=0.25] (100.center)
			 to [bend left=15, looseness=0.50] (101.center)
			 to [bend left=15, looseness=0.50] (102.center)
			 to [bend left=15, looseness=0.50] (103.center)
			 to [bend left=15, looseness=0.50] (97.center)
			 to [bend left=15, looseness=0.50] (96.center)
			 to [bend left=15, looseness=0.50] (94.center)
			 to [bend left=15, looseness=0.50] (92.center)
			 to [bend left=15, looseness=0.50] (90.center)
			 to [bend left, looseness=0.25] (87.center)
			 to [bend left=15, looseness=0.75] (86.center)
			 to [bend left, looseness=0.75] (62.center);
	\end{pgfonlayer}
\end{tikzpicture}}
    \end{subfloat}
    
    \caption{Contractible squares. The solid lines indicate paths, whereas the dotted line shows the face's boundary on which the contractible square in Figures (c), (e), and (f) is pinched. The zone represented in blue is $\mathcal{B}(C)$ (defined, e.g., in \cref{def:contractible_square}), and the zone represented in pink is $\mathcal{I}(C)$ (defined in \cref{def:good_bad_contractible_square}).}
    \label{fig:contractible_square}
\end{figure}

    Let $k \geq 1$ and let $C_0, ..., C_k$ be $\Pi_H$-contractible nested cycles of $H$. We say that $C_0, ..., C_k$ are \textit{$\Pi_H$-well-nested} if either:
    
    \begin{itemize}
        \item for every $0 \leq i \leq k-1$, $C_i$ is $\Pi_H$-freely-well-nested in $C_{i+1}$, or 
        \item there is a piece $p$ such that, for every $0 \leq i \leq k-1$, $C_i$ is $\Pi_H$-well-nested in $C_{i+1}$ pinched on $p$, or
        \item there are two pieces $p$ and $p'$ such that, for every $0 \leq i \leq k-1$, $C_i$ is $\Pi_H$-well-nested in $C_{i+1}$ pinched on $p$ and $p'$.
    \end{itemize}

    The six figures of \cref{fig:contractible_square} show cycles $C, C', C''$ that are respectively freely-well-nested, well-nested pinched on a vertex, well-nested pinched on a face, well-nested pinched on two vertices, well-nested pinched on two faces, and well-nested pinched on a vertex and a face in the reversed order.
\end{defi}

\begin{defi}[Closest]
    Let $H$ be a graph with embedding $\Pi_H$. Let $C$ and $C'$ be two $\Pi_H$-contractible cycles with $C'$ nested in $C$ of $H$. Let $\mathcal{P}$ be a property on the cycles of $(H,\Pi_H)$. 
    
    Then, we say that $C'$ is the cycle \textit{closest} to $C$ that has property $\mathcal{P}$ if, for every cycle $C''$ that is $\Pi_H$-contractible nested in $C$ and has property $\mathcal{P}$, we have $C' \subseteq \text{Int}(C \cup C'', \Pi_H)$.
\end{defi}

\begin{defi}[Cycles in this order]
    Let $H$ be a graph with embedding $\Pi_H$. 
    
    Let $k \geq 1$ and let $C_0, ..., C_k$ be $\Pi_H$-noncontractible homotopic cycles of $H$.
    We say that $C_0, ..., C_k$ are \textit{in this order} in $\Pi_H$ if for $0 \leq i \leq k-1$ and $0 \leq j \leq k$, $C_j \cap \text{int}(C_i \cup C_{i+1}, \Pi_H) = \emptyset$.

    For instance, in \cref{fig:at_most_10_homotopic_cycles}(b), the cycles $C_1, C_2, C_3$ are homotopic in this order.
\end{defi}

\begin{defi}[Contractible square]
    \label{def:contractible_square}
    Let $C, C', C''$ be $\Pi$-contractible cycles in $G$, well nested in the reverse order. Let $\mathcal{B}(C)$ be the set of faces in $\mathcal{F}(C \cup C', \Pi)$ that touch $C$ (we call this set the boundary of $C$). 
    
    We say that $(C, C', C'')$ is a \textit{contractible square} with respect to $C$ if $C'$ is the cycle closest to $C$ so that $\mathcal{B}(C) \subseteq \text{Int}(C \cup C',\Pi)$.

    We say that a contractible square $(C, C', C'')$ is respectively \textit{full}, \textit{pinched on $p$} and \textit{pinched on $p$ and $p'$} if $C, C', C''$ are respectively $\Pi$-fully-well-nested, $\Pi$-well-nested pinched on $p$ and $\Pi$-well-nested pinched on $p$ and $p'$.

    The six figures of \cref{fig:contractible_square} depict contractible squares $(C, C', C'')$ that are respectively full, pinched on a vertex, pinched on a face, pinched on two vertices, pinched on two faces, and pinched on a vertex and a face.
\end{defi}

\begin{defi}[Good/bad contractible square]
    \label{def:good_bad_contractible_square}
    Let $e \in \text{int}(C'', \Pi)$.
    Let $(C, C', C'')$ be a contractible square with respect to $C$, let $\mathcal{I}(C)$ be the set of $\Pi$-faces in $\text{Int}(C'', \Pi)$ and let $\mathcal{I}_{\rm N}(C)$ be a maximal size subset of $\Pi$-faces of $\mathcal{I}(C)$ that are almost-disjoint and that are not $\Pi_e$-faces. Let $\mathcal{B}(C)$ be the set of $\Pi$-faces in $\text{Int}(C \cup C', \Pi)$ that touch $C$.
    We say that $(C, C', C'')$ is a \textit{bad square with respect to $e$} if $\mathcal{I}_{\rm N}(C)$ contains more than $18 \times (42 |\mathcal{B}_{\rm N}(C)| - 3)$ $\Pi$-faces with $\mathcal{B}_{\rm N}(C)$ being the set of $\Pi$-faces in $\mathcal{B}(C)$ that are not $\Pi_e$-faces. 
    
    Otherwise, we say that this square is \textit{good with respect to $e$}.
\end{defi}

\begin{defi}[Edge-sharing components]
    Let $H$ be a nonplanar graph and let $\Pi_H$ be an embedding of $H$ in a surface. Let $\mathcal{F}$ be a set of faces of $(H, \Pi_H)$. We say that two faces of $(H, \Pi_H)$ are edge-sharing if they share at least one edge. Let the edge-sharing components of $\mathcal{F}$ be its edge-sharing equivalence classes. 

    Let $\mathcal{C}$ be an edge-sharing component of $\mathcal{F}$. Suppose it induces a $\Pi_H$-contractible subgraph $H'$ of $H$. Then, remark that, by \cref{Whitney_planar_graph_flipping}, the $\Pi_H$-embedding of $H'$ is the only planar embedding of $H'$ up to equivalence. In particular, if $H'$ is $\Pi_H'$-contractible for some other embedding $\Pi_H'$ of $H$ in a surface, the boundary of $H'$ in $\Pi_H'$ is the same as in $\Pi_H$.

    \vspace{1em}

    Remark that the subgraphs induced by the edge-sharing components of $\mathcal{F}$ are $2$-connected, but they do not always correspond to the $2$-connected blocks of the subgraph induced by $\mathcal{F}$.
\end{defi}

\begin{lem}
    \label{at_most_10_homotopic_cycles}
    Let $(C,C',C'')$ be a contractible square. Let $e \in \text{int}(C'', \Pi)$.

    Let $\mathcal{I}(C)$ be the set of faces in $\text{Int}(C'', \Pi)$ in $\Pi$ and let $\mathcal{I}_{\rm N}(C)$ be a maximal size subset of almost disjoint $\Pi$-faces from $\mathcal{I}(C)$ that are not $\Pi_e$-faces.
    
    Suppose that embedding $\Pi_e$ was chosen to minimize the size of $\mathcal{I}_{\rm N}(C)$. Let $F_1, ..., F_m \in \mathcal{I}_{\rm N}(C)$ be a maximal size set of almost disjoint $\Pi$-faces that are $\Pi_e$-noncontractible homotopic, then $m \leq 10$.
\end{lem}

\begin{proof}
    See \cref{fig:at_most_10_homotopic_cycles} for illustrations for the proof.
    
    Suppose that there are at least $11$ almost disjoint $\Pi$-faces in $\mathcal{I}_{\rm N}(C)$ that are $\Pi_e$-noncontractible homotopic. Let $\mathcal{C}$ be the set of these faces (that are cycles by \cref{2_connected_faces_cycles}).

    If the square is pinched on two vertices $v$ and $v'$, let's select cycles from $\mathcal{C}$ so that one of $v, v'$ does not intersect any of the selected cycles. First, remark that there is at most one cycle in $\mathcal{C}$ that intersects both $v$ and $v'$: indeed, the faces are chosen to be almost-disjoint.
    Finally, there are at most $\frac{11 - 1}{2} = 5$ cycles in $\mathcal{C}$ so that one of $v,v'$ does not intersect any of the selected cycles.

    If the square is pinched on two vertices $v$ and $v'$, let $C_1, C_2, C_3, C_4, C_5$ be the cycles from $\mathcal{C}$ so that one of $v,v'$ does not intersect any of these cycles. If the square is not pinched on two vertices, let $C_1, C_2, C_3, C_4, C_5$ be distinct cycles from $\mathcal{C}$ (arbitrarily chosen). Without loss of generality, we have that $C_1, C_2, C_3, C_4, C_5$ are $\Pi_e$-homotopic in this order.

    Remark that there are at least two of the cycles $C_1, C_3, C_5$ that have the same relative orientation in $\Pi$ and $\Pi_e$, suppose without loss of generality that it is $C_1$ and $C_3$.
    Let's show that it is possible to modify $\Pi_e$ so that $C_2$ becomes contractible. 

    First, note that:
    \begin{itemize}
        \item \underline{If the square is freely well nested:} no vertex of $C$ is in $\text{Int}(C_1 \cup C_3, \Pi_e)$, otherwise the whole subgraph $\text{Ext}(C,\Pi)$ would be in $\text{Int}(C_1 \cup C_3, \Pi_e)$ which is impossible because $\text{Ext}(C,\Pi)$ cannot not be $\Pi_e$-contractible. Therefore, $\text{Int}(C_1 \cup C_3, \Pi_e) \subseteq \text{Int}(C,\Pi)$.
        \item \underline{If the square is pinched on a piece $p$:} the same reasoning for $C - p$ shows that no vertex of $C - p$ is in $\text{Int}(C_1 \cup C_3, \Pi_e)$. If $p$ is a face, it implies that $\text{Int}(C_1 \cup C_3, \Pi_e) \subseteq \text{Int}(C,\Pi)$ as there is no edge in $\text{Ext}(C, \Pi)$ with an endpoint on $p$. If $p$ is a vertex, then a bridge from $\text{ext}(C, \Pi)$ in $\text{Int}(C_1 \cup C_3, \Pi_e)$, if it exists, would be attached solely to $v$, which contradicts the fact that $G$ is $2$-connected and therefore has no cutvertex.
        \item \underline{If the square is pinched on two pieces $p$ and $p'$:} the same reasoning for $C - (p \cup p')$ shows that no vertex of $C - (p \cup p')$ is in $\text{Int}(C_1 \cup C_3, \Pi_e)$. If both $p$ and $p'$ are faces, it implies that $\text{Int}(C_1 \cup C_3, \Pi_e) \subseteq \text{Int}(C,\Pi)$ as there is no edge in $\text{Ext}(C, \Pi)$ with an endpoint on $p$ or $p'$. If $p$ is a vertex and $p'$ is a face, then a bridge from $\text{ext}(C, \Pi)$ in $\text{Int}(C_1 \cup C_3, \Pi_e)$, if it exists, would be attached solely to $p$, which contradicts the fact that $G$ is $2$-connected and therefore has no cutvertex. If both $p$ and $p'$ are vertices, then, as we showed before, $C_1 \cup C_3$ does not intersect both $p$ and $p'$, and we can conclude with the same reasoning as when the square is pinched on one vertex.
    \end{itemize} 

    Therefore, $\text{Int}(C_1 \cup C_3, \Pi_e) \subseteq \text{Int}(C,\Pi)$ and $\Pi(\text{Int}(C_1 \cup C_3, \Pi_e))$ is an embedding in a disk in $S$ in which $C_1$, $C_2$ and $C_3$ are $\Pi$-facial walks. 

    Finally, by \cref{cylinder_reembedding}, we can modify the embedding $\Pi_e$ so that $C_2$ is now a $\Pi_e$-contractible cycle. This contradicts the minimality of $\mathcal{I}_N(C)$.
\end{proof}

\begin{lem}
    \label{genus_tildeG}
    Let $e \in G$ and let $H_1, H_2$ be subgraphs of $G_e$ such that either:
    \begin{itemize}
        \item $H_1$ and $H_2$ are disjoint, or
        \item $H_1$ and $H_2$ intersect in a path $P$ (possibly reduced to a vertex), and the internal vertices of $P$ have degree $2$ in $H_1$, or
        \item $H_1$ and $H_2$ intersect in two paths $P$ and $P'$ (possibly reduced vertices) and the internal vertices of $P$ and $P'$ have degree $2$ in $H_1$.
    \end{itemize}

    Then, \[g(\Pi_e(H_1)) \leq g - g(\Pi_e(H_2))+6\]
\end{lem}

\begin{proof}
    Let $v = |V(H_1) \cap V(H_2)|$ and $e = |E(H_1) \cap E(H_2)|$. Let $f_{H_1}$ and $f_{H_2}$ be respectively the number of faces in an embedding of $H_1$ of genus $g(\Pi_e(H_1))$ and in an embedding of $H_2$ of genus $g(\Pi_e(H_2))$.

    \begin{itemize}
        \item \underline{If $H_1$ and $H_2$ are disjoint:} Then, $g(\Pi_e(G_e)) \geq g(\Pi_e(H_1)) + g(\Pi_e(H_2))$.

        \item \underline{If $H_1$ and $H_2$ intersect in a path:} Remark that $e = v-1$ and that, when trying to merge embeddings of $H_1$ and $H_2$ into an embedding of $H_1 \cup H_2$, at least $1$ face from each embedding is modified and at most $4$ new faces appear.

        We have:

        \[\chi(H_1) = V(H_1) - E(H_1) + f_{H_1} \text{\quad and \quad} \chi(H_2) = V(H_2) - E(H_2) + f_{H_2}\]
        \begin{align*}
            \chi(G_e) & \leq (V(H_1)+V(H_2)-v) - (E(H_1)+E(H_2)-e) + (f_{H_1}+f_{H_2}-2+4) \\
            & \leq \chi(H_1) + \chi(H_2) - v + v - 1 - 2 + 4 \\
            & = \chi(H_1) + \chi(H_2)+1    
        \end{align*}    

        Therefore, 
        \begin{align*}
            g(G_e) & = 2 - \chi(G_e) \\
            & \geq 2 - \chi(H_1) + 2 -\chi(H_2) - 2 - 1 \\
            & = g(H_1) + g(H_2) - 3
        \end{align*}
            
        \item \underline{If $H_1$ and $H_2$ intersect in two paths:}
        Remark that $e = v-2$ and that, when trying to merge embeddings of $H_1$ and $H_2$ into an embedding of $H_1 \cup H_2$, at least $1$ face from each embedding is modified and at most $8$ new faces appear.

        We have:
            
        \[\chi(H_1) = V(H_1) - E(H_1) + f_{H_1} \text{\quad and \quad} \chi(H_2) = V(H_2) - E(H_2) + f_{H_2}\]

        \begin{align*}
            \chi(G_e) & \leq (V(H_1)+V(H_2)-v) - (E(H_1)+E(H_2)-e) + (f_{H_1}+f_{H_2}-2+8) \\
            & \leq \chi(H_1) + \chi(H_2) - v + v - 2 - 2 + 8 \\
            & = \chi(H_1) + \chi(H_2)+4
        \end{align*}

        Therefore, 
        \begin{align*}
            g(G_e) & = 2 - \chi(G_e) \\ 
            & \geq 2 - \chi(H_1) + 2 -\chi(H_2) - 2 - 4 \\
            & = g(H_1) + g(H_2) - 6
        \end{align*} 
    \end{itemize}

    In every case, we get $g(\Pi_e(H_1)) \leq g(\Pi_e(G_e)) - (g(\Pi_e(H_2))-6)$. Finally, \[g(\Pi_e(H_1)) \leq g - g(\Pi_e(H_2))+6\]
\end{proof}

\begin{figure}[h!]
    \centering
    \begin{subfloat}[The cycles $C, C_1, C_2$ and $C_3$ in $\Pi$]{\begin{tikzpicture}[scale=1.35]
	\begin{pgfonlayer}{nodelayer}
		\node [style=none] (0) at (0, 4) {};
		\node [style=none] (1) at (-4, 0) {};
		\node [style=none] (2) at (0, -4) {};
		\node [style=none] (3) at (4, 0) {};
		\node [style=none] (4) at (-2, 0.5) {};
		\node [style=none] (5) at (-1, 0.75) {};
		\node [style=none] (6) at (-2, -0.75) {};
		\node [style=none] (7) at (-1, -1.25) {};
		\node [style=none] (8) at (-0.5, -0.5) {};
		\node [style=none] (9) at (0.5, 0) {};
		\node [style=none] (10) at (1.25, 0.25) {};
		\node [style=none] (11) at (-0.25, 1.25) {};
		\node [style=none] (12) at (0.75, 1.75) {};
		\node [style=none] (13) at (2, -0.25) {};
		\node [style=none] (14) at (2.75, -1) {};
		\node [style=none] (15) at (2.25, 1.25) {};
		\node [style=none] (16) at (3.25, 0.5) {};
		\node [style=none] (17) at (-0.25, 1.25) {};
		\node [style=none] (18) at (-2.7, 1) {};
		\node [style=none] (19) at (-2.975, -1.125) {};
		\node [style=none] (20) at (-1, -2.1) {};
		\node [style=none] (21) at (-2.175, -1.75) {};
		\node [style=none] (22) at (0.025, -1.275) {};
		\node [style=none] (23) at (0.5, -0.9) {};
		\node [style=none] (24) at (1, -0.55) {};
		\node [style=none] (25) at (-1.6, 1.275) {};
		\node [style=none] (26) at (-1, 1.575) {};
		\node [style=none] (27) at (-0.675, 1.875) {};
		\node [style=none] (28) at (-0.15, 2.025) {};
		\node [style=none] (29) at (1.725, 2.075) {};
		\node [style=none] (30) at (2.25, 2.125) {};
		\node [style=none] (31) at (2.825, 1.9) {};
		\node [style=none] (32) at (0.9, 2.45) {};
		\node [style=none] (33) at (3.675, 1.075) {};
		\node [style=none] (34) at (2.175, -1.55) {};
		\node [style=none] (35) at (2.75, -1.85) {};
		\node [style=none] (36) at (3.25, -1.675) {};
		\node [style=none] (37) at (3.575, -1) {};
		\node [style=none] (38) at (1.725, -1.15) {};
		\node [style=none] (39) at (1.225, -0.8) {};
		\node [style=none, label={$C_1$}] (40) at (-1.35, -0.75) {};
		\node [style=none, label={$C_2$}] (41) at (0.55, 0.25) {};
		\node [style=none, label={$C_3$}] (42) at (2.575, -0.4) {};
		\node [style=none, label={above:$C$}] (43) at (2.725, 2.975) {};
	\end{pgfonlayer}
	\begin{pgfonlayer}{edgelayer}
		\draw [bend left=45] (1.center) to (0.center);
		\draw [bend left=45] (0.center) to (3.center);
		\draw [bend left=45] (3.center) to (2.center);
		\draw [bend left=45] (2.center) to (1.center);
		\draw [style=new edge style 4] (8.center)
			 to (5.center)
			 to (4.center)
			 to (6.center)
			 to (7.center)
			 to cycle;
		\draw (5.center) to (11.center);
		\draw [style=new edge style 4] (11.center)
			 to (9.center)
			 to (10.center)
			 to (12.center)
			 to (17.center);
		\draw (9.center) to (8.center);
		\draw (12.center) to (15.center);
		\draw [style=new edge style 4] (13.center)
			 to (15.center)
			 to (16.center)
			 to (14.center)
			 to cycle;
		\draw (13.center) to (10.center);
		\draw [style=new edge style 3] (4.center) to (18.center);
		\draw [style=new edge style 3] (6.center) to (19.center);
		\draw [style=new edge style 3] (6.center) to (21.center);
		\draw [style=new edge style 3] (7.center) to (20.center);
		\draw [style=new edge style 3] (8.center) to (22.center);
		\draw [style=new edge style 3] (5.center) to (25.center);
		\draw [style=new edge style 3] (5.center) to (26.center);
		\draw [style=new edge style 3] (17.center) to (27.center);
		\draw [style=new edge style 3] (17.center) to (28.center);
		\draw [style=new edge style 3] (12.center) to (32.center);
		\draw [style=new edge style 3] (15.center) to (29.center);
		\draw [style=new edge style 3] (15.center) to (30.center);
		\draw [style=new edge style 3] (15.center) to (31.center);
		\draw [style=new edge style 3] (16.center) to (33.center);
		\draw [style=new edge style 3] (14.center) to (34.center);
		\draw [style=new edge style 3] (14.center) to (35.center);
		\draw [style=new edge style 3] (14.center) to (36.center);
		\draw [style=new edge style 3] (14.center) to (37.center);
		\draw [style=new edge style 3] (13.center) to (38.center);
		\draw [style=new edge style 3] (13.center) to (39.center);
		\draw [style=new edge style 3] (10.center) to (24.center);
		\draw [style=new edge style 3] (9.center) to (23.center);
	\end{pgfonlayer}
\end{tikzpicture}}
    \end{subfloat}
    \hspace{0.75cm}
    \begin{subfloat}[The cycles $C_1, C_2$ and $C_3$ in $\Pi_e$]{\begin{tikzpicture}
	\begin{pgfonlayer}{nodelayer}
		\node [style=none] (1) at (4.55, 4) {};
		\node [style=none] (3) at (4.575, -4) {};
		\node [style=none] (4) at (0.35, 0) {};
		\node [style=none] (5) at (1.6, 4) {};
		\node [style=none] (6) at (2.85, 0) {};
		\node [style=none] (7) at (1.6, -4) {};
		\node [style=none] (8) at (-3.15, 0) {};
		\node [style=none] (9) at (-1.9, 4) {};
		\node [style=none] (10) at (-0.65, 0) {};
		\node [style=none] (11) at (-1.9, -4) {};
		\node [style=none] (13) at (-5.15, 4) {};
		\node [style=none] (14) at (-3.9, 0) {};
		\node [style=none] (15) at (-5.15, -4) {};
		\node [style=none] (17) at (-8, 4) {};
		\node [style=none] (19) at (-7.975, -4) {};
		\node [style=none] (28) at (-4.4, 3) {};
		\node [style=none] (30) at (-4.15, 1.75) {};
		\node [style=none] (40) at (-4.35, -2.75) {};
		\node [style=none, label={right:$C_2$}] (52) at (-1.15, 2.75) {};
		\node [style=none, label={right:$C_3$}] (53) at (2.35, 2.75) {};
		\node [style=none, label={right:$C_1$}] (55) at (-4.4, 2.75) {};
		\node [style=none] (56) at (-0.725, -1) {};
		\node [style=none] (57) at (-0.725, 1.075) {};
		\node [style=none] (58) at (-2.875, 2.15) {};
		\node [style=none] (59) at (-2.725, -2.625) {};
		\node [style=none] (62) at (-0.25, 4) {};
		\node [style=none] (63) at (2.6, 1.75) {};
		\node [style=none] (64) at (2.675, -1.775) {};
		\node [style=none] (65) at (-0.25, -4) {};
	\end{pgfonlayer}
	\begin{pgfonlayer}{edgelayer}
		\draw [style=new edge style 6] (9.center) to (5.center);
		\draw [in=90, out=-15, looseness=0.50] (5.center) to (6.center);
		\draw [in=0, out=-90, looseness=0.50] (6.center) to (7.center);
		\draw [style=new edge style 6] (7.center) to (11.center);
		\draw [in=-90, out=0, looseness=0.50] (11.center) to (10.center);
		\draw [in=-15, out=90, looseness=0.50] (10.center) to (9.center);
		\draw [style=new edge style 3, in=165, out=-90, looseness=0.50] (4.center) to (7.center);
		\draw [style=new edge style 3, in=180, out=90, looseness=0.50] (4.center) to (5.center);
		\draw [style=new edge style 3, in=165, out=-90, looseness=0.50] (8.center) to (11.center);
		\draw [style=new edge style 3, in=180, out=90, looseness=0.50] (8.center) to (9.center);
		\draw [style=new edge style 4] (17.center) to (13.center);
		\draw [in=90, out=-15, looseness=0.50] (13.center) to (14.center);
		\draw [in=0, out=-90, looseness=0.50] (14.center) to (15.center);
		\draw [style=new edge style 4] (15.center) to (19.center);
		\draw (13.center) to (9.center);
		\draw (5.center) to (1.center);
		\draw (7.center) to (3.center);
		\draw (15.center) to (11.center);
		\draw [style=new edge style 3, bend left=75, looseness=0.50] (15.center) to (13.center);
		\draw (40.center) to (56.center);
		\draw (30.center) to (57.center);
		\draw [style=new edge style 3, in=165, out=45, looseness=0.50] (58.center) to (62.center);
		\draw [in=135, out=0, looseness=0.50] (62.center) to (63.center);
		\draw [style=new edge style 3, in=180, out=-45] (59.center) to (65.center);
		\draw [in=-135, out=0, looseness=0.75] (65.center) to (64.center);
	\end{pgfonlayer}
\end{tikzpicture}}
    \end{subfloat}
    \vspace{0.8cm}
    
    \begin{subfloat}[The embedding constructed in \cref{at_most_10_homotopic_cycles}]{\begin{tikzpicture}
	\begin{pgfonlayer}{nodelayer}
		\node [style=none] (0) at (6.55, 4) {};
		\node [style=none] (1) at (6.575, -4) {};
		\node [style=none] (2) at (2.35, 0) {};
		\node [style=none] (3) at (3.6, 4) {};
		\node [style=none] (4) at (4.85, 0) {};
		\node [style=none] (5) at (3.6, -4) {};
		\node [style=none] (7) at (0.1, 4) {};
		\node [style=none] (9) at (0.1, -4) {};
		\node [style=none] (10) at (-3.15, 4) {};
		\node [style=none] (11) at (-1.9, 0) {};
		\node [style=none] (12) at (-3.15, -4) {};
		\node [style=none] (13) at (-6, 4) {};
		\node [style=none] (14) at (-5.975, -4) {};
		\node [style=none] (15) at (-2.4, 3) {};
		\node [style=none] (16) at (-2.15, 1.75) {};
		\node [style=none] (17) at (-2.35, -2.75) {};
		\node [style=none, label={left:$C_2$}] (18) at (1.55, -0.025) {};
		\node [style=none, label={right:$C_3$}] (19) at (4.35, 2.75) {};
		\node [style=none, label={right:$C_1$}] (20) at (-2.4, 2.75) {};
		\node [style=none] (21) at (-0.325, -2) {};
		\node [style=none] (22) at (-0.5, 2) {};
		\node [style=none] (23) at (1.975, 2.075) {};
		\node [style=none] (24) at (2.175, -2) {};
		\node [style=none] (25) at (4.675, -1.6) {};
		\node [style=none] (26) at (4.625, 1.75) {};
	\end{pgfonlayer}
	\begin{pgfonlayer}{edgelayer}
		\draw [style=new edge style 6] (7.center) to (3.center);
		\draw [in=90, out=-15, looseness=0.50] (3.center) to (4.center);
		\draw [in=0, out=-90, looseness=0.50] (4.center) to (5.center);
		\draw [style=new edge style 6] (5.center) to (9.center);
		\draw [style=new edge style 3, in=165, out=-90, looseness=0.50] (2.center) to (5.center);
		\draw [style=new edge style 3, in=180, out=90, looseness=0.50] (2.center) to (3.center);
		\draw [style=new edge style 4] (13.center) to (10.center);
		\draw [in=90, out=-15, looseness=0.50] (10.center) to (11.center);
		\draw [in=0, out=-90, looseness=0.50] (11.center) to (12.center);
		\draw [style=new edge style 4] (12.center) to (14.center);
		\draw (10.center) to (7.center);
		\draw (3.center) to (0.center);
		\draw (5.center) to (1.center);
		\draw (12.center) to (9.center);
		\draw [style=new edge style 3, bend left=75, looseness=0.50] (12.center) to (10.center);
		\draw (17.center) to (21.center);
		\draw (16.center) to (22.center);
		\draw (22.center) to (21.center);
		\draw (22.center) to (23.center);
		\draw (23.center) to (24.center);
		\draw (24.center) to (21.center);
		\draw (24.center) to (25.center);
		\draw (23.center) to (26.center);
	\end{pgfonlayer}
\end{tikzpicture}}
    \end{subfloat} 

    \caption{Illustration for the proof of \cref{at_most_10_homotopic_cycles}.}
    \label{fig:at_most_10_homotopic_cycles}
    \end{figure}
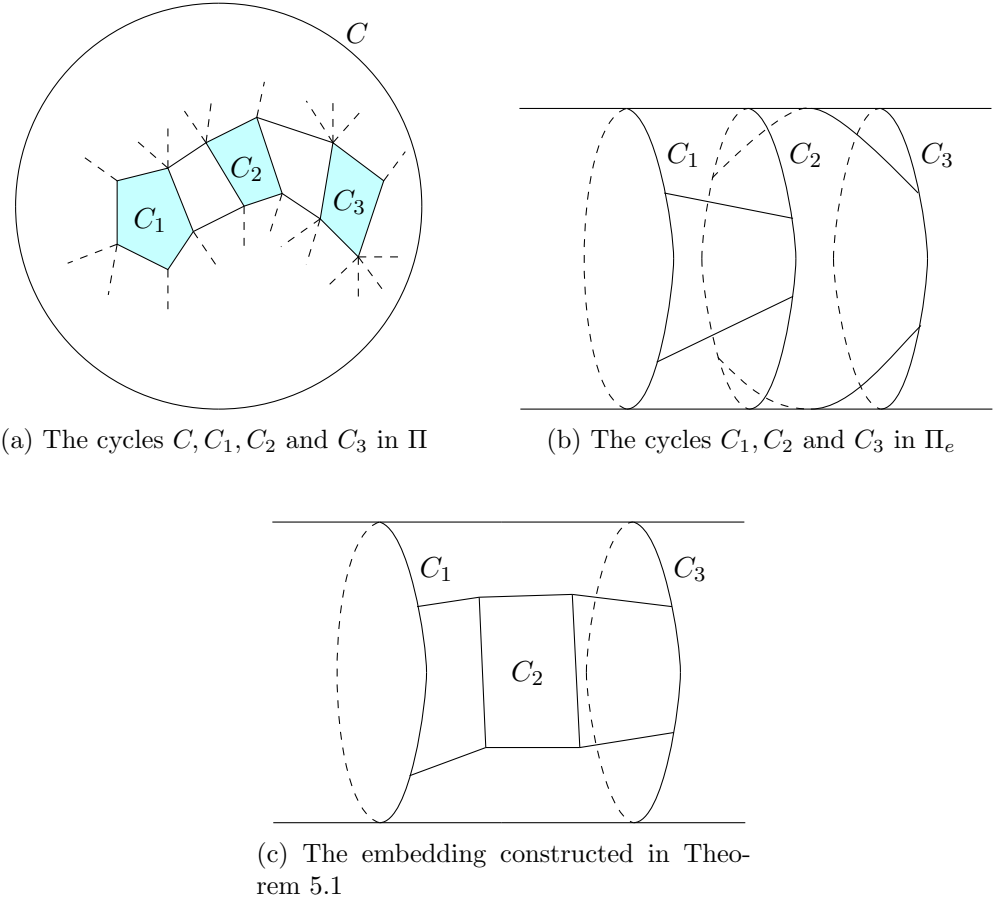

\begin{lem}
    \label{bad_square}
    Let $e \in \text{int}(C'', \Pi)$, $G$ contains no bad contractible square with respect to $e$.
\end{lem}

\begin{proof}
    Suppose by contradiction that $G$ contains a bad contractible square $(C, C', C'')$ with respect to $e$. Let $\mathcal{I}(C)$ be the set of faces in $\text{Int}(C'', \Pi)$ in $\Pi$ and let $\mathcal{B}(C)$ be the set of faces in $\text{Int}(C \cup C', \Pi)$ that touch $C$. Let $\mathcal{B}_\text{Con}(C)$ and $\mathcal{B}_{\rm N}(C)$ be the set of the faces in $\mathcal{B}(C)$ that are respectively $\Pi_e$-contractible and $\Pi_e$-noncontractible. Moreover, let $\mathcal{I}_{\rm N}(C)$ be a maximal size subset of almost disjoint $\Pi$-faces from $\mathcal{I}(C)$ that are not $\Pi_e$-faces.
    We suppose that $|\mathcal{I}_{\rm N}(C)| \geq 18 \times (42 |\mathcal{B}_{\rm N}(C)| - 3)$.
    
    First, let's show that there exists no bad contractible square with respect to $e$ with $|\mathcal{B}_{\rm N}(C)| = 0$.

    \begin{claim}
        $|\mathcal{B}_{\rm N}(C)| > 0$
    \end{claim}

    \begin{proof_claim}
        Suppose that $|\mathcal{B}_{\rm N}(C)| = 0$.

        Then, $\mathcal{B}(C)$ contains at most $2$ edge-sharing components and remark that the embedding of each of these is the same in $\Pi$ and $\Pi_e$.

        Let $\Tilde{G}$ be the graph obtained from $G_e$ by removing $H = \text{int}(C', \Pi)$. $\Tilde{G}$ has genus at most $g$ (because it is a subgraph of $G_e$).
        
        If $(C,C',C'')$ are freely well nested or well nested pinched on a piece $p$, then $\mathcal{B}(C)$ contains a unique edge-sharing component, and remark that its embeddings in $\Pi_e$ and $\Pi$ are identical. It is then quite easy to find an empty disk in $\Pi_e(\tilde{G})$ with the vertices of $C'$ on its boundary in this order (simply by following the walk on $C'$ in $\Pi_e(\tilde{G})$). 

        Now suppose that $(C,C',C'')$ are well nested pinched on two pieces. In that case, 
        the two edge-sharing components of $\mathcal{B}(C)$ have identical embeddings in $\Pi_e$ and $\Pi$. We only need to show that they are connected with the expected orientation. 

        Let $B_1$ and $B_2$ be the two edge-sharing components of $\mathcal{B}(C)$. If the square is pinched on two vertices $v$ and $v'$, let's define $r = v$ and $r' = v'$. If the square is pinched on a vertex $v$ and a face $f$, let $r = v$ and $r' = C' \cap f$. If the square is pinched on two faces $f$ and $f'$, let $r = C' \cap f$ and $r' = C' \cap f'$.

        We define $D_1$ and $D_2$ to be respectively $(B_1 \cup r \cup r') \cap C'$ and $(B_2 - (r \cup r'))  \cap C'$.
        If there is no path in $\text{Int}(C',\Pi)$ from $D_1$ to $D_2$, then it is enough to find two empty disks in $\Pi_e(\tilde{G})$ with respectively on their boundaries the vertices of $D_1$ and $D_2$ in this order, which we can easily do. Otherwise, let $P$ be a path from $D_1$ to $D_2$.
            
        If $P$ connects $B_1$ and $B_2$ with the expected orientation, then we can find an empty disk in $\Pi_e(\tilde{G})$ with the vertices of $C'$ on its boundary in this order. Otherwise, suppose that no path $P$ connects $B_1$ and $B_2$ with the expected orientation. In particular, if the square is pinched on a face $f$ (maybe together with another piece), the path $P_f$ between $B_1$ and $B_2$ that is a subwalk of $f$ does not connect them with the expected orientation. We deduce that $P$ and $P_f$ have the same signature.
        
        Now, let's look at the two faces $f_1$ and $f_2$ on the path $P$ in $\Pi_e(\tilde{G}\cup P)$. We can see that $f_1 = f_2$ because it is the case in $\Pi_e(\tilde{G})$. Therefore, changing the signature of $P$ does not increase the genus of the embedding; thus, by this manipulation, we return to the previous case.


        In each case, we find an embedding of $G$ in the surface $S'$ of genus $g$. 
        
        This is a contradiction.
    \end{proof_claim}

    \vspace{1em}
    
    We can now suppose that $|\mathcal{B}_{\rm N}(C)| > 0$.

    \vspace{1em}
    Remove $H = \text{int}(C', \Pi)$ from $G_e$ to obtain a graph $\Tilde{G}$. 

    \begin{claim}
        \label{claim:genus_tildeG}
        The embedding of $\Tilde{G}$ induced by $\Pi_e$ has genus at most $g - 8|\mathcal{B}_{\rm N}(C)|$.
    \end{claim}

    \begin{proof_claim}
        Remark that $\mathcal{I}_{\rm N}(C) \subseteq \mathcal{I}(C)$ and therefore $\Pi_e(\mathcal{I}(C))$ contains $18 \times (42 |\mathcal{B}_{\rm N}(C)|-3)$ almost disjoint $\Pi_e$-noncontractible cycles. Moreover, by \cref{at_most_10_homotopic_cycles}, there are at least $\frac{18}{10} \times (42 |\mathcal{B}_{\rm N}(C)|-3) \geq 42 |\mathcal{B}_{\rm N}(C)|-3$ of them that are pairwise $\Pi_e$-nonhomotopic. Finally, by \cref{homotopic_cycles_variant1}, $\Pi_e(\mathcal{I}(C))$ has genus at least $14|\mathcal{B}_{\rm N}(C)|$. 

        By \cref{genus_tildeG}, we have $g(\Pi_e(\tilde{G})) \leq g - g(\Pi_e(\mathcal{I}(C))+6$.
        
        As $\Pi_e(\mathcal{I}(C))$ has genus at least $14|\mathcal{B}_{\rm N}(C)|$, we finally conclude that $\Pi_e(\Tilde{G})$ has genus at most $g - 14|\mathcal{B}_{\rm N}(C)|+6 \leq g - 8|\mathcal{B}_{\rm N}(C)|$. 
    \end{proof_claim}

    \vspace{0.3cm}

    Now, we will construct a new embedding of $G$ from the embedding $\Pi_e(\Tilde{G})$. See \cref{fig:contractible_nested_square_end} for an illustration.

    \begin{figure}[h!]
    \centering
    \input{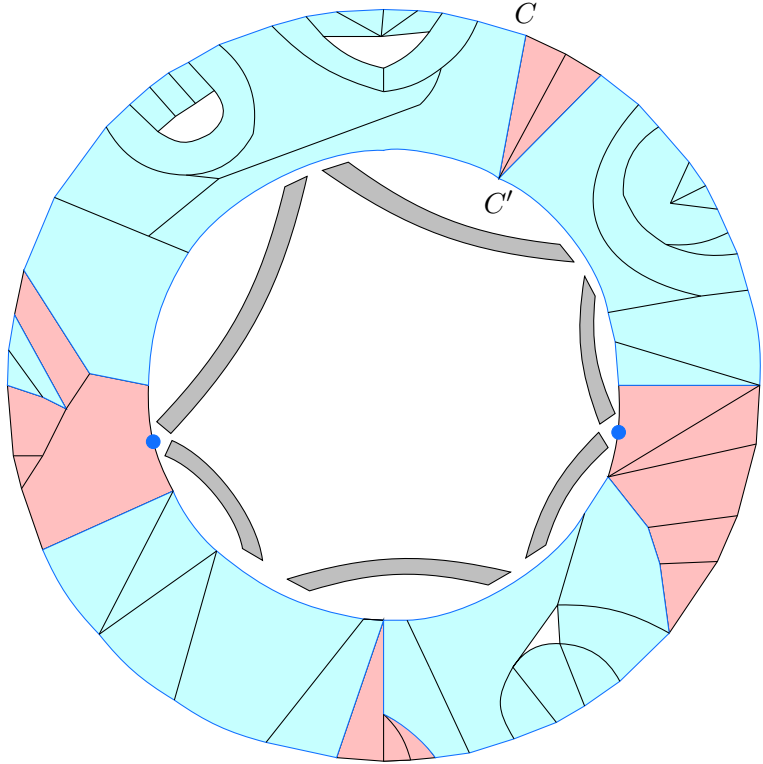}

    \caption{Illustration for the end of the proof of \cref{bad_square}. The vertices in $\tilde{V}(C')$ are depicted in blue. The edge-sharing components of $\mathcal{B}_{\rm Con}(C)$ are in light blue and have a blue outline. The bridges attached to only one edge-sharing component are in the white zones in the blue edge-sharing components. The faces from $\mathcal{B}(C)$ that are $\Pi_e$-noncontractible are depicted in pink. Finally, the added handles are depicted in grey. They are between consecutive edge-sharing components/vertices and can be handles or twisted handles.}
    \label{fig:contractible_nested_square_end}
    \end{figure}

    Let $\tilde{\mathcal{B}_{\rm N}}$ be the subset of faces of $\mathcal{B}_{\rm N}(C)$ that touch $C'$. 

    
    The edge-sharing components of $\mathcal{B}_{\rm Con}(C)$ that touch $C'$ and the faces in $\tilde{\mathcal{B}_{\rm N}}$ can be ordered with respect to $C'$.
    Then, we add a handle or twisted handle between each consecutive edge-sharing component of $\mathcal{B}_{\rm Con}(C)$/face of $\tilde{\mathcal{B}_{\rm N}}$ $x$ and $y$ so that, for an edge-sharing component, the part of its boundary that contains endpoints of edges from $H$ is reachable from the handle and has the expected orientation. 
    Moreover, we add at most one handle or twisted handle when the square is pinched into two pieces to make sure that the two edge-sharing components of $\mathcal{B}(C)$ are connected with the expected orientation.
    
    By \cref{claim:genus_tildeG}, the embedding $\Pi_e(\tilde{G})$ has genus at most $g - 6|\mathcal{B}_{\rm N}(C)|$. Remark that we add at most $2|\mathcal{B}_{\rm N}(C)| + 1 \leq 3 |\mathcal{B}_{\rm N}(C)|$ handles and twisted handles, and therefore creates a new embedding of genus at most $g$. Moreover, this new embedding contains an empty cylinder with $C'$ on one of its boundaries. Then, $H$ can be embedded in a planar way (for example, as in $\Pi$) in this cylinder. The graph $G$ can then be embedded in a surface of genus at most $g$, a contradiction.
\end{proof}

\begin{lem}
    \label{almost_disjoint_faces_between_two_cycles}
    Let $C$ and $C'$ be two $\Pi$-contractible nested cycles. Let $\mathcal{F}$ be a subset of the faces in $\text{Int}(C \cup C', \Pi)$ that intersect both $C$ and $C'$.

    There exists a subset of faces from $\mathcal{F}$ of size at least $\frac{|\mathcal{F}|}{6}$ that are almost disjoint.
\end{lem}

\begin{proof}
    Let $\mathcal{B}(\mathcal{F})$ be the partition of $\mathcal{F}$ into its edge-sharing components. 

    We will then partition each edge-sharing component $B \in \mathcal{B}(\mathcal{F})$ into units $\mathcal{U}(B)$.
    Let $B \in \mathcal{B}(\mathcal{F})$. Remark that the faces in $B$ are naturally (circularly) ordered inside $B$ so that two faces are consecutive in this order if they are edge-sharing. If the faces in $B$ are linearly ordered, take the first face $f_0$ in this order, or any face $f_0$ of $B$ if the faces in $B$ are circularly ordered. Let $P_0$ be the path from $C$ to $C'$ in $f_0$ that is the first we go through when following the order of the faces of $B$ from $f_0$ (if the order is circular, then $P_0$ is chosen so that we go through $P_0$ and then the other path from $C$ to $C'$ in $f_0$ consecutively in our traversal). 
    We will define the units of $B$ from $f_0$ using the following method: let $f_1$ be the first face vertex-disjoint from $P_0$ in order. Then, the first unit of $B$ contains $f_0$ and all the faces between $f_0$ (included) and $f_1$ (excluded) in order. Iterate this method with $f_1$ to find the next unit and so on until every face of $B$ belongs to a unit of $B$. See \cref{fig:square_lemma_between} for an illustration.

    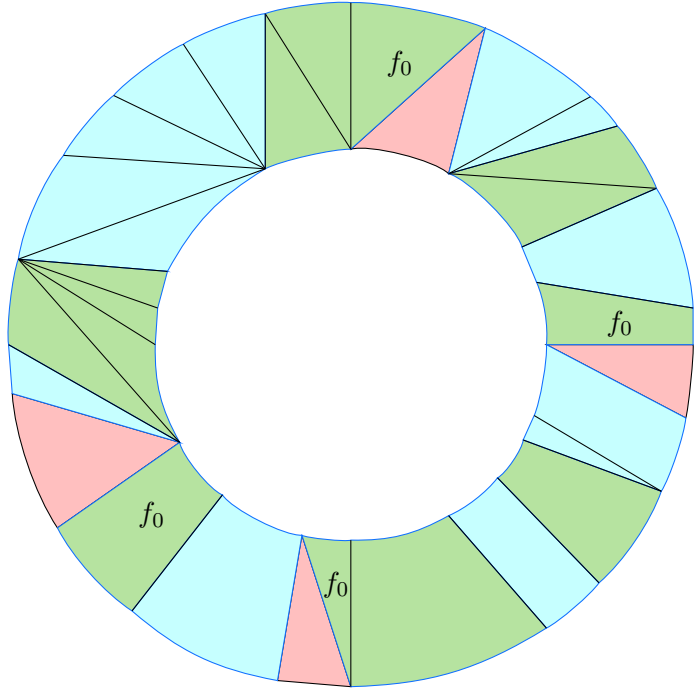
\begin{figure}[h!]
    \centering
    \begin{tikzpicture}[scale=1.3]
	\begin{pgfonlayer}{nodelayer}
		\node [style=none] (0) at (0, 4) {};
		\node [style=none] (23) at (2, 3.5) {};
		\node [style=none] (30) at (2.75, 6.475) {};
		\node [style=none] (2) at (4, 0) {};
		\node [style=none] (5) at (7, 0) {};
		\node [style=none] (32) at (6.85, -1.5) {};
		\node [style=none] (9) at (-1.5, -6.875) {};
		\node [style=none] (4) at (0, -7) {};
		\node [style=none] (8) at (-1, -3.9) {};
		\node [style=none] (1) at (-4, 0) {};
		\node [style=none] (6) at (0, 7) {};
		\node [style=none] (7) at (-7, 0) {};
		\node [style=none] (10) at (-4.475, -5.45) {};
		\node [style=none] (11) at (-2.625, -3.075) {};
		\node [style=none] (12) at (-6, -3.75) {};
		\node [style=none] (13) at (-3.5, -2) {};
		\node [style=none] (14) at (-6.925, -1) {};
		\node [style=none] (15) at (-3.95, 0.75) {};
		\node [style=none] (16) at (-3.75, 1.5) {};
		\node [style=none] (17) at (-1.75, 3.6) {};
		\node [style=none] (18) at (-5.875, 3.875) {};
		\node [style=none] (19) at (-6.8, 1.75) {};
		\node [style=none] (20) at (-4.85, 5.1) {};
		\node [style=none] (21) at (-3.425, 6.15) {};
		\node [style=none] (22) at (-1.75, 6.775) {};
		\node [style=none] (24) at (3.5, 2) {};
		\node [style=none] (25) at (3.8, 1.275) {};
		\node [style=none] (26) at (7, 0.75) {};
		\node [style=none] (27) at (6.25, 3.2) {};
		\node [style=none] (28) at (4.9, 5.075) {};
		\node [style=none] (29) at (5.45, 4.475) {};
		\node [style=none] (31) at (6.35, -3) {};
		\node [style=none] (33) at (5.075, -4.875) {};
		\node [style=none] (34) at (4, -5.8) {};
		\node [style=none] (35) at (2, -3.5) {};
		\node [style=none] (36) at (3, -2.725) {};
		\node [style=none] (37) at (3.75, -1.45) {};
		\node [style=none] (38) at (3.525, -1.95) {};
		\node [style=none] (3) at (0, -4) {};
		\node [style=none, label={$f_0$}] (39) at (-4.075, -4) {};
		\node [style=none, label={$f_0$}] (40) at (-0.3, -5.425) {};
		\node [style=none, label={$f_0$}] (41) at (5.5, -0.075) {};
		\node [style=none, label={$f_0$}] (42) at (1, 5.25) {};
	\end{pgfonlayer}
	\begin{pgfonlayer}{edgelayer}
		\draw [style=new edge style 6] (0.center)
			 to (30.center)
			 to (23.center)
			 to [bend right, looseness=0.50] cycle;
		\draw [style=new edge style 6] (5.center)
			 to (2.center)
			 to (32.center)
			 to [bend right=15, looseness=0.25] cycle;
		\draw [style=new edge style 6] (8.center)
			 to (9.center)
			 to (4.center)
			 to cycle;
		\draw [style=new edge style 15] (30.center)
			 to (0.center)
			 to [bend left=345, looseness=0.50] (17.center)
			 to (22.center)
			 to [bend left=15, looseness=0.50] (6.center)
			 to [bend left=15, looseness=0.50] cycle;
		\draw (22.center) to (0.center);
		\draw (0.center) to (6.center);
		\draw [style=new edge style 15] (27.center)
			 to (24.center)
			 to [bend right, looseness=0.50] (23.center)
			 to (29.center)
			 to [bend left=15, looseness=0.50] cycle;
		\draw (23.center) to (27.center);
		\draw [style=new edge style 10] (24.center)
			 to (25.center)
			 to (26.center)
			 to [bend right=15, looseness=0.75] (27.center)
			 to cycle;
		\draw [style=new edge style 15] (5.center)
			 to (2.center)
			 to [bend right=15, looseness=0.75] (25.center)
			 to (26.center)
			 to cycle;
		\draw [style=new edge style 15] (33.center)
			 to (36.center)
			 to [bend right, looseness=0.50] (38.center)
			 to (31.center)
			 to [bend left=15, looseness=0.75] cycle;
		\draw [style=new edge style 10] (34.center)
			 to (35.center)
			 to [bend right=15, looseness=0.75] (36.center)
			 to (33.center)
			 to [bend left=15, looseness=0.50] cycle;
		\draw [style=new edge style 15] (35.center)
			 to (34.center)
			 to [bend left=15] (4.center)
			 to (8.center)
			 to [bend right, looseness=0.25] (3.center)
			 to [bend right=15] cycle;
		\draw [style=new edge style 10] (18.center)
			 to [bend left=15, looseness=0.50] (20.center)
			 to [bend left=15, looseness=0.50] (21.center)
			 to [bend left=15, looseness=0.50] (22.center)
			 to (17.center)
			 to [bend right=15] (16.center)
			 to (19.center)
			 to [bend left=15, looseness=0.75] cycle;
		\draw (19.center) to (17.center);
		\draw (18.center) to (17.center);
		\draw (20.center) to (17.center);
		\draw [style=new edge style 15] (16.center)
			 to (19.center)
			 to [bend right=15, looseness=0.50] (7.center)
			 to (13.center)
			 to [bend left=15] (1.center)
			 to (15.center)
			 to cycle;
		\draw [style=new edge style 6] (12.center)
			 to (13.center)
			 to (14.center)
			 to [bend right=15, looseness=0.75] cycle;
		\draw [style=new edge style 15] (13.center)
			 to (12.center)
			 to [bend right=15, looseness=0.75] (10.center)
			 to (11.center)
			 to [bend left, looseness=0.50] cycle;
		\draw [style=new edge style 10] (7.center)
			 to (13.center)
			 to (14.center)
			 to cycle;
		\draw [style=new edge style 10] (10.center)
			 to [bend right=15] (9.center)
			 to (8.center)
			 to [bend left, looseness=0.50] (11.center)
			 to cycle;
		\draw [style=new edge style 10] (28.center)
			 to [bend left=15, looseness=0.50] (29.center)
			 to (23.center)
			 to (30.center)
			 to [bend left=15, looseness=0.50] cycle;
		\draw [style=new edge style 10] (32.center)
			 to (2.center)
			 to [bend left=15, looseness=0.50] (37.center)
			 to (38.center)
			 to (31.center)
			 to [bend right=15, looseness=0.50] cycle;
		\draw (11.center) to (10.center);
		\draw (3.center) to (4.center);
		\draw (35.center) to (34.center);
		\draw (36.center) to (33.center);
		\draw (37.center) to (31.center);
		\draw (38.center) to (31.center);
		\draw (25.center) to (26.center);
		\draw (24.center) to (27.center);
		\draw (23.center) to (28.center);
		\draw (23.center) to (29.center);
		\draw (21.center) to (17.center);
		\draw (22.center) to (17.center);
		\draw (19.center) to (1.center);
		\draw (19.center) to (15.center);
		\draw (19.center) to (13.center);
		\draw (19.center) to (16.center);
		\draw (7.center) to (13.center);
	\end{pgfonlayer}
\end{tikzpicture}

    \caption{Illustration for the edge-sharing components and units in \cref{almost_disjoint_faces_between_two_cycles}. The subgraphs in pink correspond to faces that are not in $\mathcal{F}$. The edge-sharing components in $\mathcal{B}(\mathcal{F})$ have a dark blue outline. The units of each edge-sharing component are in alternating blue and green. We also indicate the face $f_0$ selected for each edge-sharing component.}
    \label{fig:square_lemma_between}
    \end{figure}

    If there exists a vertex $v \in C$ (resp. $v \in C'$) such that units from different edge-sharing components intersect $C$ (resp. $C'$) only in $v$, then we merge these units.

    Remark then that the units of an edge-sharing component are naturally ordered with respect to $C$ and $C'$ and that the faces from two units that are not (circularly) consecutive are vertex-disjoint.

    Let $\mathcal{U}(\mathcal{F}) = \bigcup_{B \in \mathcal{B}(\mathcal{F})} \mathcal{U}(B)$.
    Let $\Gamma_G$ be an auxiliary graph whose vertices are the units in $\mathcal{U}(\mathcal{F})$, and there is an edge between two units $U, U' \in \mathcal{U}(\mathcal{F})$ if they are not vertex-disjoint. By construction, $\Gamma_G$ is a subgraph of a Hamiltonian cycle. Therefore, this graph is 3-colorable (because it is a subgraph of a cycle).

    Therefore, by taking a color class of $\Gamma_G$ which corresponds to a subset of $\mathcal{F}$ that has size at least $\frac{|\mathcal{F}|}{3}$ (this is possible because the color classes partition $\mathcal{F}$ into three sets), we get a subset of units that are vertex-disjoint and contain at least $\frac{|\mathcal{F}|}{3}$ of the faces in $\mathcal{F}$.

    Moreover, remark that the faces in a unit $U \in \mathcal{U}(\mathcal{F})$ are naturally ordered with respect to $C$ and $C'$ and that by taking every other face in $U$, we get a subset of almost disjoint faces of $U$ of size at least $\frac{|U|}{2}$.

    By combining the two observations, we conclude that there exists a subset of $\mathcal{F}$ of almost disjoint faces of size at least $\frac{|\mathcal{F}|}{6}$.
\end{proof}

The following proposition follows the proof of \cite[Proposition 5.14]{HK2026}; however, the constant $q$ differs significantly between the two results.
This discrepancy arises from the broader notion of well-nested cycles adopted here (see \cref{def:well_nested_cycles}) compared to \cite{HK2026}.

\begin{prop}
    \label{good_square}
     Let $q = \frac{9073}{9072}$ and $m' \in \mathbb{N}$. Let $C_1, ..., C_{2m'} = C$ be $\Pi$-contractible cycles of $G$ that are $\Pi$-well-nested in this order. Let $e \in E(C_1-C_2)$ and let $\mathcal{I}_{\rm N}(C_{2m'})$ be a maximal size subset of almost disjoint $\Pi$-faces in $\text{Int}(C_{2m'-2})$ that are not $\Pi_e$-faces. Then, \[ |\mathcal{I}_{\rm N}(C_{2m'})| \geq q^{m'-2}\]
\end{prop} 

\begin{proof}
    For $2 \leq i \leq 2m'-2$, let $\mathcal{I}(C_{i+2})$ be the set of faces in $\text{Int}(C_i, \Pi)$ in $\Pi$ and let $\mathcal{B}(C_{i+2})$ be the set of faces in $\text{Int}(C_{i+2} \cup C_{i+1}, \Pi)$ that touch $C_{i+2}$. Let $\mathcal{B}_{\rm N}(C_{i+2})$ be the set of the faces in $\mathcal{B}(C_{i+2})$ that are $\Pi_e$-noncontractible. Moreover, let $\mathcal{I}_{\rm N}(C_{i+2})$ be a maximal size subset of almost disjoint cycles from $\mathcal{I}(C_{i+2})$ that are $\Pi_e$-noncontractible. Then, we define $\Tilde{C}_{i+1}$ to be the cycle closest to $C_{i+2}$ so that $\mathcal{B}(C_{i+2}) \subseteq \text{Int}(C_{i+2} \cup \Tilde{C}_{i+1}, \Pi)$. Finally, by \cref{bad_square}, $(C_{i+2}, \Tilde{C}_{i+1}, C_i)$ is a good contractible square with respect to $e$. 
    
    By \cref{C_e_non_contractible}, as $\text{int}(C_2, \Pi)$ contains $e$, it is $\Pi_e$-noncontractible.
    Hence, $\mathcal{I}_{\rm N}(C_4)$ is not empty.

    Let's first show that, for every $4 \leq i \leq 2m'$ there exists a subset of almost disjoint faces of $\mathcal{B}_{\rm N}(C_i)$ of size at least $\frac{| \mathcal{B}_{\rm N}(C_i)|}{12}$.

    \setcounter{claim}{0}

    \begin{claim}
        \label{fraction_of_B_N_almost_disjoint}
        For $4 \leq i \leq 2m'$, there exists a subset of almost disjoint faces of $\mathcal{B}_{\rm N}(C_i)$ of size at least $\frac{|\mathcal{B}_{\rm N}(C_i)|}{12}$.
    \end{claim}

    \begin{proof_claim}
        First, let's partition the faces in $\mathcal{B}(C_i)$ into sets $(\mathcal{F}_j)_{j \in \mathbb{N}^*}$. We define the sets $(\mathcal{F}_j)_{j \in \mathbb{N}^*}$ inductively as follows:

        For $j = 1$, $\mathcal{F}_1$ is the set of faces from $\mathcal{B}(C_i)$ that intersect $\Tilde{C}_{i-1}$. Moreover, we define $C_1'$ to be the cycle closest to $C_i$ so that $\mathcal{B}(C_i) - \mathcal{F}_1 \subseteq \text{Int}(C_i \cup C_1', \Pi)$.

        Suppose that for $j \geq 1$, $\mathcal{F}_j$ and $C'_j$ have been defined. Then, we define $\mathcal{F}_{j+1}$ as the set of faces from $\mathcal{B}(C_i) - \cup_{k=1}^j \mathcal{F}_k$ that intersect $C'_j$. Moreover, we define $C_{j+1}'$ to be the cycle closest to $C_i$ so that $\mathcal{B}(C_i) - \bigcup_{k=1}^{j+1} \mathcal{F}_k \subseteq \text{Int}(C_i \cup C_{j+1}', \Pi)$.

        Finally, $(\mathcal{F}_i)_{i \in \mathbb{N}^*}$ is defined. See \cref{fig:good_square_the_F_i} for an illustration. As $\mathcal{B}(C_i)$ is finite, there exists $M \in \mathbb{N}^*$ such that $\mathcal{F}_j = \emptyset$ for all $j > M$.

		\begin{figure}[h!]
	    \centering
    	\input{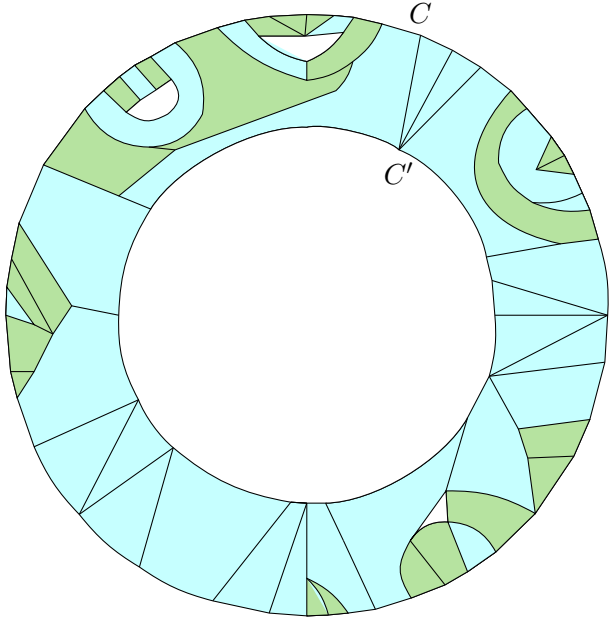}

    	\caption{Illustration for the sets $(\mathcal{F}_i)_{i \in \mathbb{N}^*}$ in \cref{fraction_of_B_N_almost_disjoint}. The subgraphs in white correspond to faces that are not in $\mathcal{B}(C)$ because they don't touch $C$. The sets $\bigcup_{j \geq 1} \mathcal{F}_{2j}$ and $\bigcup_{j \geq 0} \mathcal{F}_{2j+1}$ are depicted respectively in blue and green. Remark that, in this example, for $i > 5$, $\mathcal{F}_i$ is empty.}
    	\label{fig:good_square_the_F_i}
    	\end{figure}

        Remark that by construction, for $j \geq 1$, $\mathcal{F}_j$ and $\mathcal{F}_{j+2}$ contain vertex-disjoint faces. Moreover, remark that, for $j \geq 1$, the faces in $\mathcal{F}_j \cap \mathcal{B}_{\rm N}(C_i)$ touch the two cycles $C_i$ and $C_{j-1}'$ (or $\Tilde{C}_{i-1}$ if $j=1$). Hence, by \cref{almost_disjoint_faces_between_two_cycles}, there is a subset of faces from $\mathcal{F}_j \cap \mathcal{B}_{\rm N}(C_i)$ that are almost disjoint and this subset has size at least $\frac{|\mathcal{F}_j \cap \mathcal{B}_{\rm N}(C_i)|}{6}$.

        To conclude, by taking the biggest subset in the partition of $\mathcal{B}_{\rm N}(C_i)$ into $\bigcup_{j \geq 1} \mathcal{F}_{2j} \cap \mathcal{B}_{\rm N}(C_i)$ and $\bigcup_{j \geq 0} \mathcal{F}_{2j+1} \cap \mathcal{B}_{\rm N}(C_i)$ and then applying \cref{almost_disjoint_faces_between_two_cycles}, we get a subset of almost disjoint faces from $\mathcal{B}_{\rm N}(C_i)$ of size at least $\frac{|\mathcal{B}_{\rm N}(C_i)|}{12}$.
    \end{proof_claim}

    \vspace{0.3cm}

    As $(C_{2i}, \tilde{C}_{2i-1}, C_{2i-2})$ is a good contractible square, for $2 \leq i \leq m'-1$, $\mathcal{B}_{\rm N}(C_{2i})$ has size at least $\frac{|\mathcal{I}_{\rm N}(C_{2i})|+54}{756}$. 
    Hence, by \cref{fraction_of_B_N_almost_disjoint}, $\mathcal{I}_{\rm N}(C_{2i+2})$ has size at least 
    \begin{align*}
        |\mathcal{I}_{\rm N}(C_{2i})| + \frac{|\mathcal{B}_{\rm N}(C_{2i})|}{12}  &\geq |\mathcal{I}_{\rm N}(C_{2i})| + \frac{|\mathcal{I}_{\rm N}(C_{2i})|+54}{9072} \\
        & = \frac{9073|\mathcal{I}_{\rm N}(C_{2i})|+54}{9072} \\
        & \geq \frac{9073}{9072} |\mathcal{I}_{\rm N}(C_{2i})| = q |\mathcal{I}_{\rm N}(C_{2i})| \\
    \end{align*}

    By recurrence, \[ |\mathcal{I}_{\rm N}(C_{2m'})| \geq q^{m'-2} |\mathcal{I}_{\rm N}(C_4)| \geq q^{m'-2}\]
\end{proof}

\begin{cor}
    \label{good_square_cor_general}
    Let $q = \frac{9073}{9072}$ and $m = 2(\lfloor\log_{q}(3g+4)\rfloor + 2)$. The graph $G$ contains at most $m$ cycles that are $\Pi$-well-nested. 
\end{cor} 

\begin{proof}
    Let $m' \in \mathbb{N}$. Let $C_1, ..., C_{2m'}$ be $\Pi$-contractible cycles of $G$ that are $\Pi$-well-nested in this order. Let $e \in E(C_1-C_2)$.
    
    By \cref{good_square}, $|\mathcal{I}_{\rm N}(C_{2m'})| \geq q^{m'-2}$.
    
    If $m'-2 \geq \log_{q}(3g+4)$, then $q^{m'-2} \geq 3g+4$. However, by \cref{homotopic_cycles_variant1}, as $g(\Pi) \leq g+2$, $|\mathcal{I}_{\rm N}(C_{2m'})| \leq 3g+3$. Then, this implies that $m' \leq \lfloor\log_{q}(3g+4)\rfloor + 2$. Finally, there are at most $m = 2m' = 2(\lfloor\log_{q}(3g+4)\rfloor + 2)$ cycles that are $\Pi$-well-nested.
\end{proof}

\begin{cor}
    \label{good_square_cor_attachments}
    Let $q = \frac{9073}{9072}$. Let $C$ be a $\Pi$-contractible cycle of $G$ and suppose that $\text{ext}(C,\Pi)$ has at most $k$ attachments on $C$. Then, $\text{Int}(C,\Pi)$ contains at most $2 (\lfloor\log_q(60k+180)\rfloor+2)$ cycles that are $\Pi$-well-nested.
\end{cor} 

\begin{proof}
    Let $m' \in \mathbb{N}$. Let $C_1, ..., C_{2m'} = C$ be $\Pi$-contractible cycles of $G$ that are $\Pi$-well-nested in this order. Let $e \in E(C_1-C_2)$.

    Let $\tilde{G} = \Pi_e(\text{Ext}(C,\Pi))$. First, let's show that $\tilde{G}$ has genus at most $g - \frac{|\mathcal{I}_{\rm N}(C)|}{30} +5$.
    
    \setcounter{claim}{0}

    \begin{claim}
        \label{claim:genus_I(C)}
        The graph $\tilde{G}$ has genus at most $g - \frac{|\mathcal{I}_{\rm N}(C)|}{30}+7$.
    \end{claim}

    \begin{proof_claim}
        By \cref{at_most_10_homotopic_cycles}, there are at most 10 $\Pi_e$-homotopic cycles in $\mathcal{I}_{\rm N}(C)$. Therefore, as $\mathcal{I}_{\rm N}(C) \subseteq \mathcal{I}(C)$, $\mathcal{I}(C)$ contains at least $\frac{|\mathcal{I}_{\rm N}(C)|}{10}$ almost disjoint $\Pi_e$-noncontractible nonhomotopic cycles. Finally, by \cref{homotopic_cycles_variant1}, $\Pi_e(\mathcal{I}(C))$ has genus at least $\frac{|\mathcal{I}_{\rm N}(C)|}{30}-1$.

        Moreover, by \cref{genus_tildeG}, $g(\Pi_e(\tilde{G})) \leq g - g(\Pi_e(\mathcal{I}(C)))+6$. 
        
        Therefore $g(\Pi_e(\tilde{G})) \leq g - \frac{|\mathcal{I}_{\rm N}(C)|}{30}+7$.
    \end{proof_claim}

    \vspace{1em}
    
    By \cref{claim:genus_I(C)}, we know that $\Pi_e(\tilde{G})$ has genus at most $g - \frac{|\mathcal{I}_{\rm N}(C)|}{30}+7$. Moreover, as there are at most $k$ attachments of $\text{ext}(C,\Pi)$ on $C$, then by adding at most $k$ handles or twisted handles, we create an empty cylinder with $C$ on one of its boundaries. Then, $\text{int}(C,\Pi)$ can be embedded (for example, as in $\Pi$) in this cylinder. Finally, as we add at most $k$ handles and twisted handles, we manage to embed the graph $G$ in a surface of genus at most $g - \frac{|\mathcal{I}_{\rm N}(C)|}{30}+7 + 2k$.

    Because $G$ has genus at least $g+1$, we get that $g -\frac{|\mathcal{I}_{\rm N}(C)|}{30}+7 + 2k \geq g+1$. Therefore, $\frac{|\mathcal{I}_{\rm N}(C)|}{30} \leq 2k+6$.

    By \cref{good_square}, $|\mathcal{I}_{\rm N}(C)| \geq q^{m'-2}$. Therefore, 
    
    \begin{align*}
        \frac{q^{m'-2}}{30} & \leq 2k+6 \\
        q^{m'-2} & \leq 60k+ 180 \\
        m' & \leq \lfloor\log_q(60k+180)\rfloor+2
    \end{align*}

    Finally, there are at most $m = 2m' \leq 2 (\lfloor\log_q(60k+180)\rfloor+2)$.
\end{proof}

\subsection{Homotopic nested cycles} \label{subsec:homotopic_nested_cycles}

\begin{defi}[Well-homotopic]
    Let $H$ be a graph with embedding $\Pi_H$. Let $C$ and $C'$ be two $\Pi_H$-noncontractible homotopic cycles of $H$. We say that $C$ is \textit{$\Pi_H$-well-homotopic} (or \textit{well homotopic}, if it is clear in the context) in $C'$ if one of the following condition holds:
    \begin{itemize}
        \item $C$ and $C'$ are disjoint (\textit{(free,free)-well-homotopic} or \textit{freely well homotopic}), or 
        \item $C$ and $C'$ intersect in a vertex $v$ (\textit{($v$,free)-well-homotopic} or \textit{well homotopic pinched on the vertex $v$}), or 
        \item $C$ and $C'$ both intersect the same face $f$ of $(G, \Pi)$ and if $C$ intersect $f$ in a path $P$ (of length at least three) then $C'$ intersect $f$ in a subpath that is in the interior of $P$ (\textit{($f$,free)-well-homotopic} or \textit{well homotopic pinched on the face $f$}), or
        \item $C$ and $C'$ intersect at two disjoint vertices $v$ and $v'$ (\textit{($v$,$v'$)-well-homotopic} or \textit{well homotopic pinched on two vertices $v$ and $v'$}), or
        \item $C$ and $C'$ both intersect the same two faces $f$ and $f'$ of $(G,\Pi)$ and if $C$ intersects $f$ (resp. $f'$) in a subpath $P$ (resp. $P'$), of length at least $3$, then $C'$ intersect $f$ (resp. $f'$) in a subpath that is in the interior of $P$ (resp. $P'$) (\textit{($f$,$f'$)-well-homotopic} or \textit{well homotopic pinched on the faces $f$ and $f'$}), or
        \item $C$ and $C'$ both intersect at a vertex $v$ and a face $f$ of $(G,\Pi)$ and if $C$ intersects $f$ in a subpath $P$ (of length at least $3$) then $C'$ intersect $f$ in a subpath that is in the interior of $P$ (\textit{($v$,$f$)-well-homotopic} or \textit{well homotopic pinched on the vertex $v$ and the face $f$}).
    \end{itemize}

    In the second and third cases, we say that $C$ and $C'$ are \textit{well homotopic pinched on a piece $p$}. In the last three cases, we say that $C$ and $C'$ are \textit{well homotopic pinched on two pieces $p$ and $p'$}.

    Let $k \geq 1$ and let $C_0, ..., C_k$ be $\Pi_H$-noncontractible homotopic cycles of $H$ in this order in $\Pi_H$. We say that $C_0, ..., C_k$ are \textit{$\Pi_H$-well-homotopic} if:
    \begin{itemize}
        \item for every $0 \leq i \leq k-1$, $C_i$ is $\Pi_H$-freely-well-homotopic in $C_{i+1}$, or 
        \item there is a piece $p$ such that, for every $0 \leq i \leq k-1$, $C_i$ is $\Pi_H$-well-homotopic in $C_{i+1}$ pinched on $p$, or
        \item there are two pieces $p$ and $p'$ such that, for every $0 \leq i \leq k-1$, $C_i$ is $\Pi_H$-well-homotopic in $C_{i+1}$ pinched on $p$ and $p'$.
    \end{itemize}

    The \cref{fig:homotopic_square} show cycles $C^1, C^{1'}, C^{1''}, C^{2''}, C^{2'}, C^2$ that are freely-well-homotopic in this order. 
\end{defi}

We extend the definition of closest cycles to noncontractible homotopic cycles:

\begin{defi}[Closest]
    Let $H$ be a graph with embedding $\Pi_H$. Let $C$ and $C'$ be two $\Pi_H$-noncontractible cycles of $H$, $\Pi_H$-homotopic in this order. Let $\mathcal{P}$ be a property on the cycles of $(H,\Pi_H)$. 
    
    Then, we say that $C'$ is the cycle \textit{closest} to $C$ that has property $\mathcal{P}$ if, for every cycle $C''$ that is $\Pi_H$-homotopic to $C$ in the order $C, C''$ and has property $\mathcal{P}$, we have $C' \subseteq \text{Int}(C \cup C'', \Pi_H)$.
\end{defi}

Let $C^1$ and $C^2$ be two $\Pi$-well-homotopic cycles in $G$. Then, let $\mathcal{F}(C^1 \cup C^2, \Pi)$ be the set of all faces inside of $\text{Int}(C^1 \cup C^2, \Pi)$ in $\Pi$.

\begin{defi}[Homotopic square]
    \label{def:homotopic_square}
    Let $C^1, C^2, C^{1'}, C^{2'}, C^{1''}, C^{2''}$ be $\Pi$-noncontractible homotopic cycles in $G$, well homotopic in the following order: $C^1, C^{1'}, C^{1''}, C^{2''}, C^{2'}, C^2$. Let $C^{12} = C^1 \cup C^2$. Let $\mathcal{B}(C^{12})$ be the union of the sets of faces in $\mathcal{F}(C^1 \cup C^{1'}, \Pi)$ that touch $C^1$ and of faces in $\mathcal{F}(C^2 \cup C^{2'}, \Pi)$ that touch $C^2$ (we call this set the boundary of $C^{12}$). 
    
    We say that $(C^1, C^2, C^{1'}, C^{2'}, C^{1''}, C^{2''})$ is a \textit{homotopic square} with respect to $C^1$ and $C^2$ if $C^{1'}$ and $C^{2'}$ are so that, for every other cycles $\Tilde{C}^1$ and $\Tilde{C}^2$ with $\mathcal{B}(C^{12}) \subseteq \text{Int}(C^1 \cup \Tilde{C}^1,\Pi) \cup \text{Int}(C^2 \cup \Tilde{C}^2,\Pi)$, we have $C^{1'} \subseteq \text{Int}(C^1 \cup \Tilde{C}^1,\Pi)$ and $C^{2'} \subseteq \text{Int}(C^2 \cup \Tilde{C}^2,\Pi)$. 

    We say that a homotopic square $(C^1, C^2, C^{1'}, C^{2'}, C^{1''}, C^{2''})$ is respectively \textit{full}, \textit{pinched on $p$} and \textit{pinched on $p$ and $p'$} if $C^1, C^2, C^{1'}, C^{2'},$ $C^{1''}, C^{2''}$ are respectively $\Pi$-freely-well-homotopic, $\Pi$-well-homotopic pinched on $p$, and $\Pi$-well-homotopic pinched on $p$ and $p'$ in the order $C^1, C^{1'}, C^{1''}, C^{2''}, C^{2'}, C^2$.

    \cref{fig:homotopic_square} depicts the full homotopic square $(C^1, C^2, C^{1'}, C^{2'}, C^{1''}, C^{2''})$. 
\end{defi}

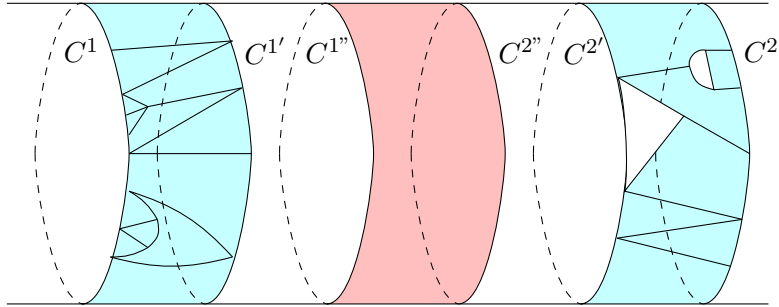
\begin{figure}[h!]
    \centering
    \begin{tikzpicture}
	\begin{pgfonlayer}{nodelayer}
		\node [style=none] (0) at (3.75, 0) {};
		\node [style=none] (1) at (5, 4) {};
		\node [style=none] (2) at (6.25, 0) {};
		\node [style=none] (3) at (5, -4) {};
		\node [style=none] (4) at (0.5, 0) {};
		\node [style=none] (5) at (1.75, 4) {};
		\node [style=none] (6) at (3, 0) {};
		\node [style=none] (7) at (1.75, -4) {};
		\node [style=none] (8) at (-3, 0) {};
		\node [style=none] (9) at (-1.75, 4) {};
		\node [style=none] (10) at (-0.5, 0) {};
		\node [style=none] (11) at (-1.75, -4) {};
		\node [style=none] (12) at (-6.25, 0) {};
		\node [style=none] (13) at (-5, 4) {};
		\node [style=none] (14) at (-3.75, 0) {};
		\node [style=none] (15) at (-5, -4) {};
		\node [style=none] (16) at (-9.5, 0) {};
		\node [style=none] (17) at (-8.25, 4) {};
		\node [style=none] (18) at (-7, 0) {};
		\node [style=none] (19) at (-8.25, -4) {};
		\node [style=none] (20) at (7, 0) {};
		\node [style=none] (21) at (8.25, 4) {};
		\node [style=none] (22) at (9.5, 0) {};
		\node [style=none] (23) at (8.25, -4) {};
		\node [style=none] (24) at (-10.25, 4) {};
		\node [style=none] (25) at (-10.25, -4) {};
		\node [style=none] (26) at (10, 4) {};
		\node [style=none] (27) at (10, -4) {};
		\node [style=none] (29) at (-4.25, 3) {};
		\node [style=none] (30) at (-7.175, 1.575) {};
		\node [style=none] (31) at (-4, 1.75) {};
		\node [style=none] (32) at (-6.5, 1.25) {};
		\node [style=none] (33) at (-7.075, 1.025) {};
		\node [style=none] (34) at (-7, 0.5) {};
		\node [style=none, label={left:$C^1$}] (35) at (-7.475, 2.75) {};
		\node [style=none] (36) at (-7, -1) {};
		\node [style=none] (37) at (-7.25, -2) {};
		\node [style=none] (38) at (-7.5, -2.75) {};
		\node [style=none] (39) at (-6.5, -2.5) {};
		\node [style=none] (40) at (-6.25, -1.75) {};
		\node [style=none] (41) at (-4.25, -2.75) {};
		\node [style=none, label={right:$C^2$}] (42) at (9.025, 2.75) {};
		\node [style=none] (43) at (9.25, 1.75) {};
		\node [style=none] (44) at (8.325, 2.75) {};
		\node [style=none] (45) at (8.55, 1.7) {};
		\node [style=none] (46) at (7.9, 2.325) {};
		\node [style=none] (47) at (6, 2.025) {};
		\node [style=none] (48) at (6.175, -1) {};
		\node [style=none] (49) at (9.3, -1.75) {};
		\node [style=none] (50) at (5.975, -2.25) {};
		\node [style=none] (51) at (9, -3) {};
		\node [style=none] (52) at (7.75, 1) {};
		\node [style=none, label={left:$C^{1"}$}] (53) at (-0.75, 2.75) {};
		\node [style=none, label={right:$C^{2"}$}] (54) at (2.5, 2.75) {};
		\node [style=none, label={left:$C^{2'}$}] (55) at (6, 2.75) {};
		\node [style=none, label={right:$C^{1'}$}] (56) at (-4.25, 2.75) {};
	\end{pgfonlayer}
	\begin{pgfonlayer}{edgelayer}
		\draw [style=new edge style 4] (23.center)
			 to (3.center)
			 to [in=-90, out=0, looseness=0.50] (2.center)
			 to [in=-15, out=90, looseness=0.50] (1.center)
			 to (21.center)
			 to [in=90, out=-15, looseness=0.50] (22.center)
			 to [in=0, out=-90, looseness=0.50] cycle;
		\draw [style=new edge style 3, in=165, out=-90, looseness=0.50] (0.center) to (3.center);
		\draw [style=new edge style 3, in=180, out=90, looseness=0.50] (0.center) to (1.center);
		\draw [style=new edge style 6] (9.center)
			 to (5.center)
			 to [in=90, out=-15, looseness=0.50] (6.center)
			 to [in=0, out=-90, looseness=0.50] (7.center)
			 to (11.center)
			 to [in=-90, out=0, looseness=0.50] (10.center)
			 to [in=-15, out=90, looseness=0.50] cycle;
		\draw [style=new edge style 3, in=165, out=-90, looseness=0.50] (4.center) to (7.center);
		\draw [style=new edge style 3, in=180, out=90, looseness=0.50] (4.center) to (5.center);
		\draw [style=new edge style 3, in=165, out=-90, looseness=0.50] (8.center) to (11.center);
		\draw [style=new edge style 3, in=180, out=90, looseness=0.50] (8.center) to (9.center);
		\draw [style=new edge style 4] (18.center)
			 to [in=-15, out=90, looseness=0.50] (17.center)
			 to (13.center)
			 to [in=90, out=-15, looseness=0.50] (14.center)
			 to [in=0, out=-90, looseness=0.50] (15.center)
			 to (19.center)
			 to [in=-90, out=0, looseness=0.50] cycle;
		\draw [style=new edge style 3, in=165, out=-90, looseness=0.50] (12.center) to (15.center);
		\draw [style=new edge style 3, in=180, out=90, looseness=0.50] (12.center) to (13.center);
		\draw [style=new edge style 3, in=165, out=-90, looseness=0.50] (16.center) to (19.center);
		\draw [style=new edge style 3, in=180, out=90, looseness=0.50] (16.center) to (17.center);
		\draw [style=new edge style 3, in=165, out=-90, looseness=0.50] (20.center) to (23.center);
		\draw [style=new edge style 3, in=180, out=90, looseness=0.50] (20.center) to (21.center);
		\draw (13.center) to (9.center);
		\draw (5.center) to (1.center);
		\draw (7.center) to (3.center);
		\draw (15.center) to (11.center);
		\draw (24.center) to (17.center);
		\draw (25.center) to (19.center);
		\draw (21.center) to (26.center);
		\draw (23.center) to (27.center);
		\draw (35.center) to (29.center);
		\draw (30.center) to (29.center);
		\draw (30.center) to (32.center);
		\draw (33.center) to (32.center);
		\draw (34.center) to (32.center);
		\draw (32.center) to (31.center);
		\draw (18.center) to (31.center);
		\draw (18.center) to (14.center);
		\draw [bend left=15] (36.center) to (40.center);
		\draw [bend right, looseness=1.25] (39.center) to (40.center);
		\draw [bend right=15] (38.center) to (39.center);
		\draw (37.center) to (39.center);
		\draw (37.center) to (40.center);
		\draw [bend left=15] (36.center) to (41.center);
		\draw [bend right=15] (38.center) to (41.center);
		\draw (42.center) to (44.center);
		\draw [style=new edge style 5] (45.center)
			 to (44.center)
			 to [bend right=60] (46.center)
			 to [bend right, looseness=1.25] cycle;
		\draw (43.center) to (45.center);
		\draw (46.center) to (47.center);
		\draw (48.center) to (49.center);
		\draw (49.center) to (50.center);
		\draw (50.center) to (51.center);
		\draw [style=new edge style 5] (48.center)
			 to [bend right=15, looseness=0.50] (47.center)
			 to (52.center)
			 to cycle;
		\draw (52.center) to (22.center);
	\end{pgfonlayer}
\end{tikzpicture}
    \caption{Homotopic square. The zone represented in blue is $\mathcal{B}(C^{12})$ (defined e.g. in \cref{def:homotopic_square}) and the zone represented in pink is $\mathcal{I}(C^{12})$ (defined in \cref{def:good_bad_homotopic_square}) where $C^{12} = C^1 \cup C^2$.}
    \label{fig:homotopic_square}
\end{figure}

\begin{defi}[Good/bad homotopic square]
    \label{def:good_bad_homotopic_square}
    Let $e \in \text{int}(C^{1''} \cup C^{2''}, \Pi)$.    
    Let $(C^1, C^2, C^{1'}, C^{2'}, C^{1''}, C^{2''})$ be a homotopic square with respect to $C^1$ and $C^2$. Let $C^{12} = C^1 \cup C^2$. Let $\mathcal{I}(C^{12})$ be the set of $\Pi$-faces in $\text{Int}(C^{1''} \cup C^{2''}, \Pi)$ and let $\mathcal{I}_{\rm N}(C^{12})$ be a maximal size subset of $\Pi$-faces from $\mathcal{I}(C^{12})$ that are almost-disjoint and that are not $\Pi_e$-faces. Let $\mathcal{B}(C^{12})$ be the union of the sets of faces in $\mathcal{F}(C^1 \cup C^{1'}, \Pi)$ that touch $C^1$ and of faces in $\mathcal{F}(C^2 \cup C^{2'}, \Pi)$ that touch $C^2$.
    We say that $(C^1, C^2, C^{1'}, C^{2'}, C^{1''}, C^{2''})$ is a \textit{bad square with respect to $e$} if $\mathcal{I}_{\rm N}(C^{12})$ contains more than $18 \times (42|\mathcal{B}_{\rm N}(C^{12})|-3)$ $\Pi$-faces with $\mathcal{B}_{\rm N}(C^{12})$ being the set of the faces in $\mathcal{B}(C^{12})$ that are not $\Pi_e$-faces. 
    
    Otherwise, we say that this square is \textit{good with respect to $e$}.
\end{defi}

\begin{lem}
        \label{at_most_18_homotopic_cycles_variant}
        Let $(C^1, C^2, C^{1'}, C^{2'}, C^{1''}, C^{2''})$ be a homotopic square with respect to $C^1$ and $C^2$.
        Let $\mathcal{I}(C^{12})$ be the set of $\Pi$-faces in $\text{Int}(C^{1''} \cup C^{2''}, \Pi)$ and let $\mathcal{I}_{\rm N}(C^{12})$ be a maximal size subset of $\Pi$-faces from $\mathcal{I}(C^{12})$ that are almost-disjoint and that are not $\Pi_e$-faces.
        Suppose that embedding $\Pi_e$ was chosen so that it minimizes the size of $\mathcal{I}_{\rm N}(C^{12})$. Let $F_1, ..., F_m \in \mathcal{I}_{\rm N}(C^{12})$ be a maximal size set of almost disjoint $\Pi$-faces that are $\Pi_e$-noncontractible homotopic, then $m \leq 18$.
\end{lem}

\begin{proof}
        Suppose that there are at least $19$ almost disjoint $\Pi$-faces that are $\Pi_e$-noncontractible homotopic. Let $\mathcal{F}$ be the $\Pi$-faces in $\mathcal{I}_{\rm N}(C^{12})$. Contrary to \cref{at_most_10_homotopic_cycles}, these faces may not be cycles (but either two cycles that intersect in a vertex or two disjoint cycles together with an edge that connects them). 
        Let $C_f$ be a cycle contained in the face $f \in \mathcal{F}$ (there is at least one) and let $\mathcal{C} = \{C_f, f \in \mathcal{F}\}$.

        If the square is pinched on two vertices $v$ and $v'$, let's select cycles from $\mathcal{C}$ so that one of $v,v'$ does not intersect any of the selected cycles. First, remark that there is at most one cycle in $\mathcal{C}$ that intersects both $v$ and $v'$: indeed, the faces are chosen to be almost-disjoint.
        Finally, there are at most $\frac{19 - 1}{2} = 9$ cycles in $\mathcal{C}$ so that one of $v,v'$ does not intersect any of the selected cycles.

        If the square is pinched on two vertices $v$ and $v'$, let $C_1, ..., C_9$ be the cycles from $\mathcal{C}$ so that one of $v,v'$ does not intersect any of these cycles. If the square is not pinched on two vertices, let $C_1, ..., C_9$ be distinct cycles from $\mathcal{C}$ (randomly chosen). Without loss of generality, we have that $C_1, ..., C_9$ are $\Pi_e$-homotopic in this order.

        There are at least two of the cycles $C_1, C_5, C_9$ that have the same relative orientation in $\Pi$ and $\Pi_e$, suppose without loss of generality that it is $C_1$ and $C_5$. Let's show that it is possible to modify $\Pi_e$ so that one of $F_2$, $F_3$, or $F_4$ becomes a $\Pi_e$-face. 

        First, note that:
        \begin{itemize}
            \item \underline{If the square is freely well homotopic:} no vertex of $C^1$ (resp. $C^2$) is in $\text{Int}(C_1 \cup C_5, \Pi_e)$, as otherwise, the whole subgraph $\text{Ext}(C^{12},\Pi)$ would be in $\text{Int}(C_1 \cup C_5, \Pi_e)$ which is impossible because $\text{Ext}(C^{12},\Pi)$ cannot be $\Pi_e$-contractible. Therefore, $\text{Int}(C_1 \cup C_5, \Pi_e) \subseteq \text{Int}(C^{12},\Pi)$.
            
            \item \underline{If the square is pinched on a piece $p$:} the same reasoning for $C^1 - p$ (resp. $C^2 - p$) shows that no vertex of $C^1-p$ (resp. $C^2-p$) is in $\text{Int}(C_1 \cup C_5, \Pi_e)$. If $p$ is a face, it implies that $\text{Int}(C_1 \cup C_5, \Pi_e) \subseteq \text{Int}(C^{12},\Pi)$ as there is no edge in $\text{Ext}(C^{12}, \Pi)$ with an endpoint on $p$. If $p$ is a vertex, then a bridge from $\text{ext}(C^{12}, \Pi)$ in $\text{Int}(C_1 \cup C_5, \Pi_e)$, if it exists, would be attached solely to $p$, which contradicts the fact that $G$ is $2$-connected and therefore has no cutvertex.
            
            \item \underline{If the square is pinched on two pieces $p$ and $p'$:} the same reasoning for $C^1 - (p \cup p')$ (resp. $C^2 - (p \cup p')$) shows that no vertex of $C^1 - (p \cup p')$ (resp. $C^2 - (p \cup p')$) is in $\text{Int}(C_1 \cup C_5, \Pi_e)$. If both $p$ and $p'$ are faces, it implies that $\text{Int}(C_1 \cup C_5, \Pi_e) \subseteq \text{Int}(C^{12},\Pi)$ as there is no edge in $\text{Ext}(C^{12}, \Pi)$ with an endpoint on $p$ or $p'$. If $p$ is a vertex and $p'$ is a face, then a bridge from $\text{ext}(C^{12}, \Pi)$ in $\text{Int}(C_1 \cup C_5, \Pi_e)$, if it exists, would be attached solely to $p$, which contradicts the fact that $G$ is $2$-connected and therefore has no cutvertex. If both $p$ and $p'$ are vertices, then as we showed before that $C_1 \cup C_5$ does not intersect both $p$ and $p'$, we can conclude with the same reasoning as when the square is pinched on one vertex.
        \end{itemize}

        Therefore, $\text{Int}(C_1 \cup C_5, \Pi_e) \subseteq \text{Int}(C^{12},\Pi)$ and $\Pi(\text{Int}(C_1 \cup C_5, \Pi_e))$ is an embedding in a disk in $\Pi$ in which $C_1$ and $C_5$ are $\Pi$-facial walks. 

        Let's show that one face $F$ of $F_2, F_3, F_4$ is contained in $\Pi(\text{Int}(C_1 \cup C_5, \Pi_e))$ and therefore $F$ is a $\Pi$-facial walk in this embedding. First, if one face $F$ of $F_2, F_3, F_4$ is a cycle, then this face $F$ is contained in $\Pi(\text{Int}(C_1 \cup C_5, \Pi_e))$. Now, suppose that all of $F_2, F_3, F_4$ are not cycles. Let's show that one of these faces is in $\text{Int}(C_1 \cup C_5, \Pi_e)$. Suppose, by contradiction, that none of them is entirely in $\text{Int}(C_1 \cup C_5, \Pi_e)$. Then, without loss of generality, two of them (say $F_2$ and $F_3$) intersect $C_1$, and because the faces were chosen almost disjoint, they intersect $C_1$ in the same vertex $v$. Remark that $v$ is either the vertex at which the two cycles in $F_2$ (resp. $F_3$) intersect or an endpoint of the edge in $F_2$ (resp. $F_3$) that is not part of its two cycles. However, note that $v$ is then a cutvertex. Indeed, then there exists $C_2'$ and $C_3'$ cycles respectively in $F_2$ and $F_3$ such that $v$ separates $\text{Int}(C_2' \cup C'_3, \Pi)$ from the rest of the graph. This contradicts the fact that $G$ is $2$-connected and therefore has no cutvertex.
        
        Therefore, we indeed find that one of $F_2, F_3, F_4$ is contained in $\Pi(\text{Int}(C_1 \cup C_5, \Pi_e))$ and therefore is a $\Pi$-facial walk in this embedding.

        Finally, by \cref{cylinder_reembedding}, we can modify the embedding $\Pi_e$ so that one of $F_2, F_3, F_4$ is now a $\Pi_e$-face.

        This concludes the proof.
\end{proof}

\begin{lem}
    \label{bad_square_variant}
     Let $e \in \text{int}(C^{1''}\cup C^{2''}, \Pi)$, $G$ contains no bad homotopic square with respect to $e$.
\end{lem}

\begin{proof}
    Suppose by contradiction that $G$ contains a bad homotopic square $(C^1, C^2, C^{1'}, C^{2'}, C^{1''}, C^{2''})$. Let $C^{12} = C^1 \cup C^2$.. Let $\mathcal{I}(C^{12})$ be the set of faces in $\text{Int}(C^{1''} \cup C^{2''}, \Pi)$ and let $\mathcal{B}(C^{12})$ be the union of the sets of faces in $\mathcal{F}(C^1 \cup C^{1'}, \Pi)$ that touch $C^1$ and of faces in $\mathcal{F}(C^2 \cup C^{2'}, \Pi)$ that touch $C^2$. Let $\mathcal{B}_\text{Con}(C^{12})$ and $\mathcal{B}_{\rm N}(C^{12})$ be the set of the faces in $\mathcal{B}(C^{12})$ that are respectively $\Pi_e$-contractible and $\Pi_e$-noncontractible. Moreover, let $\mathcal{I}_{\rm N}(C^{12})$ be a maximal size subset of $\Pi$-faces of $\mathcal{I}(C^{12})$ that are almost disjoint and that are not $\Pi_e$-faces.
    We suppose $|\mathcal{I}_{\rm N}(C^{12})| \geq 18 \times (42 |\mathcal{B}_{\rm N}(C^{12})|-3)$.
    
    First, let's show that there exists no bad homotopic square with respect to $e$ with $|\mathcal{B}_{\rm N}(C^{12})| = 0$.

    \begin{claim}
        $|\mathcal{B}_{\rm N}(C^{12})| > 0$
    \end{claim}

    \begin{proof_claim}
        Suppose that $|\mathcal{B}_{\rm N}(C^{12})| = 0$.

        Then, remark that $\mathcal{B}(C^{12})$ contains at most $4$ edge-sharing components and remark that the embedding of each of these is the same in $\Pi$ and $\Pi_e$. 

        Let $\Tilde{G}$ be the graph obtained from $G_e$ by removing $H = \text{int}(C^{1'} \cup C^{2'}, \Pi)$. $\Tilde{G}$ has genus at most $g$ (because it is a subgraph of $G_e$).


        \begin{itemize}
            \item \underline{If the square is free or pinched on a piece:}
            Let $B_1$ and $B_2$ be the two edge-sharing components of $\mathcal{B}(C^{12})$: $B_1 = \text{Int}(C^1 \cup C^{1'}, \Pi)$ and $B_2 = \text{Int}(C^2 \cup C^{2'}, \Pi)$. If the square is free, let's define $r = \emptyset$. If the square is pinched on a vertex $v$, let $r = v$ and, if the square is pinched on a face $f$, let $r = (C^{1'} \Delta C^{2'}) \cap f$.

            We define $D_1$ and $D_2$ to be respectively $(B_1 \cup r) \cap C^{1'}$ and $(B_2 - r) \cap C^{2'}$.
            If there is no path in $\text{Int}(C^{1'} \cup C^{2'},\Pi)$ from $D_1$ to $D_2$, then it is enough to find two empty disks in $\Pi_e(\tilde{G})$ with respectively on their boundaries the vertices of $D_1$ and $D_2$ in this order, which we can easily do. Otherwise, let $P$ be a path from $D_1$ to $D_2$.
            
            If $P$ connects $B_1$ and $B_2$ with the expected orientation, then we can find an empty disk in $\Pi_e(\tilde{G})$ with the vertices of $C^{1'} \cup C^{2'}$ on its boundary in the right order. Otherwise, suppose that no path $P$ connects $B_1$ and $B_2$ with the expected orientation. In particular, if the square is pinched on a face $f$, the path $P_f$ between $B_1$ and $B_2$ that is a subwalk of $f$ does not connect them with the expected orientation. We deduce that $P$ and $P_f$ have the same signature.
            
            Now, let's look at the two faces $f_1$ and $f_2$ on the path $P$ in $\Pi_e(\tilde{G}\cup P)$. We can see that $f_1 = f_2$ because it is the case in $\Pi_e(\tilde{G})$. Therefore, changing the signature of $P$ does not increase the genus of the embedding; thus, by this manipulation, we return to the previous case.


            \item \underline{If the square is pinched on two pieces:} Let $B_1$, $B_1'$ and $B_2$, $B_2'$ be the four edge-sharing components of $\mathcal{B}(C^{12})$ such that $B_1 \cup B_1' = \text{Int}(C^1 \cup C^{1'}, \Pi)$ and $B_2 \cup B_2' = \text{Int}(C^2 \cup C^{2'}, \Pi)$. If the square is pinched on two vertices $v$ and $v'$, let's define $r = v$ and $r' = v'$. If the square is pinched on a vertex $v$ and a face $f$, let $r = v$ and $r' = (C^{1'} \Delta C^{2'}) \cap f$. If the square is pinched on two faces $f$ and $f'$, let $r = (C^{1'} \Delta C^{2'}) \cap f$ and $r' = (C^{1'} \Delta C^{2'}) \cap f'$.

            We define $D_1$, $D_1'$ and $D_2$, $D_2'$ to be respectively the connected component of $B_1$ in $(B_1 \cup r \cup r') \cap C^{1'}$, the connected component of $B_1'$ in $(B_1' \cup r \cup r') \cap C^{1'}$ and $(B_2 - (r \cup r'))  \cap C^{2'}$, $(B_2' - (r \cup r')) \cap C^{2'}$.

            Now, remark that bridges from $\text{Int}(C^{1'} \cup C^{2'},\Pi)$ are attached either to $D_1 \cup D_2$ or to $D_1' \cup D_2'$.
            Therefore, we can apply the same strategy as above for $D_1$ and $D_2$ (resp. $D_1'$ and $D_2'$). 
         \end{itemize}

        In each case, we manage to find an embedding of $G$ in the surface $S'$ of genus $g$.
        
        This is a contradiction.
    \end{proof_claim}

    \vspace{1em}

    We can now suppose that $|\mathcal{B}_{\rm N}(C^{12})| > 0$.

    \vspace{1em}
    Remove $H = \text{int}(C^{1'} \cup C^{2'}, \Pi)$ from $G_e$ to obtain the graph $\Tilde{G}$. 

    \begin{claim}
        \label{claim:genus_tildeG_variant}
        The embedding of $\Tilde{G}$ induced by $\Pi_e$ has genus at most $g - 8|\mathcal{B}_{\rm N}(C)|$.
    \end{claim}

    \begin{proof_claim}
        Remark that $\mathcal{I}_{\rm N}(C) \subseteq \mathcal{I}(C)$ and therefore $\Pi_e(\mathcal{I}(C))$ contains $18 \times (42 |\mathcal{B}_{\rm N}(C)|-3)$ almost disjoint $\Pi_e$-noncontractible cycles. Moreover, by \cref{at_most_18_homotopic_cycles_variant}, there are at least $42 |\mathcal{B}_{\rm N}(C)|-3$ of them that are pairwise $\Pi_e$-nonhomotopic. Finally, by \cref{homotopic_cycles_variant1}, $\Pi_e(\mathcal{I}(C))$ has genus at least $14|\mathcal{B}_{\rm N}(C)|$. 

        By \cref{genus_tildeG}, we have $g(\Pi_e(\tilde{G})) \leq g - g(\Pi_e(\mathcal{I}(C))+6$.
        
        As $\Pi_e(\mathcal{I}(C^{12}))$ has genus at least $14|\mathcal{B}_{\rm N}(C^{12})|$, we finally conclude that $\Pi_e(\Tilde{G})$ has genus at most $g - 14|\mathcal{B}_{\rm N}(C^{12})|+6 \leq g - 8 |\mathcal{B}_{\rm N}(C^{12})|$.
    \end{proof_claim}

    \vspace{0.3cm}

    Now, we will construct a new embedding of $G$ from the embedding $\Pi_e(\Tilde{G})$. 

    Let $C^{12'} = C^{1'} \cup C^{2'}$ and let $\tilde{\mathcal{B}}_{\rm N}(C^{12})$ be the subset of faces of $\mathcal{B}_{\rm N}(C^{12})$ that touch $C^{12'}$.


    Remark that there are at most $4$ edge-sharing components of $\mathcal{B}(C^{12})$. Let $\mathcal{D}$ be the set of these edge-sharing components.

    Let $D \in \mathcal{D}$.
    The edge-sharing components of $\mathcal{B}_{\rm Con}(C^{12}) \cap D$ that touch $C^{12'}$ and the faces in $\tilde{\mathcal{B}}_{\rm N}(C^{12}) \cap D$ can be ordered with respect to $C^{12'}$. Then, we add a handle or twisted handle between each consecutive edge-sharing component of $\mathcal{B}_{\rm Con}(C^{12}) \cap D$/face of $\tilde{\mathcal{B}}_{\rm N}(C^{12}) \cap D$ $x$ and $y$ so that, for an edge-sharing component, the part of its boundary that contains endpoints of edges from $H$ is reachable from the handle and has the expected orientation. 
    Moreover, we add at most two handles or twisted handles to make sure that the edge-sharing components from $\mathcal{D}$ are connected with the expected orientation.
    
    By \cref{claim:genus_tildeG_variant}, the embedding $\Pi_e(\tilde{G})$ has genus at most $g - 8|\mathcal{B}_{\rm N}(C)|$. Remark that we add at most $2|\mathcal{B}_{\rm N}(C)| + 2 \leq 4 |\mathcal{B}_{\rm N}(C)|$ handles and twisted handles, and therefore creates a new embedding of genus at most $g$. Moreover, this new embedding contains one or two empty cylinders with $C^{12'}$ on their boundaries. Then, $H$ can be embedded (for example, as in $\Pi$) in this cylinder or these cylinders. The graph $G$ can then be embedded in a surface of genus at most $g$, a contradiction.
\end{proof}

\begin{lem}
    \label{almost_disjoint_faces_between_two_cycles_variant}
    Let $C_1,C'_1,C'_2,C_2$ be $\Pi$-well-homotopic cycles of $G$ in this order. Let $\mathcal{F}_1$ (resp. $\mathcal{F}_2$) be a subset of the faces in $\text{Int}(C_1 \cup C'_1, \Pi)$ (resp. $\text{Int}(C_2 \cup C'_2, \Pi)$) that intersect both $C_1$ and $C'_1$ (resp. $C_2$ and $C_2'$).

    There exists a subset of faces from $\mathcal{F}_1 \cup \mathcal{F}_2$ of size at least $\frac{|\mathcal{F}_1| + |\mathcal{F}_2|}{6}$ that are almost disjoint.
\end{lem}

\begin{proof} \quad

    \begin{itemize}
        \item \underline{If $C_1' \cap C_2' = \emptyset$, or $C_1'$ and $C_2'$ intersect in a face or in two faces:} Remark that $\mathcal{F}_1$ and $\mathcal{F}_2$ are disjoint, then, using \cref{almost_disjoint_faces_between_two_cycles} for both sets, we find that there exists a subset of faces from $\mathcal{F}_1 \cup \mathcal{F}_2$ of size at least $\frac{|\mathcal{F}_1| + |\mathcal{F}_2|}{6}$ that are almost disjoint.

        \item \underline{If $C_1'$ and $C_2'$ intersect in a vertex $v$ or in a vertex $v$ and a face $f$:} Then faces from $\mathcal{F}_1$ and $\mathcal{F}_2$ can only intersect in $v$. Let $\mathcal{U}_1,\Gamma_1$ and $\mathcal{U}_2,\Gamma_2$ be the units and the auxiliary graphs obtained respectively from $\mathcal{F}_1$ and $\mathcal{F}_2$ by following the process in \cref{almost_disjoint_faces_between_two_cycles}. By \cref{almost_disjoint_faces_between_two_cycles}, $\Gamma_1$ and $\Gamma_2$ are cycles and therefore $3$-colorable. Let $\Gamma$ be the auxiliary graph obtained from $\mathcal{F}_1 \cup \mathcal{F}_2$ following the process in \cref{almost_disjoint_faces_between_two_cycles}, then, as $\mathcal{U}_1$ and $\mathcal{U}_2$ intersect only in $v$, it is easy to see that $\Gamma$ is $3$-colorable.
        Finally, we conclude, by the same reasoning as \cref{almost_disjoint_faces_between_two_cycles}, that there exists a subset of faces from $\mathcal{F}_1 \cup \mathcal{F}_2$ of size at least $\frac{|\mathcal{F}_1 \cup \mathcal{F}_2|}{6} = \frac{|\mathcal{F}_1| + |\mathcal{F}_2|}{6}$ that are almost disjoint.

        \item \underline{If $C_1'$ and $C_2'$ intersect in two vertices $v$ and $v'$:} Then faces from $\mathcal{F}_1$ and $\mathcal{F}_2$ can only intersect in $v$ and $v'$. Let $\mathcal{U}_1,\Gamma_1$ and $\mathcal{U}_2,\Gamma_2$ be the units and the auxiliary graphs obtained respectively from $\mathcal{F}_1$ and $\mathcal{F}_2$ by following the process in \cref{almost_disjoint_faces_between_two_cycles}. By \cref{almost_disjoint_faces_between_two_cycles}, $\Gamma_1$ and $\Gamma_2$ are cycles and therefore $3$-colorable. Let $\Gamma$ be the auxiliary graph obtained from $\mathcal{F}_1 \cup \mathcal{F}_2$ following the process in \cref{almost_disjoint_faces_between_two_cycles}, then, as $\mathcal{U}_1$ and $\mathcal{U}_2$ intersect only in $v$ and $v'$, and it is easy to see that $\Gamma$ is $3$-colorable.
        Finally, we conclude, by the same reasoning as \cref{almost_disjoint_faces_between_two_cycles}, that there exists a subset of faces from $\mathcal{F}_1 \cup \mathcal{F}_2$ of size at least $\frac{|\mathcal{F}_1 \cup \mathcal{F}_2|}{6} = \frac{|\mathcal{F}_1| + |\mathcal{F}_2|}{6}$ that are almost disjoint. 
    \end{itemize}
\end{proof}

\begin{prop}
    \label{good_square_variant}
    Let $q = \frac{9073}{9072}$ and $m = 2(\lfloor\log_{q}(3g+4)\rfloor + 2)$. $G$ contains at most $2m$ $\Pi$-well-homotopic cycles. 
\end{prop} 

\begin{proof}
    Let $m' \in \mathbb{N}$. Let $C_1, ..., C_{4m'}$ be $\Pi$-well-homotopic cycles of $G$. Suppose that $C_1, ..., C_{4m'}$ are in this order in $\Pi$.

    Let $e \in E(C_{2m'} - C_{2m'-1})$.
    By \cref{bad_square_variant}, each pair of cycles $D_i=(C_{2m'-i+1}, C_{2m'+i})$ ($4 \leq i \leq 4m'$) induces a good homotopic square with respect to $e$.
    For $2 \leq i \leq 2m'$, let $\mathcal{I}(D_{i+2})$ be the set of faces in $\text{Int}(C_{2m'-i+1} \cup C_{2m'+i}, \Pi)$ and let $\mathcal{B}(D_{i+2})$ be the union of the sets of faces in $\text{Int}((C_{2m'-(i+2)+1} \cup C_{2m'-(i+1)+1}, \Pi)$ that touch $C_{2m'-(i+1)+1}$ and of faces in $\text{Int}((C_{2m'+(i+2)} \cup C_{2m'+(i+1)}, \Pi)$ that touch $C_{2m'+(i+2)}$. Let $\mathcal{B}_{\rm N}(D_{i+2})$ be the set of the faces in $\mathcal{B}(D_{i+2})$ that are $\Pi_e$-noncontractible. Moreover, let $\mathcal{I}_{\rm N}(D_{i+2})$ be a maximal size subset of almost disjoint cycles from $\mathcal{I}(D_{i+2})$ that are $\Pi_e$-noncontractible. Then, we define $\tilde{C}_{2m'-(i+1)+1}$ and $\tilde{C}_{2m'+(i+1)}$ to be the cycles closest respectively to $C_{2m'-(i+2)+1}$ and $C_{2m'+(i+2)}$ so that $\mathcal{B}(D_{i+2}) \subseteq \text{Int}(C_{2m'-(i+2)+1} \cup C_{2m'-(i+1)+1}, \Pi) \cup \text{Int}(C_{2m'+(i+2)} \cup C_{2m'+(i+1)}, \Pi)$. By \cref{bad_square_variant}, for $2 \leq i \leq 2m'-2$, \[(C_{2m'-(i+2)+1}, \tilde{C}_{2m'-(i+1)+1}, C_{2m'-i+1}, C_{2m'+i}, \tilde{C}_{2m'+(i+1)}, C_{2m'+(i+2)})\] is a good homotopic square with respect to $e$. 
    
    By \cref{C_e_non_contractible}, as $\text{int}(D_2, \Pi)$ contains $e$, it is $\Pi_e$-noncontractible.
    Hence, $\mathcal{I}_{\rm N}(D_4)$ is not empty.

    Let's first show that, for every $4 \leq i \leq 2m'$ there exists a subset of almost disjoint faces of $\mathcal{B}_{\rm N}(D_i)$ of size at least $\frac{| \mathcal{B}_{\rm N}(D_i)|}{12}$.

    \setcounter{claim}{0}

    \begin{claim}
        \label{fraction_of_B_N_almost_disjoint_variant}
        For $4 \leq i \leq 2m'$, there  exists a subset of almost disjoint faces of $\mathcal{B}_{\rm N}(D_i)$ of size at least $\frac{|\mathcal{B}_{\rm N}(D_i)|}{12}$.
    \end{claim}

    \begin{proof_claim}
        First, let's partition the faces in $\mathcal{B}(D_i)$ into sets $(\mathcal{F}_j)_{j \in \mathbb{N}^*}$. We will define the sets $(\mathcal{F}_j)_{j \in \mathbb{N}^*}$ inductively as follows:

        For $j = 1$, $\mathcal{F}_1$ is the set of faces from $\mathcal{B}(D_i)$ that intersect $\Tilde{C}_{2m'-(i-1)+1} \cup \Tilde{C}_{2m'+(i-1)}$. Moreover, we define $C_1'$ and $C_1''$ to be the cycles closest to respectively $C_{2m'-i+1}$ and $C_{2m'+i}$ so that $\mathcal{B}(D_i) - \mathcal{F}_1 \subseteq \text{Int}(C_{2m'-i+1} \cup C_1', \Pi) \cup \text{Int}(C_{2m'+i} \cup C_1'', \Pi)$.

        Suppose that for $j \geq 1$, $\mathcal{F}_j$, $C'_j$ and $C_j''$ have been defined. Then, we define $\mathcal{F}_{j+1}$ as the set of faces from $\mathcal{B}(D_i) - \bigcup_{k=1}^{j} \mathcal{F}_k$ that intersect $C'_j \cup C_j''$. Moreover, we define $C_{j+1}'$ and $C_{j+1}''$ to be the cycles closest to respectively $C_{2m'-i+1}$ and $C_{2m'+i}$ so that $\mathcal{B}(D_i) - \bigcup_{k=1}^{j+1} \mathcal{F}_k \subseteq \text{Int}(C_{2m'-i+1} \cup C_{j+1}', \Pi) \cup \text{Int}(C_{2m'+i} \cup C_{j+1}'', \Pi)$.

        Finally, $(\mathcal{F}_i)_{i \in \mathbb{N}^*}$ is defined.

        Remark that by construction, for $j \geq 1$, $\mathcal{F}_j$ and $\mathcal{F}_{j+2}$ contain vertex-disjoint faces. Moreover, remark that, for $j \geq 1$, the faces in $\mathcal{F}_j \cap \mathcal{B}_{\rm N}(D_i)$ touch either the two cycles $C_{2m'-i+1}$ and $C_{j-1}'$ (or $\Tilde{C}_{2m'-(i-1)+1}$ if $j=1$) or the two cycles $C_{2m'+i}$ and $C_{j-1}''$ (or $\Tilde{C}_{2m'+(i-1)}$ if $j=1$). Hence, by \cref{almost_disjoint_faces_between_two_cycles_variant}, there is a subset of faces from $\mathcal{F}_j \cap \mathcal{B}_{\rm N}(D_i)$ that are almost disjoint and this subset has size at least $\frac{|\mathcal{F}_j \cap \mathcal{B}_{\rm N}(D_i)|}{6}$.

        To conclude, by taking the biggest subset in the partition of $\mathcal{B}_{\rm N}(D_i)$ into $\bigcup_{j \geq 1} \mathcal{F}_{2j} \cap \mathcal{B}_{\rm N}(D_i)$ and $\bigcup_{j \geq 0} \mathcal{F}_{2j+1} \cap \mathcal{B}_{\rm N}(D_i)$ and then applying \cref{almost_disjoint_faces_between_two_cycles_variant}, we get a subset of faces from $\mathcal{B}_{\rm N}(D_i)$ of size at least $\frac{|\mathcal{B}_{\rm N}(D_i)|}{12}$.
    \end{proof_claim}

    \vspace{0.3cm}

    As the homotopic square with outer cycles $D_{2i}$ is good with respect to $e$ for $2 \leq i \leq m'-1$, $\mathcal{B}_{\rm N}(D_{2i})$ has size at least $\frac{|\mathcal{I}_{\rm N}(D_{2i})|+54}{756}$. Hence, by \cref{fraction_of_B_N_almost_disjoint_variant}, $\mathcal{I}_{\rm N}(D_{2i+2})$ has size at least 
    \begin{align*}
        |\mathcal{I}_{\rm N}(D_{2i})| + \frac{|\mathcal{B}_{\rm N}(D_{2i})|}{12} & \geq |\mathcal{I}_{\rm N}(D_{2i})| + \frac{|\mathcal{I}_{\rm N}(D_{2i})|+54}{9072} \\
        & = \frac{9073}{9072} |\mathcal{I}_{\rm N}(D_{2i})| = q |\mathcal{I}_{\rm N}(D_{2i})| \\
    \end{align*}

    By recurrence, \[|\mathcal{I}_{\rm N}(D_{2m'})| \geq q^{m'-2} |\mathcal{I}_{\rm N}(D_4)| \geq q^{m'-2}\].

    If $m'-2 \geq \log_{q}(3g+4)$, $q^{m'-2} \geq 3g+4$. However, by \cref{homotopic_cycles_variant1}, as $g(\Pi) \leq g+2$, $|\mathcal{I}_{\rm N}(D_{2m'})| \leq 3g+3$. Then, this implies that $m' \leq \lfloor\log_{q}(3g+4)\rfloor + 2$. Finally, there are at most $2m = 4m' = 2\times (2\lfloor(\log_{q}(3g+4)\rfloor + 2))$ disjoint cycles that are $\Pi$-well-homotopic.
\end{proof}

We then prove the following simple consequence of \cref{good_square_variant}:

\begin{cor}
    \label{nb_non_contractible_cycles}
    Let $\mathcal{C}$ be a set of almost disjoint $\Pi$-noncontractible cycles of $G$.
    Then, $\mathcal{C}$ contains at most $4m^2 \times (3g+3)$ cycles.
\end{cor}

\begin{proof}
    By \cref{homotopic_cycles_variant1}, as $g(\Pi) \leq g+2$, there are at most $3g+3$ almost disjoint $\Pi$-noncontractible cycles that are not pairwise $\Pi$-homotopic in $G$.

    Let's now show that $\mathcal{C}$ contains at most $4m^2$ $\Pi$-homotopic cycles. If there are more than $4m^2$ of them, then either more than $2m$ of them are disjoint or more than $2m$ of them all intersect in the same vertex. In both cases, $\mathcal{C}$ induces $\Pi$-well-homotopic cycles, either free or pinched at one vertex, which contradicts \cref{good_square_variant}. We conclude that there are at most $4m^2$ $\Pi$-homotopic cycles in $\mathcal{C}$.
    
    Finally, we get that $\mathcal{C}$ contains at most $4m^2 \times (3g+3)$ cycles.
\end{proof}

\section{Bound on the order of a \texorpdfstring{$\Pi$}{Pi}-contractible 2-connected subgraph of \texorpdfstring{$G$}{G} with a logarithmic-size separator} \label{sec:bound_on_planar_subgraph_logarithmic_separator}

In this section, we will prove that a $\Pi$-contractible $2$-connected subgraph of $G$ with a separator of size logarithmic in $g$ is of order sub-polynomial in $g$. More precisely, the order of this subgraph is bounded by $a (\log g)^{b(\log \log g)^3}$ for some constant $a$ and $b$. 

\vspace{1em}

Recall that $q$ and $m$ where defined in \cref{sec:nested_cycles}: $q = \frac{9073}{9072}$ and $m = 2(\lfloor\log_{q}(3g+4)\rfloor + 2)$. Let $m' = \lfloor\log_{\frac{4}{3}} 3(4m^2 (3g+3)+1)\rfloor$.
Moreover, we define $G_1$ to be a $\Pi$-contractible $2$-connected subgraph of $G$ such that there exists a cycle $C_1$ of $G$ so that $G_1 = \text{Int}(C_1,\Pi)$ and $G_1$ has a separator $A_1$ that separates it from the rest of $G$ and has size at most $A(g) = 6\lfloor\log_{\frac{4}{3}} 3((T(g)+1)\times m' +1)\rfloor \times (12m+8)$ and $T(g) = 264(g+2)(m+1)-1$.

\vspace{1em}

In \cref{subsec:max_degree}, we first bound the degree of a vertex in $G_1$ and the size of a face in $(G_1,\Pi)$. Then, in \cref{subsec:bound_on_planar_subgraph_logarithmic_separator}, we show the bound on the order of $G_1$.

The results in this section are adapted from \cite[Subsections 6.1 and 6.2]{HK2026} to apply in the specific context of $G_1$.

\subsection{Maximum degree of a vertex and maximum size of a face of \texorpdfstring{$(G_1,\Pi)$}{G1, pi}} \label{subsec:max_degree}

Let $\tilde{m} = 2(\lfloor\log_q(60 A(g)+180)\rfloor+2)$.

\vspace{1em}

We prove that the maximum degree of a vertex of $G_1$ and the maximum size of a face in $(G_1,\Pi)$ are bounded by \[\Delta(g) = \left(4 \sqrt{2A(g) (2\tilde{m}+1)^4 \tilde{m}^3}\right)^{\tilde{m}^2}\]
To prove this, we proceed by contradiction and reach a contradiction by proving that a high degree of a vertex (resp. big size of a face) leads to a large number of $\Pi$-contractible nested cycles, which contradicts \cref{good_square_cor_attachments}.

\vspace{1em}

First, remark that, by using \cref{good_square_cor_attachments}, we can show that there are at most $\tilde{m} = O(\log \log g)$ nested cycles in $(G_1, \Pi)$:

\begin{prop}
    \label{loglog_bound_G0}
    There are at most $\tilde{m}$ $\Pi$-well-nested cycles in $(G_1,\Pi)$.
\end{prop}

\begin{proof}
    By definition of $G_1$, there are at most $A(g)$ attachments on the outer cycle $C_1$ of $G_1$. Therefore, by \cref{good_square_cor_attachments}, there are at most $\tilde{m} = 2(\lfloor\log_q(60 A(g)+180)\rfloor+2)$ nested cycles in $(G_1,\Pi)$.
\end{proof}

\vspace{1em}

Then, we define a notion of faces of $G_1$ with respect to $\Pi$ that restricts the faces of $(G,\Pi)$ to $G_1$. This definition of faces makes it possible to still use properties that $G$ has as an excluded minor, which would not have been possible if we had defined the faces of $(G_1,\Pi)$ to be the faces of $(G_1,\Pi(G_1))$. Therefore, it is a more natural definition to use here.

\begin{defi}[Face of $G_1$ with respect to $\Pi$]
    Let $f \in \mathcal{F}(G,\Pi)$. We define the \textit{faces of $G_1$ with respect to $\Pi$} to be the intersection between a face in $\mathcal{F}(G, \Pi)$ and $G_1$, that is to say if $f \in \mathcal{F}(G, \Pi)$ then each connected component of $f \cap G_1$ is a face of $G_1$ with respect to $\Pi$. We denote by $\mathcal{F}(G_1, \Pi)$ the set of these faces.
    
    Moreover, we define the \textit{size of a face $f \in \mathcal{F}(G_1, \Pi)$} to be the number of vertices of $f$.

    Finally, a \textit{piece of $(G_1, \Pi)$} is either a vertex of $G_1$ or a face of $G_1$ with respect to $\Pi$.
\end{defi}

We define a structure called a \textit{fan (with an arch)} that will be used extensively in the proof of \cref{max_degree}.

\begin{defi}[Fan (with an arch)]
    \label{fan}
    Let $p$ be a piece of $(G_1, \Pi)$. A \textit{fan} $H$ from $p$ in $G_1$ is a subgraph of $G_1$ such that:
    \begin{itemize}
        \item $H$ contain a path $P$ (\textit{horizontal path} of $H$), disjoint from $p$;
        \item There are $M$ paths $Q_1, ..., Q_M$ from $p$ to $P$ for some $M \geq 1$, such that, for every $1 \leq j \neq k \leq M$, if $p$ is a vertex, $Q_j \cap Q_k = p$ and, if $p$ is a face, $Q_j \cap Q_k = \emptyset$, $Q_1, ..., Q_M$ intersect $p$ in this order (\textit{vertical paths} of $H$);
        \item For every $1 \leq k \leq M$, $Q_k \cap p$ and $Q_k \cap P$ is a vertex;
        \item $\Pi(H)$ induces a planar embedding.
    \end{itemize}

    The size of $H$ is $M$.

    An illustration of a fan from a vertex and a face can be found in \cref{fig:fan_vertex_face}.

    \begin{figure}[h!]
    \centering
    \begin{subfloat}[Fan on a vertex]{\begin{tikzpicture}
	\begin{pgfonlayer}{nodelayer}
		\node [style=none] (0) at (0, 3.5) {};
		\node [style=none] (1) at (-4, -0.5) {};
		\node [style=none] (2) at (-3, -0.5) {};
		\node [style=none] (3) at (-2, -0.5) {};
		\node [style=none] (4) at (-1, -0.5) {};
		\node [style=none] (5) at (0, -0.5) {};
		\node [style=none] (6) at (1, -0.5) {};
		\node [style=none] (7) at (2, -0.5) {};
		\node [style=none] (8) at (3, -0.5) {};
		\node [style=none] (9) at (4, -0.5) {};
		\node [style=none] (10) at (-5, -0.5) {};
		\node [style=none] (11) at (5, -0.5) {};
		\node [style=none] (12) at (6, -0.5) {};
		\node [style=none] (13) at (-6, -0.5) {};
	\end{pgfonlayer}
	\begin{pgfonlayer}{edgelayer}
		\draw (0.center) to (1.center);
		\draw (0.center) to (2.center);
		\draw (0.center) to (3.center);
		\draw (0.center) to (4.center);
		\draw (0.center) to (5.center);
		\draw (0.center) to (6.center);
		\draw (0.center) to (7.center);
		\draw (0.center) to (8.center);
		\draw (0.center) to (9.center);
		\draw (1.center)
			 to (2.center)
			 to (3.center)
			 to (4.center)
			 to (5.center)
			 to (6.center)
			 to (7.center)
			 to (8.center)
			 to (9.center);
		\draw (10.center) to (0.center);
		\draw (13.center) to (0.center);
		\draw (13.center) to (10.center);
		\draw (10.center) to (1.center);
		\draw (9.center) to (11.center);
		\draw (11.center) to (12.center);
		\draw (0.center) to (11.center);
		\draw (0.center) to (12.center);
	\end{pgfonlayer}
\end{tikzpicture}}
    \end{subfloat}
    \hspace{0.5cm}
    \begin{subfloat}[Fan on a face]{\begin{tikzpicture}
	\begin{pgfonlayer}{nodelayer}
		\node [style=none] (0) at (2, 3.5) {};
		\node [style=none] (1) at (1.5, 3.5) {};
		\node [style=none] (2) at (1, 3.5) {};
		\node [style=none] (3) at (0.5, 3.5) {};
		\node [style=none] (4) at (0, 3.5) {};
		\node [style=none] (5) at (-0.5, 3.5) {};
		\node [style=none] (6) at (-1, 3.5) {};
		\node [style=none] (7) at (-1.5, 3.5) {};
		\node [style=none] (8) at (-2, 3.5) {};
		\node [style=none] (9) at (5, -0.5) {};
		\node [style=none] (10) at (3.75, -0.5) {};
		\node [style=none] (11) at (2.5, -0.5) {};
		\node [style=none] (12) at (1.25, -0.5) {};
		\node [style=none] (13) at (0, -0.5) {};
		\node [style=none] (14) at (-1.25, -0.5) {};
		\node [style=none] (15) at (-2.5, -0.5) {};
		\node [style=none] (16) at (-3.75, -0.5) {};
		\node [style=none] (17) at (-5, -0.5) {};
		\node [style=none] (18) at (-3.75, 4.5) {};
		\node [style=none] (19) at (-3.5, 5.5) {};
		\node [style=none] (20) at (-2.775, 6.025) {};
		\node [style=none] (21) at (-1.85, 6.325) {};
		\node [style=none] (26) at (-2.5, 3.5) {};
		\node [style=none] (27) at (-3, 3.5) {};
		\node [style=none] (28) at (-7.5, -0.5) {};
		\node [style=none] (29) at (-6.25, -0.5) {};
		\node [style=none] (30) at (6.25, -0.5) {};
		\node [style=none] (31) at (7.5, -0.5) {};
		\node [style=none] (32) at (3, 3.5) {};
		\node [style=none] (33) at (2.5, 3.5) {};
		\node [style=none] (34) at (-1, 6.45) {};
		\node [style=none] (35) at (3.75, 4.55) {};
		\node [style=none] (36) at (0, 6.55) {};
		\node [style=none] (37) at (3.5, 5.55) {};
		\node [style=none] (38) at (2.8, 6.025) {};
		\node [style=none] (39) at (1, 6.475) {};
		\node [style=none] (40) at (1.925, 6.325) {};
	\end{pgfonlayer}
	\begin{pgfonlayer}{edgelayer}
		\draw (8.center) to (17.center);
		\draw (7.center) to (16.center);
		\draw (6.center) to (15.center);
		\draw (5.center) to (14.center);
		\draw (4.center) to (13.center);
		\draw (3.center) to (12.center);
		\draw (2.center) to (11.center);
		\draw (1.center) to (10.center);
		\draw (0.center) to (9.center);
		\draw (8.center)
			 to (7.center)
			 to (6.center)
			 to (5.center)
			 to (4.center)
			 to (3.center)
			 to (2.center)
			 to (1.center)
			 to (0.center);
		\draw (17.center)
			 to (16.center)
			 to (15.center)
			 to (14.center)
			 to (13.center)
			 to (12.center)
			 to (11.center)
			 to (10.center)
			 to (9.center);
		\draw [style=new edge style 3, bend left=15] (18.center) to (19.center);
		\draw [style=new edge style 3, bend left=5] (19.center) to (20.center);
		\draw [style=new edge style 3] (20.center) to (21.center);
		\draw (0.center) to (33.center);
		\draw (33.center) to (32.center);
		\draw (9.center) to (30.center);
		\draw (30.center) to (31.center);
		\draw (29.center) to (17.center);
		\draw (28.center) to (29.center);
		\draw (27.center) to (26.center);
		\draw (26.center) to (8.center);
		\draw (27.center) to (28.center);
		\draw (26.center) to (29.center);
		\draw (33.center) to (30.center);
		\draw (32.center) to (31.center);
		\draw [style=new edge style 3, bend right=15] (18.center) to (27.center);
		\draw [style=new edge style 3] (21.center) to (34.center);
		\draw [style=new edge style 3] (34.center) to (36.center);
		\draw [style=new edge style 3] (36.center) to (39.center);
		\draw [style=new edge style 3] (39.center) to (40.center);
		\draw [style=new edge style 3] (40.center) to (38.center);
		\draw [style=new edge style 3, bend left=5] (38.center) to (37.center);
		\draw [style=new edge style 3, bend left=15, looseness=0.75] (37.center) to (35.center);
		\draw [style=new edge style 3, bend left=15] (35.center) to (32.center);
	\end{pgfonlayer}
\end{tikzpicture}}
    \end{subfloat}

    \caption{Fans from a vertex and a face. The solid lines indicate paths, whereas the dotted line shows the boundary of the face from which the fan in Figure (b) is.}
    \label{fig:fan_vertex_face}
    \end{figure}
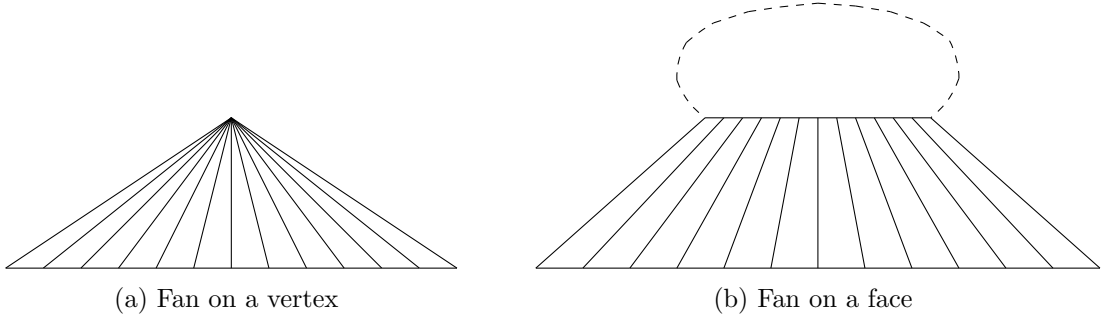

    \vspace{0.3cm}

    Let $p$ be a piece of $(G_1, \Pi)$. A \textit{fan with an arch} $H$ from $p$ in $G_1$ is a subgraph of $G_1$ such that:
    \begin{itemize}
        \item $H$ contains a fan $H'$ from $p$;
        \item Let $Q_1, ..., Q_M$ be the vertical paths of $H'$ and $P$ be its horizontal path, $H$ contains a path $A$  (\textit{arch path} of $H$) whose endpoints are on $Q_1 \cap P$ and $Q_M \cap P$, which is disjoint from $H'$ except on its endpoints. Moreover, $H' \cup A$ is $\Pi$-contractible;
        \item There is at most one vertex in the interior of $A$ (\textit{intersection point} of $A$) such that there is no edge in $\text{Int}(H,\Pi)$ which has an endpoint in $\text{int}(A) - v$ with $v = \emptyset$ if $A$ has no intersection point and $v$ the intersection point of $A$ if $A$ has one. 
    \end{itemize}

    An illustration of a fan with an arch from a vertex and a face can be found in \cref{fig:fan_with_an_arch_vertex_face}.

    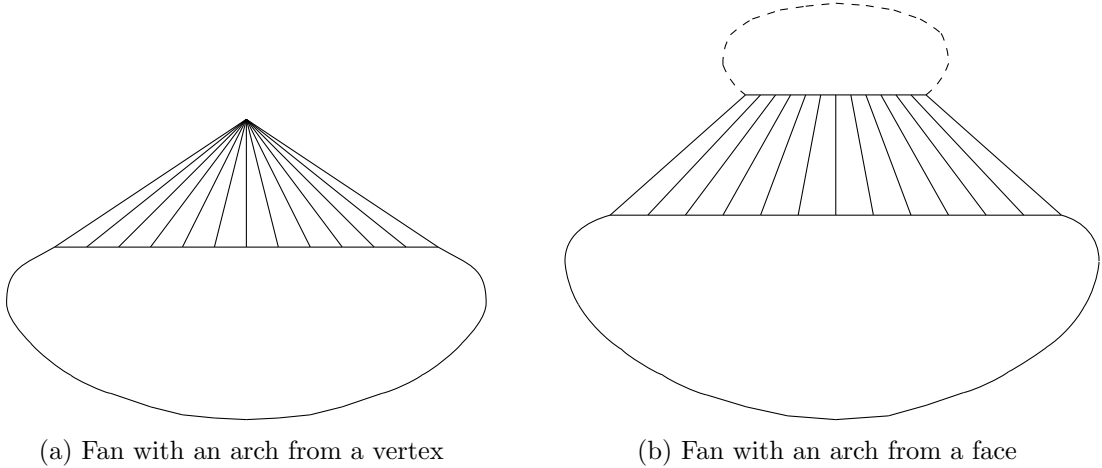
\begin{figure}[h!]
    \centering
    \begin{subfloat}[Fan with an arch from a vertex]{\begin{tikzpicture}[scale=0.85]
	\begin{pgfonlayer}{nodelayer}
		\node [style=none] (0) at (0, 9) {};
		\node [style=none] (1) at (-4, 5) {};
		\node [style=none] (2) at (-3, 5) {};
		\node [style=none] (3) at (-2, 5) {};
		\node [style=none] (4) at (-1, 5) {};
		\node [style=none] (5) at (0, 5) {};
		\node [style=none] (6) at (1, 5) {};
		\node [style=none] (7) at (2, 5) {};
		\node [style=none] (8) at (3, 5) {};
		\node [style=none] (9) at (4, 5) {};
		\node [style=none] (10) at (-5, 5) {};
		\node [style=none] (11) at (5, 5) {};
		\node [style=none] (12) at (6, 5) {};
		\node [style=none] (13) at (-6, 5) {};
		\node [style=none] (14) at (-7.5, 3.25) {};
		\node [style=none] (15) at (-6.75, 2) {};
		\node [style=none] (16) at (-5.5, 1) {};
		\node [style=none] (17) at (-4.25, 0.425) {};
		\node [style=none] (18) at (-3, 0) {};
		\node [style=none] (19) at (-2, -0.25) {};
		\node [style=none] (20) at (-1, -0.35) {};
		\node [style=none] (21) at (0, -0.4) {};
		\node [style=none] (22) at (1, -0.35) {};
		\node [style=none] (23) at (2, -0.25) {};
		\node [style=none] (24) at (3, 0) {};
		\node [style=none] (25) at (4.25, 0.425) {};
		\node [style=none] (26) at (5.5, 1) {};
		\node [style=none] (27) at (6.75, 2) {};
		\node [style=none] (28) at (7.5, 3.25) {};
		\node [style=none] (29) at (6, 5) {};
		\node [style=none] (30) at (7.5, 3.25) {};
	\end{pgfonlayer}
	\begin{pgfonlayer}{edgelayer}
		\draw (0.center) to (1.center);
		\draw (0.center) to (2.center);
		\draw (0.center) to (3.center);
		\draw (0.center) to (4.center);
		\draw (0.center) to (5.center);
		\draw (0.center) to (6.center);
		\draw (0.center) to (7.center);
		\draw (0.center) to (8.center);
		\draw (0.center) to (9.center);
		\draw [style=new edge style 5] (1.center) to (10.center);
		\draw [style=new edge style 5] (10.center) to (13.center);
		\draw [style=new edge style 5, in=135, out=-90, looseness=0.75] (14.center) to (15.center);
		\draw [style=new edge style 5, bend right=15, looseness=0.50] (15.center) to (16.center);
		\draw [style=new edge style 5, bend right, looseness=0.25] (16.center) to (17.center);
		\draw [style=new edge style 5] (17.center) to (18.center);
		\draw [style=new edge style 5] (18.center) to (19.center);
		\draw [style=new edge style 5] (19.center) to (20.center);
		\draw [style=new edge style 5] (20.center) to (21.center);
		\draw [style=new edge style 5] (21.center) to (22.center);
		\draw [style=new edge style 5] (22.center) to (23.center);
		\draw [style=new edge style 5] (23.center) to (24.center);
		\draw [style=new edge style 5] (24.center) to (25.center);
		\draw [style=new edge style 5, bend right=15, looseness=0.50] (25.center) to (26.center);
		\draw [style=new edge style 5, bend right=15, looseness=0.50] (26.center) to (27.center);
		\draw [style=new edge style 5, in=-90, out=45, looseness=0.75] (27.center) to (28.center);
		\draw [style=new edge style 5] (12.center) to (11.center);
		\draw [style=new edge style 5] (11.center) to (9.center);
		\draw [style=new edge style 5] (9.center) to (8.center);
		\draw [style=new edge style 5] (8.center) to (7.center);
		\draw [style=new edge style 5] (7.center) to (6.center);
		\draw [style=new edge style 5] (6.center) to (5.center);
		\draw [style=new edge style 5] (5.center) to (4.center);
		\draw [style=new edge style 5] (4.center) to (3.center);
		\draw [style=new edge style 5] (3.center) to (2.center);
		\draw [style=new edge style 5] (2.center) to (1.center);
		\draw (10.center) to (0.center);
		\draw (13.center) to (0.center);
		\draw (0.center) to (11.center);
		\draw (0.center) to (12.center);
		\draw [in=-150, out=90, looseness=1.25] (14.center) to (13.center);
		\draw [in=-30, out=90, looseness=1.25] (30.center) to (29.center);
	\end{pgfonlayer}
\end{tikzpicture}}
    \end{subfloat} \hspace{0.5cm}
    \begin{subfloat}[Fan with an arch from a face]{\begin{tikzpicture}[scale=0.8]
	\begin{pgfonlayer}{nodelayer}
		\node [style=none] (0) at (2, 5) {};
		\node [style=none] (1) at (1.5, 5) {};
		\node [style=none] (2) at (1, 5) {};
		\node [style=none] (3) at (0.5, 5) {};
		\node [style=none] (4) at (0, 5) {};
		\node [style=none] (5) at (-0.5, 5) {};
		\node [style=none] (6) at (-1, 5) {};
		\node [style=none] (7) at (-1.5, 5) {};
		\node [style=none] (8) at (-2, 5) {};
		\node [style=none] (9) at (5, 1) {};
		\node [style=none] (10) at (3.75, 1) {};
		\node [style=none] (11) at (2.5, 1) {};
		\node [style=none] (12) at (1.25, 1) {};
		\node [style=none] (13) at (0, 1) {};
		\node [style=none] (14) at (-1.25, 1) {};
		\node [style=none] (15) at (-2.5, 1) {};
		\node [style=none] (16) at (-3.75, 1) {};
		\node [style=none] (17) at (-5, 1) {};
		\node [style=none] (18) at (-3.75, 6) {};
		\node [style=none] (19) at (-3.5, 7) {};
		\node [style=none] (20) at (-2.775, 7.525) {};
		\node [style=none] (21) at (-1.85, 7.825) {};
		\node [style=none] (22) at (-2.5, 5) {};
		\node [style=none] (23) at (-3, 5) {};
		\node [style=none] (24) at (-7.5, 1) {};
		\node [style=none] (25) at (-6.25, 1) {};
		\node [style=none] (26) at (6.25, 1) {};
		\node [style=none] (27) at (7.5, 1) {};
		\node [style=none] (28) at (3, 5) {};
		\node [style=none] (29) at (2.5, 5) {};
		\node [style=none] (30) at (-1, 7.95) {};
		\node [style=none] (31) at (3.75, 6.05) {};
		\node [style=none] (32) at (0, 8.05) {};
		\node [style=none] (33) at (3.5, 7.05) {};
		\node [style=none] (34) at (2.8, 7.525) {};
		\node [style=none] (35) at (1, 7.975) {};
		\node [style=none] (36) at (1.925, 7.825) {};
		\node [style=none] (37) at (-9, -0.525) {};
		\node [style=none] (38) at (-8.55, -1.975) {};
		\node [style=none] (39) at (-7.1, -3.475) {};
		\node [style=none] (40) at (-5.775, -4.3) {};
		\node [style=none] (41) at (-4.3, -4.975) {};
		\node [style=none] (42) at (-2.975, -5.375) {};
		\node [style=none] (43) at (-1.6, -5.675) {};
		\node [style=none] (45) at (1.75, -5.675) {};
		\node [style=none] (46) at (3, -5.325) {};
		\node [style=none] (47) at (4.25, -4.9) {};
		\node [style=none] (48) at (5.75, -4.2) {};
		\node [style=none] (49) at (7.025, -3.375) {};
		\node [style=none] (50) at (8.25, -2.05) {};
		\node [style=none] (51) at (8.75, -0.55) {};
		\node [style=none] (52) at (0, -5.8) {};
		\node [style=none] (53) at (-3.5, 5.5) {};
		\node [style=none] (54) at (3.5, 5.5) {};
		\node [style=none] (55) at (7.5, 1) {};
		\node [style=none] (56) at (8.75, -0.525) {};
	\end{pgfonlayer}
	\begin{pgfonlayer}{edgelayer}
		\draw (8.center) to (17.center);
		\draw (7.center) to (16.center);
		\draw (6.center) to (15.center);
		\draw (5.center) to (14.center);
		\draw (4.center) to (13.center);
		\draw (3.center) to (12.center);
		\draw (2.center) to (11.center);
		\draw (1.center) to (10.center);
		\draw (0.center) to (9.center);
		\draw (8.center)
			 to (7.center)
			 to (6.center)
			 to (5.center)
			 to (4.center)
			 to (3.center)
			 to (2.center)
			 to (1.center)
			 to (0.center);
		\draw [style=new edge style 5] (46.center) to (47.center);
		\draw [style=new edge style 5, bend right=15, looseness=0.50] (47.center) to (48.center);
		\draw [style=new edge style 5, bend right=15, looseness=0.25] (48.center) to (49.center);
		\draw [style=new edge style 5, bend right=15, looseness=0.75] (49.center) to (50.center);
		\draw [style=new edge style 5, bend right=15, looseness=0.75] (50.center) to (51.center);
		\draw [style=new edge style 5] (27.center) to (26.center);
		\draw [style=new edge style 5] (26.center) to (9.center);
		\draw [style=new edge style 5] (9.center) to (10.center);
		\draw [style=new edge style 5] (10.center) to (11.center);
		\draw [style=new edge style 5] (11.center) to (12.center);
		\draw [style=new edge style 5] (12.center) to (13.center);
		\draw [style=new edge style 5] (13.center) to (14.center);
		\draw [style=new edge style 5] (14.center) to (15.center);
		\draw [style=new edge style 5] (15.center) to (16.center);
		\draw [style=new edge style 5] (16.center) to (17.center);
		\draw [style=new edge style 5] (17.center) to (25.center);
		\draw [style=new edge style 5] (25.center) to (24.center);
		\draw [style=new edge style 5, bend right=15, looseness=0.75] (37.center) to (38.center);
		\draw [style=new edge style 5, bend right=15, looseness=0.75] (38.center) to (39.center);
		\draw [style=new edge style 5, bend right, looseness=0.25] (39.center) to (40.center);
		\draw [style=new edge style 5, bend right=15, looseness=0.50] (40.center) to (41.center);
		\draw [style=new edge style 5] (41.center) to (42.center);
		\draw [style=new edge style 5] (42.center) to (43.center);
		\draw [style=new edge style 5] (43.center) to (52.center);
		\draw [style=new edge style 5] (52.center) to (45.center);
		\draw [style=new edge style 5] (45.center) to (46.center);
		\draw [style=new edge style 3, bend left=15] (18.center) to (19.center);
		\draw [style=new edge style 3, bend left=5] (19.center) to (20.center);
		\draw [style=new edge style 3] (20.center) to (21.center);
		\draw (0.center) to (29.center);
		\draw (29.center) to (28.center);
		\draw (23.center) to (22.center);
		\draw (22.center) to (8.center);
		\draw (23.center) to (24.center);
		\draw (22.center) to (25.center);
		\draw (29.center) to (26.center);
		\draw (28.center) to (27.center);
		\draw [style=new edge style 3, bend right=15] (18.center) to (23.center);
		\draw [style=new edge style 3] (21.center) to (30.center);
		\draw [style=new edge style 3] (30.center) to (32.center);
		\draw [style=new edge style 3] (32.center) to (35.center);
		\draw [style=new edge style 3] (35.center) to (36.center);
		\draw [style=new edge style 3] (36.center) to (34.center);
		\draw [style=new edge style 3, bend left=5] (34.center) to (33.center);
		\draw [style=new edge style 3, bend left=15, looseness=0.75] (33.center) to (31.center);
		\draw [style=new edge style 3, bend left=15] (31.center) to (28.center);
		\draw [in=-165, out=90] (37.center) to (24.center);
		\draw [in=-15, out=90] (56.center) to (55.center);
	\end{pgfonlayer}
\end{tikzpicture}}
    \end{subfloat}

    \caption{Fans with an arch from a vertex and a face. The solid lines indicate paths, whereas the dotted line shows the boundary of the face from which the fan in Figure (b) is.}
    \label{fig:fan_with_an_arch_vertex_face}
    \end{figure}
\end{defi}

\begin{theo}
    \label{max_degree}
    \[\Delta(G_1) \leq \Delta(g) \text{\quad and \quad} \Delta_F(G_1, \Pi) \leq \Delta(g)\]
\end{theo}

\begin{proof}
    We first define several objects that will play an essential role throughout the proof.

    Let $H$ be a fan (with or without an arch) from $p$ in $G_1$. Let $P$ be its horizontal path and $Q_1, ..., Q_M$ be its vertical paths with $M \geq 2$.
    We define a \textit{column} of $H$ to be the subpath of $P$ between the two endpoints of two consecutive vertical paths from $p$: let $1 \leq k \leq M-1$, then the $k$-th column of $H$ is the subpath of $P$ between the endpoints $Q_k \cap P$ (included) and $Q_{k+1} \cap P$ (excluded).

    Then, we define $\mathcal{F}(H)$ to be the set of faces that touch the interior of the horizontal path $P$ of $H$ in $\text{Ext}(H,\Pi)$. Assume now that $H$ is a fan with an arch, let $A$ be its arch path. Let $C$ be the cycle induced by $A \cup p$ that does not contain $p$ in $\Pi$ and let $C'$ be the cycle bounding $P \cup \bigcup_{1 \leq i \leq M} Q_i$ in $\Pi$, we define $C_A$ to be the circuit that bounds $\text{Int}(C \cup C', \Pi)$. Then, we define $\mathcal{F}(H)$ to be the set of faces in $\text{Int}(C_A,\Pi)$ that touch the horizontal path $P$ of $H$ and do not intersect $A$. We also define $\mathcal{F}_{arch}(H)$ to be the set of faces in $\text{Int}(C_A, \Pi)$ that touches the horizontal path $P$ of $H$ and intersect $A$. See \cref{fig:fan_faces} for illustrations.

    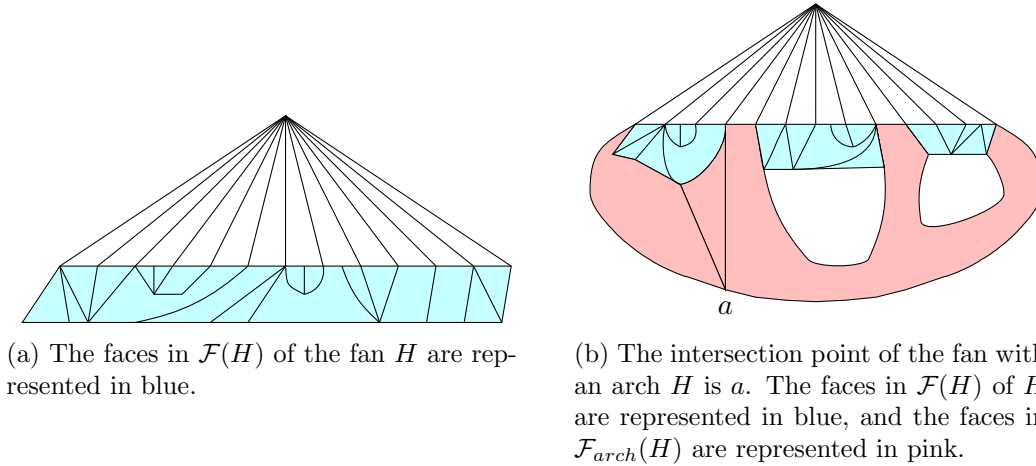
\begin{figure}[h!]
    \centering
    \begin{subfloat}[The faces in $\mathcal{F}(H)$ of the fan $H$ are represented in blue.]{\begin{tikzpicture}
	\begin{pgfonlayer}{nodelayer}
		\node [style=none] (0) at (0, 3.5) {};
		\node [style=none] (1) at (-4, -0.5) {};
		\node [style=none] (2) at (-3, -0.5) {};
		\node [style=none] (3) at (-2, -0.5) {};
		\node [style=none] (4) at (-1, -0.5) {};
		\node [style=none] (5) at (0, -0.5) {};
		\node [style=none] (6) at (1, -0.5) {};
		\node [style=none] (7) at (2, -0.5) {};
		\node [style=none] (8) at (3, -0.5) {};
		\node [style=none] (9) at (4, -0.5) {};
		\node [style=none] (10) at (-5, -0.5) {};
		\node [style=none] (11) at (5, -0.5) {};
		\node [style=none] (12) at (6, -0.5) {};
		\node [style=none] (13) at (-6, -0.5) {};
		\node [style=none] (14) at (-7, -2) {};
		\node [style=none] (15) at (-5.75, -2) {};
		\node [style=none] (16) at (-4, -2) {};
		\node [style=none] (17) at (-5.25, -2) {};
		\node [style=none] (18) at (-5.25, -2) {};
		\node [style=none] (19) at (-3.5, -1.25) {};
		\node [style=none] (20) at (-3.5, -0.5) {};
		\node [style=none] (21) at (-2.75, -1.25) {};
		\node [style=none] (22) at (-2, -2) {};
		\node [style=none] (23) at (-1, -2) {};
		\node [style=none] (24) at (2.5, -2) {};
		\node [style=none] (25) at (0.5, -0.5) {};
		\node [style=none] (26) at (0.5, -1.25) {};
		\node [style=none] (27) at (1, -0.5) {};
		\node [style=none] (28) at (1.5, -0.5) {};
		\node [style=none] (29) at (3.75, -2) {};
		\node [style=none] (30) at (4.75, -2) {};
		\node [style=none] (31) at (5.75, -2) {};
	\end{pgfonlayer}
	\begin{pgfonlayer}{edgelayer}
		\draw (0.center) to (1.center);
		\draw (0.center) to (2.center);
		\draw (0.center) to (3.center);
		\draw (0.center) to (4.center);
		\draw (0.center) to (5.center);
		\draw (0.center) to (6.center);
		\draw (0.center) to (7.center);
		\draw (0.center) to (8.center);
		\draw (0.center) to (9.center);
		\draw [style=new edge style 4] (13.center)
			 to (14.center)
			 to (15.center)
			 to (18.center)
			 to (16.center)
			 to (22.center)
			 to (23.center)
			 to [in=180, out=0] (24.center)
			 to (29.center)
			 to (30.center)
			 to (31.center)
			 to (12.center)
			 to (11.center)
			 to (9.center)
			 to (8.center)
			 to (7.center)
			 to (6.center)
			 to (5.center)
			 to (4.center)
			 to (3.center)
			 to (2.center)
			 to (1.center)
			 to (10.center)
			 to cycle;
		\draw (10.center) to (0.center);
		\draw (13.center) to (0.center);
		\draw (0.center) to (11.center);
		\draw (0.center) to (12.center);
		\draw (15.center) to (13.center);
		\draw (18.center) to (13.center);
		\draw (18.center) to (10.center);
		\draw (18.center) to (1.center);
		\draw (1.center) to (19.center);
		\draw (19.center) to (20.center);
		\draw (19.center) to (2.center);
		\draw (19.center) to (21.center);
		\draw (21.center) to (3.center);
		\draw [bend right=15] (16.center) to (4.center);
		\draw (22.center) to (5.center);
		\draw (5.center) to (23.center);
		\draw [bend right, looseness=1.50] (5.center) to (26.center);
		\draw (26.center) to (25.center);
		\draw [bend left=45] (27.center) to (26.center);
		\draw [bend right=15] (28.center) to (24.center);
		\draw (7.center) to (24.center);
		\draw (8.center) to (24.center);
		\draw (9.center) to (29.center);
		\draw (30.center) to (11.center);
		\draw (11.center) to (31.center);
	\end{pgfonlayer}
\end{tikzpicture}}
    \end{subfloat} \hspace{0.5cm}
    \begin{subfloat}[The intersection point of the fan with an arch $H$ is $a$. The faces in $\mathcal{F}(H)$ of $H$ are represented in blue, and the faces in $\mathcal{F}_{arch}(H)$ are represented in pink.]{\begin{tikzpicture}[scale=0.8]
	\begin{pgfonlayer}{nodelayer}
		\node [style=none] (0) at (0, 2.5) {};
		\node [style=none] (1) at (-4, -1.5) {};
		\node [style=none] (2) at (-3, -1.5) {};
		\node [style=none] (3) at (-2, -1.5) {};
		\node [style=none] (4) at (-1, -1.5) {};
		\node [style=none] (5) at (0, -1.5) {};
		\node [style=none] (6) at (1, -1.5) {};
		\node [style=none] (7) at (2, -1.5) {};
		\node [style=none] (8) at (3, -1.5) {};
		\node [style=none] (9) at (4, -1.5) {};
		\node [style=none] (10) at (-5, -1.5) {};
		\node [style=none] (11) at (5, -1.5) {};
		\node [style=none] (12) at (6, -1.5) {};
		\node [style=none] (13) at (-6, -1.5) {};
		\node [style=none] (29) at (5.5, -1.5) {};
		\node [style=none] (30) at (-4.5, -1.5) {};
		\node [style=none] (31) at (-4.5, -2.25) {};
		\node [style=none] (32) at (-4.5, -3.5) {};
		\node [style=none] (33) at (-1.725, -3) {};
		\node [style=none] (34) at (-0.75, -3) {};
		\node [style=none] (35) at (0.5, -1.5) {};
		\node [style=none] (36) at (1.25, -2.25) {};
		\node [style=none] (37) at (3.75, -2.5) {};
		\node [style=none] (38) at (4.5, -2.5) {};
		\node [style=none] (39) at (5, -1.5) {};
		\node [style=none] (40) at (5.675, -2.5) {};
		\node [style=none] (41) at (2.275, -2.925) {};
		\node [style=none] (42) at (-6, -2.675) {};
		\node [style=none] (43) at (-7.5, -3.75) {};
		\node [style=none] (44) at (-6.75, -5) {};
		\node [style=none] (45) at (-5.5, -6) {};
		\node [style=none] (46) at (-4.25, -6.575) {};
		\node [style=none, label={below:$a$}] (47) at (-3, -7) {};
		\node [style=none] (48) at (-2, -7.25) {};
		\node [style=none] (49) at (-1, -7.35) {};
		\node [style=none] (50) at (0, -7.4) {};
		\node [style=none] (51) at (1, -7.35) {};
		\node [style=none] (52) at (2, -7.25) {};
		\node [style=none] (53) at (3, -7) {};
		\node [style=none] (54) at (4.25, -6.575) {};
		\node [style=none] (55) at (5.5, -6) {};
		\node [style=none] (56) at (6.75, -5) {};
		\node [style=none] (57) at (7.5, -3.75) {};
		\node [style=none] (58) at (3.5, -4.75) {};
		\node [style=none] (59) at (6.25, -3.75) {};
		\node [style=none] (60) at (-0.25, -6) {};
		\node [style=none] (61) at (1.75, -6) {};
		\node [style=none] (62) at (-6.75, -2.5) {};
	\end{pgfonlayer}
	\begin{pgfonlayer}{edgelayer}
		\draw (0.center) to (1.center);
		\draw (0.center) to (2.center);
		\draw (0.center) to (3.center);
		\draw (0.center) to (4.center);
		\draw (0.center) to (5.center);
		\draw (0.center) to (6.center);
		\draw (0.center) to (7.center);
		\draw (0.center) to (8.center);
		\draw (0.center) to (9.center);
		\draw [style=new edge style 4] (10.center)
			 to (1.center)
			 to (2.center)
			 to [bend left=45, looseness=0.75] (32.center)
			 to (42.center)
			 to (62.center)
			 to (13.center)
			 to cycle;
		\draw [style=new edge style 4] (10.center) to (42.center);
		\draw [style=new edge style 6] (32.center)
			 to [bend right=45, looseness=0.75] (2.center)
			 to (47.center)
			 to cycle;
		\draw [style=new edge style 9] (48.center)
			 to (49.center)
			 to (50.center)
			 to (51.center)
			 to (52.center)
			 to (53.center)
			 to (54.center)
			 to [bend right=15, looseness=0.50] (55.center)
			 to [bend right=15, looseness=0.50] (56.center)
			 to [in=-90, out=45, looseness=0.75] (57.center)
			 to [in=-30, out=90, looseness=1.25] (12.center)
			 to (40.center)
			 to [bend left, looseness=0.50] (59.center)
			 to [bend left=90, looseness=0.50] (58.center)
			 to [bend left=15] (37.center)
			 to (8.center)
			 to (7.center)
			 to (41.center)
			 to [bend left=15] (61.center)
			 to [bend left=45, looseness=0.50] (60.center)
			 to [bend left=15] (33.center)
			 to [in=285, out=105] (3.center)
			 to (2.center)
			 to (47.center)
			 to cycle;
		\draw [style=new edge style 6] (37.center) to (8.center);
		\draw [style=new edge style 6] (7.center) to (41.center);
		\draw [style=new edge style 6] (33.center) to (3.center);
		\draw [style=new edge style 9] (47.center)
			 to (32.center)
			 to (42.center)
			 to (62.center)
			 to (13.center)
			 to [in=90, out=-150, looseness=1.25] (43.center)
			 to [in=135, out=-90, looseness=0.75] (44.center)
			 to [bend right=15, looseness=0.50] (45.center)
			 to [bend right, looseness=0.25] (46.center)
			 to cycle;
		\draw [style=new edge style 6] (13.center) to (0.center);
		\draw [style=new edge style 6] (0.center) to (12.center);
		\draw [style=new edge style 6] (29.center) to (40.center);
		\draw [style=new edge style 4] (11.center)
			 to (9.center)
			 to (8.center)
			 to (37.center)
			 to (38.center)
			 to (40.center)
			 to (12.center)
			 to (29.center)
			 to (39.center);
		\draw [style=new edge style 4] (29.center) to (40.center);
		\draw [style=new edge style 4] (41.center)
			 to (7.center)
			 to (6.center)
			 to (5.center)
			 to (4.center)
			 to (3.center)
			 to (33.center)
			 to (34.center)
			 to cycle;
		\draw (10.center) to (0.center);
		\draw (0.center) to (11.center);
		\draw [bend right] (10.center) to (31.center);
		\draw (30.center) to (31.center);
		\draw [bend left=45, looseness=1.25] (1.center) to (31.center);
		\draw [in=165, out=-105, looseness=0.75] (10.center) to (32.center);
		\draw [in=240, out=60] (33.center) to (4.center);
		\draw (4.center) to (34.center);
		\draw (34.center) to (5.center);
		\draw [bend right=45] (35.center) to (36.center);
		\draw [bend right, looseness=0.75] (36.center) to (7.center);
		\draw (6.center) to (36.center);
		\draw [in=-90, out=0] (34.center) to (7.center);
		\draw (39.center) to (38.center);
		\draw (38.center) to (29.center);
		\draw (38.center) to (9.center);
		\draw (10.center) to (62.center);
	\end{pgfonlayer}
\end{tikzpicture}}
    \end{subfloat}

    \caption{The set $\mathcal{F}(H)$ for a fan, and $\mathcal{F}(H)$ and $\mathcal{F}_{arch}(H)$ for a fan with an arch.}
    \label{fig:fan_faces}
    \end{figure}

    Let $H$ be a fan (with or without an arch) from $p$ with the horizontal path $P$. We can suppose that the signature on the edges of $H$ in $\Pi$ is positive, because $H$ is $\Pi$-contractible. Let $C_H$ be the cycle of $H$ so that $H \subseteq \text{Int}(C_H,\Pi)$. Then, a face $f \in \mathcal{F}(H)$ is an \textit{arched face} if $f$ does not intersect $C_H$ in a single segment (vertex or path). Let $f \in \mathcal{F}(H)$ be a face; we define the \textit{arches} of $f$ to be the maximal sub walks of $f$ whose interior vertices are not in $C_H$. Remark that the arches of any face $f \in \mathcal{F}(H)$ are paths with their ends on $C_H$. For an arch $A$ of a face $f \in \mathcal{F}(H)$, we define $C_A$ to be the $\Pi$-contractible cycle induced by $A \cup C_H$ that does not contain $H$. Moreover, remark that either $\text{int}(C_A, \Pi)$ or $\text{ext}(C_A, \Pi)$ contain no edge with an endpoint on the interior of $A$ (because $A$ is a subwalk of a face in $\Pi$). 
    See \cref{fig:arches} for an illustration.

    Let $f, f' \in \mathcal{F}(H)$ be two faces intersecting in $G_1-H$. Remark that, using \cref{2_separated_subgraph_is_edge}, we can show that $f$ and $f'$ must intersect in a vertex or an edge. Let $v_1$ and $v_2$ be the first and last intersection points of $f$ and $f'$ in $G_1$ (potentially $v_1 = v_2$) and let $P_{1,f}, P_{2,f}, P_{1,f'}, P_{2,f'}$ be the subpaths of $f$ and $f'$ from $H$ to $v_1$ and $v_2$. Then, we define the \textit{arches} of $f$ and $f'$ to be $A_1 = P_{1,f} \cup P_{1,f'}$ and $A_2 = P_{2,f} \cup P_{2,f'}$. Remark that $A_1$ and $A_2$ are paths with both ends on $H$. Moreover, remark that either $\text{int}(C_A, \Pi)$ or $\text{ext}(C_A, \Pi)$ contains no edge with an endpoint on the interior of $A$ except on one sole vertex that we call the \textit{intersection point} of $A$. 
    See \cref{fig:arches} for an illustration.

    We define $\mathcal{A}(H)$ to be the set of all arches induced by either one face from $\mathcal{F}(H)$ or by two faces from $\mathcal{F}(H)$ that intersect in $G_1 - H$.
    We say that $A \in \mathcal{A}(H)$ is \textit{empty above} (resp. \textit{empty below}) if $\text{ext}(C_A, \Pi)$ (resp. $\text{int}(C_A, \Pi)$) contains no edge with an endpoint on the interior of $A$ (except maybe on one sole vertex). Remark that, in that case, every arch $A \in \mathcal{A}(H)$ is either empty above or empty below. See \cref{fig:arches} for an illustration.

    Let $A \in \mathcal{A}(H)$ and let $P_A = C_A \cap P$. We define the \textit{size} of $A$ to be the number of columns that $P_A$ intersects. 
    
    \begin{figure}[h!]
    \centering
    \begin{tikzpicture}
	\begin{pgfonlayer}{nodelayer}
		\node [style=none, label={above:$v$}] (0) at (0, 3.5) {};
		\node [style=none] (1) at (-7, -0.5) {};
		\node [style=none] (2) at (-5, -0.5) {};
		\node [style=none] (3) at (-3, -0.5) {};
		\node [style=none] (4) at (-1, -0.5) {};
		\node [style=none] (5) at (1.25, -0.5) {};
		\node [style=none] (6) at (3, -0.5) {};
		\node [style=none] (7) at (5, -0.5) {};
		\node [style=none, label={below:$f$}] (8) at (7, -0.5) {};
		\node [style=none] (9) at (9, -0.5) {};
		\node [style=none, label={below left:$a$}] (10) at (-9, -0.5) {};
		\node [style=none] (11) at (11, -0.5) {};
		\node [style=none] (13) at (-11, -0.5) {};
		\node [style=none] (14) at (-8, -0.5) {};
		\node [style=none] (15) at (-2, -0.5) {};
		\node [style=none, label={below:$g$}] (16) at (0, -3) {};
		\node [style=none, label={right:$b$}] (17) at (-7.675, -0.875) {};
		\node [style=none, label={left:$c$}] (18) at (-2.3, -0.875) {};
		\node [style=none, label={right:$d$}] (19) at (1.2, -0.8) {};
		\node [style=none, label={left:$e$}] (20) at (2.8, -0.85) {};
	\end{pgfonlayer}
	\begin{pgfonlayer}{edgelayer}
		\draw (0.center) to (1.center);
		\draw (0.center) to (2.center);
		\draw (0.center) to (3.center);
		\draw (0.center) to (4.center);
		\draw (0.center) to (5.center);
		\draw (0.center) to (6.center);
		\draw (0.center) to (7.center);
		\draw (0.center) to (8.center);
		\draw (0.center) to (9.center);
		\draw (1.center) to (2.center);
		\draw (2.center) to (3.center);
		\draw [style=new edge style 4] (15.center)
			 to (4.center)
			 to (5.center)
			 to [bend left, looseness=0.75] (16.center)
			 to [in=-105, out=-150] (10.center)
			 to (14.center)
			 to [bend right=60, looseness=1.25] cycle;
		\draw (5.center) to (6.center);
		\draw [style=new edge style 4] (8.center)
			 to [bend left=15] (16.center)
			 to [bend right=15, looseness=1.25] (6.center)
			 to (7.center)
			 to cycle;
		\draw (8.center) to (9.center);
		\draw (10.center) to (0.center);
		\draw (13.center) to (0.center);
		\draw (13.center) to (10.center);
		\draw (9.center) to (11.center);
		\draw (0.center) to (11.center);
		\draw (14.center) to (1.center);
		\draw (3.center) to (15.center);
	\end{pgfonlayer}
\end{tikzpicture}

    \caption{Illustration for arched faces and arches. The fan $H$ is represented, and two faces from $\mathcal{F}(H)$ are represented in blue. The face $F$ on the left is arched, but the face $F'$ on the right is not. The paths from $a$ to $d$ and from $b$ to $c$ are the two arches of $F$, and the path from $e$ to $f$ is the only arch of $F'$. The paths from $d$ to $e$ through $g$ and from $a$ to $f$ through $g$ are both arches of $F$ and $F'$, and both have $g$ as their intersection point. In this representation, the arches from $a$ to $d$ and from $a$ to $f$ are empty below, and the arches from $b$ to $c$ and from $d$ to $e$ are empty above.}
    \label{fig:arches}
    \end{figure}
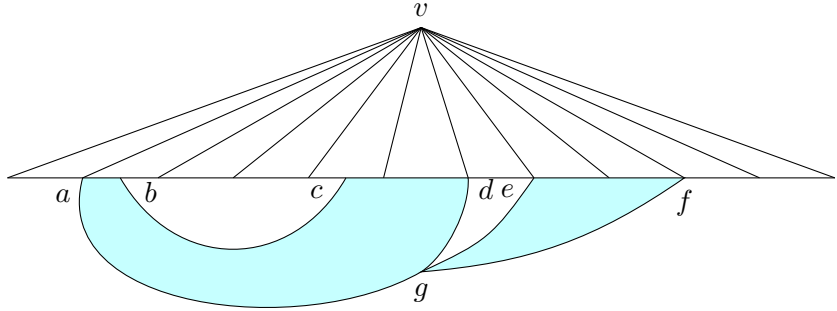

    We define $\mathcal{A}_{max}(H)$ to be the set of arches in $\mathcal{A}(H)$ so that $A \in \mathcal{A}_{max}(H)$ if, for every $A' \in \mathcal{A}(H)$ with $A \neq A'$, $A \nsubseteq \text{Int}(C_{A'}, \Pi)$. We call the arches in $\mathcal{A}_{max}(H)$ the \textit{maximal arches} of $H$. Remark that, by construction, every arch in $\mathcal{A}_{max}(H)$ is empty below. Remark moreover that they are naturally ordered along $P$.

    Let us first modify $\mathcal{A}_{max}(H)$ to ensure that the $\Pi$-contractible cycles induced by the arches in $\mathcal{A}_{max}(H)$ do not intersect. To do so, we proceed by order of the arches in $\mathcal{A}_{max}(H)$ along $P$: for the arch $A$ in $\mathcal{A}_{max}(H)$ that is leftmost on $P$ and has not yet been treated, we remove from $\mathcal{A}_{max}(H)$ every arch $A'$ so that $\text{int}(C_A, \Pi) \cap \text{int}(C_{A'},\Pi) \neq \emptyset$. Then, remark that two maximal arches are consecutive in this order if and only if they intersect.

    Let us now modify $\mathcal{A}_{max}(H)$ so that no maximal arches of size at most $2$ are adjacent by merging consecutive arches (this can also be done by order along $P$). See \cref{fig:maximal_arches} for an illustration.

    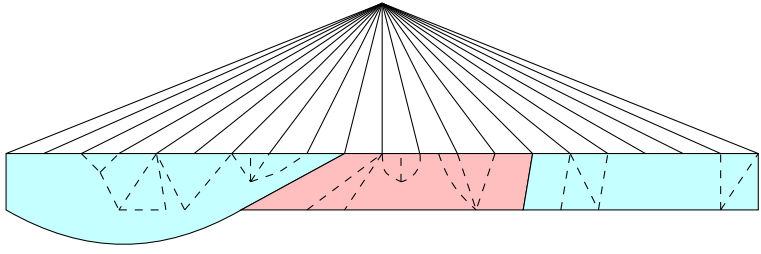
\begin{figure}[h!]
    \centering
    \begin{tikzpicture}
	\begin{pgfonlayer}{nodelayer}
		\node [style=none] (0) at (0, 3.5) {};
		\node [style=none] (1) at (-4, -0.5) {};
		\node [style=none] (2) at (-3, -0.5) {};
		\node [style=none] (3) at (-2, -0.5) {};
		\node [style=none] (4) at (-1, -0.5) {};
		\node [style=none] (5) at (0, -0.5) {};
		\node [style=none] (6) at (1, -0.5) {};
		\node [style=none] (7) at (2, -0.5) {};
		\node [style=none] (8) at (3, -0.5) {};
		\node [style=none] (9) at (4, -0.5) {};
		\node [style=none] (10) at (-5, -0.5) {};
		\node [style=none] (11) at (5, -0.5) {};
		\node [style=none] (12) at (6, -0.5) {};
		\node [style=none] (13) at (-6, -0.5) {};
		\node [style=none] (14) at (-7, -2) {};
		\node [style=none] (15) at (-5.75, -2) {};
		\node [style=none] (16) at (-3.75, -2) {};
		\node [style=none] (17) at (-5.25, -2) {};
		\node [style=none] (18) at (-5.25, -2) {};
		\node [style=none] (19) at (-3.5, -1.25) {};
		\node [style=none] (20) at (-3.5, -0.5) {};
		\node [style=none] (21) at (-2.75, -1) {};
		\node [style=none] (22) at (-2, -2) {};
		\node [style=none] (23) at (-1, -2) {};
		\node [style=none] (24) at (2.5, -2) {};
		\node [style=none] (25) at (0.5, -0.5) {};
		\node [style=none] (26) at (0.5, -1.25) {};
		\node [style=none] (27) at (1, -0.5) {};
		\node [style=none] (28) at (1.5, -0.5) {};
		\node [style=none] (29) at (3.75, -2) {};
		\node [style=none] (30) at (4.75, -2) {};
		\node [style=none] (31) at (5.75, -2) {};
		\node [style=none] (32) at (-7, -0.5) {};
		\node [style=none] (33) at (-8, -0.5) {};
		\node [style=none] (34) at (-9, -0.5) {};
		\node [style=none] (35) at (-10, -0.5) {};
		\node [style=none] (36) at (7, -0.5) {};
		\node [style=none] (37) at (8, -0.5) {};
		\node [style=none] (38) at (9, -0.5) {};
		\node [style=none] (39) at (10, -0.5) {};
		\node [style=none] (40) at (-10, -2) {};
		\node [style=none] (41) at (9, -2) {};
		\node [style=none] (42) at (10, -2) {};
		\node [style=none] (43) at (-7.5, -1) {};
		\node [style=none] (44) at (-7, -3) {};
		\node [style=none] (45) at (-7, -3) {};
		\node [style=none] (46) at (8.5, -1.25) {};
	\end{pgfonlayer}
	\begin{pgfonlayer}{edgelayer}
		\draw (0.center) to (1.center);
		\draw (0.center) to (2.center);
		\draw (0.center) to (3.center);
		\draw (0.center) to (4.center);
		\draw (0.center) to (5.center);
		\draw (0.center) to (6.center);
		\draw (0.center) to (7.center);
		\draw (0.center) to (8.center);
		\draw (0.center) to (9.center);
		\draw [style=new edge style 4] (40.center)
			 to (35.center)
			 to (34.center)
			 to (33.center)
			 to (32.center)
			 to (13.center)
			 to (10.center)
			 to (1.center)
			 to (2.center)
			 to (3.center)
			 to (4.center)
			 to (16.center)
			 to [bend left=15] (44.center)
			 to [bend right=345, looseness=0.75] cycle;
		\draw [style=new edge style 6] (4.center)
			 to (16.center)
			 to (22.center)
			 to (23.center)
			 to (24.center)
			 to (29.center)
			 to (9.center)
			 to (8.center)
			 to (7.center)
			 to (6.center)
			 to (5.center)
			 to cycle;
		\draw [style=new edge style 3] (14.center) to (15.center);
		\draw [style=new edge style 3] (22.center) to (5.center);
		\draw [style=new edge style 6] (37.center)
			 to (38.center)
			 to (39.center)
			 to (42.center)
			 to (41.center)
			 to [bend left, looseness=0.75] (12.center)
			 to (36.center)
			 to cycle;
		\draw [style=new edge style 4] (9.center)
			 to (11.center)
			 to (12.center)
			 to [bend right, looseness=0.75] (41.center)
			 to (31.center)
			 to (30.center)
			 to (29.center)
			 to cycle;
		\draw [style=new edge style 3] (30.center) to (11.center);
		\draw [style=new edge style 3] (31.center) to (12.center);
		\draw (10.center) to (0.center);
		\draw (13.center) to (0.center);
		\draw (0.center) to (11.center);
		\draw (0.center) to (12.center);
		\draw [style=new edge style 3] (15.center) to (13.center);
		\draw [style=new edge style 3] (18.center) to (13.center);
		\draw [style=new edge style 3] (18.center) to (1.center);
		\draw [style=new edge style 3] (1.center) to (19.center);
		\draw [style=new edge style 3] (19.center) to (20.center);
		\draw [style=new edge style 3] (19.center) to (2.center);
		\draw [style=new edge style 3] (19.center)
			 to (21.center)
			 to (3.center);
		\draw [style=new edge style 3] (5.center) to (23.center);
		\draw [style=new edge style 3, bend right, looseness=1.50] (5.center) to (26.center);
		\draw [style=new edge style 3] (26.center) to (25.center);
		\draw [style=new edge style 3, bend left=45] (27.center) to (26.center);
		\draw [style=new edge style 3, bend right=15] (28.center) to (24.center);
		\draw [style=new edge style 3] (7.center) to (24.center);
		\draw [style=new edge style 3] (8.center) to (24.center);
		\draw [style=new edge style 3] (11.center) to (31.center);
		\draw (0.center) to (32.center);
		\draw (0.center) to (33.center);
		\draw (0.center) to (34.center);
		\draw (0.center) to (35.center);
		\draw (0.center) to (36.center);
		\draw (0.center) to (37.center);
		\draw (0.center) to (38.center);
		\draw (0.center) to (39.center);
		\draw [style=new edge style 3] (41.center) to (38.center);
		\draw [style=new edge style 3] (41.center) to (39.center);
		\draw [style=new edge style 3] (13.center) to (14.center);
		\draw [style=new edge style 3] (33.center) to (43.center);
		\draw [style=new edge style 3] (43.center) to (32.center);
		\draw [style=new edge style 3] (43.center) to (14.center);
		\draw [style=new edge style 3] (45.center) to (14.center);
		\draw [style=new edge style 3] (45.center) to (18.center);
		\draw [style=new edge style 3] (37.center) to (46.center);
		\draw [style=new edge style 3] (46.center) to (38.center);
		\draw [style=new edge style 3] (46.center) to (41.center);
		\draw [style=new edge style 3] (18.center) to (10.center);
		\draw [style=new edge style 3, in=180, out=0] (15.center) to (18.center);
	\end{pgfonlayer}
\end{tikzpicture}

    \caption{Illustration for maximal arches. The edges inside a maximal arch are represented in dotted. The fan $H$ represented has $3$ maximal arches, represented in alternating blue and pink.} 
    \label{fig:maximal_arches}
    \end{figure}

    For the purpose of initialization, we define all the notions defined for a fan (with or without an arch) for a piece $p$: We define the \textit{horizontal path} and the \textit{vertical paths} to be the piece $p$ and therefore $\mathcal{F}(H)$ to be the set of faces that touches $p$. We define the notions of \textit{arched faces}, \textit{arch on a face}, \textit{arch on two faces}, etc., as for the general case. 
    We define a \textit{column} of $H$ to be an edge adjacent to $p$ if $p$ is a vertex and a vertex of $p$ if $p$ is a face.

    \setcounter{claim}{0}

    \begin{sublem}
        \label{face_no_multiple_columns}
         Let $p$ be a piece of $(G_1, \Pi)$ and let $H$ be a fan (with an arch) from $p$ in $(G_1,\Pi)$ or $H = p$. Let $P$ be the horizontal path of $H$. Let $f \in \mathcal{F}(H) \cup \mathcal{F}_{arch}(H)$ ($\mathcal{F}_{arch}(H) = \emptyset$ if $H$ is a fan without an arch or $H = p$). Then, $f$ touches at most $2\tilde{m}+1$ columns of $H$.
    \end{sublem}

    \begin{proof}
        If $H = p$, then the face $f$ cannot intersect $3$ columns of $p$, as otherwise, $p$ would either be a cutvertex of $G$ which is in contradiction with the fact that $G$ is $2$-connected or have a $2$-separator that separates a $\Pi$-contractible subgraph from the rest of the graph which is in contradiction with \cref{2_separated_subgraph_is_edge}.
        
        Suppose that $H$ is a fan (with an arch). Suppose that there is a face $f \in \mathcal{F}(H) \cup \mathcal{F}_{arch}(H)$ that touches at least $\tilde{2m}+2$ columns of $H$. Then, $f$ together with $H$ forms $\tilde{m}+1$ nested cycles ((vertex, face) if $p$ is a vertex or (face, face) if $p$ is a face). However, this contradicts \cref{loglog_bound_G0}.
    \end{proof}

    \begin{sublem}
        \label{touch_piece}
        Let $H$ be a fan $p$ in $(G_1,\Pi)$ of size at least $M$. Then, there exists a subset of consecutive faces from $\mathcal{F}(H)$ that does not touch $p \cup C_1$ and spans at least $\frac{M}{4A(g) (2\tilde{m}+1)^4\tilde{m}^2}$ columns.
    \end{sublem}

    \begin{proof}
        First, let $\mathcal{F}_{C_1}(H)$ be the set of faces in $\mathcal{F}(H)$ that touches $C_1$. We will show that it intersects at most $A(g) \times (2\tilde{m}+1)^3$ columns.

        Suppose that $\mathcal{F}_{C_1}(H)$ intersects more than $A(g) \times (2\tilde{m}+1)^3$ columns. By \cref{face_no_multiple_columns}, each face of $\mathcal{F}_{C_1}(H)$ intersect at most $2\tilde{m}+1$ columns. Therefore, there are more than $A(g) \times (2\tilde{m}+1)^2$ faces in $\mathcal{F}_{C_1}(H)$ that intersect pairwise disjoint columns. 

        Now, by the pigeon hole principle, there is a subpath $P_0$ of $C_1$ that corresponds to a facial subwalk of $(G_1,\Pi)$ so that more than $(2\tilde{m}+1)^2$ faces of $\mathcal{F}_{C_1}(H)$ that intersect pairwise disjoint columns touch $P_0$.

        However, by \cref{loglog_bound_G0}, there are at most $2\tilde{m}+1$ of them that touch the same vertex of $P_0$ and at most $2\tilde{m}+1$ of them that touch pairwise disjoint vertices of $P_0$—a contradiction.

        We conclude that the faces in $\mathcal{F}_{C_1}(H)$ intersect at most $A(g) \times (2\tilde{m}+1)^3$ columns. We can therefore restrict to a fan $H'$ that spans at least $\frac{M}{A(g) \times (2\tilde{m}+1)^3}$ columns and so that $\mathcal{F}(H')$ does not touch $C_1$.

        \vspace{1em}

        Next, let $\mathcal{F}_p(H')$ be the set of faces in $\mathcal{F}(H')$ that touches $p$. We will show that it intersects at most $4(2\tilde{m}+1)\tilde{m}^2$ columns.

        Suppose that $\mathcal{F}_p(H')$ intersects more than $4(2\tilde{m}+1)\tilde{m}^2$ columns. By \cref{face_no_multiple_columns}, each face of $\mathcal{F}_p(H')$ intersect at most $2\tilde{m}+1$ columns. Therefore, there are more than $4 \tilde{m}^2$ faces in $\mathcal{F}_p(H')$ that intersect pairwise disjoint columns. 

        Now, if there are more than $2\tilde{m}$ of them that touch $p$ in the same vertex, then remark that this induces $\tilde{m}+1$ $\Pi$-well-nested cycles. Therefore, there are more than $2 \tilde{m}$ of them that touch pairwise disjoint vertices of $p$ (in that case $p$ must be a face). However, if there are more than $2\tilde{m}$ of them that touch pairwise disjoint vertices of $p$, then this induces $\tilde{m}+1$ $\Pi$-well-nested cycles—a contradiction.

        We conclude that the faces in $\mathcal{F}_p(H')$ intersect at most $4 (2\tilde{m}+1) \tilde{m}^2$ columns. We can therefore restrict to a fan $H''$ that spans at least $\frac{M}{4A(g)(2\tilde{m}+1)^4\tilde{m}^2}$ columns and so that $\mathcal{F}(H'')$ does not touch $C_1$ nor $p$.
    \end{proof}

    \begin{sublem}
        \label{touch_arch}
        Let $H$ be a fan with an arch from $p$ in $(G_1,\Pi)$ of size at least $M$. Then, there exists a subset of consecutive faces from $\mathcal{F}(H)$ that spans at least $\frac{M}{(2\tilde{m}+1)^2}$ columns.
    \end{sublem}

    \begin{proof}
        Let $P$ and $P_{arch}$ be respectively the horizontal path and the arch path of $H$. By definition of $\mathcal{F}(H)$, no face in $\mathcal{F}(H)$ touches $P_{arch}$, but the faces in $\mathcal{F}(H)$ may not be consecutive on $P$.

        First, by \cref{face_no_multiple_columns}, each face of $\mathcal{F}_{arch}(H)$ intersect at most $2\tilde{m}+1$ columns.

        If $P_{arch}$ has no intersection point, then $\mathcal{F}_{arch}(H)$ contains only one face and $\mathcal{F}_{arch}(H)$ intersect at most $2\tilde{m}+1$ columns in total.
        
        Suppose $P_{arch}$ has $1$ intersection point. Suppose that the faces in $\mathcal{F}_{arch}(H)$ touch at least $(2\tilde{m}+1)^2+1$ columns, then it contains at least $2\tilde{m}+2$ faces that touch pairwise disjoint columns. However, then there would be at least $\tilde{m}+1$ nested cycles (with $p$ and the intersection point of $P_{arch}$), and we would find a contradiction with \cref{loglog_bound_G0}. 
        
        Hence, in any case, $\mathcal{F}_{arch}(H)$ touches at most $(2\tilde{m}+1)^2$ disjoint columns.

        Finally, it is possible to find a subset of at least $\frac{M}{(2\tilde{m}+1)^2}$ consecutive columns that are touched by no face from $\mathcal{F}_{arch}(H)$. Therefore, consecutive faces from $\mathcal{F}(H)$ touch these consecutive columns.
    \end{proof}

    We define the following function for every $i \in \mathbb{N}$: \[f(g,i) = \left(4\sqrt{2 A(g) (2\tilde{m}+1)^4 \tilde{m}^3}\right)^{i^2}\]

    \begin{sublem}
        \label{new_fan}
        Let $1 \leq i \leq \tilde{m}$ and let $H$ be a fan (with an arch) from $p$ in $(G_1, \Pi)$ of size at least $\frac{f(g,i)-2}{4A(g)(2\tilde{m}+1)^4 \tilde{m}^2}$. 
        Then, there exists $1 \leq k \leq i$ so that there are at least $f(g,i-k)$ maximal arches from $\mathcal{A}_{max}(H)$ of size at least $f(g,k-1)+2$.
    \end{sublem}

    \begin{proof}
        Suppose, by contradiction, that, for every $1 \leq k \leq i$, there are fewer than $f(g,i-k)$ maximal arches from $\mathcal{A}_{max}(H)$ of size at least $f(g,k-1)+2$.

        Remark that $\mathcal{A}_{max}$ can be partitioned into $i+1$ sets of arches $\mathcal{A}_0, ..., \mathcal{A}_i$ so that, for $0 \leq k \leq i$, $A \in \mathcal{A}_{max}$ is in $\mathcal{A}_k$ if its size is in the range $f(g,k-1)+2$ to $f(g,k)+1$ (with $f(g,-1) = -1$). Remark, moreover, that, by hypothesis, $\mathcal{A}_i$ is empty.

        If $i=0$, then $\mathcal{A}_{max}$ only contains maximal arches of size at most $2$. However, this contradicts the definition of $\mathcal{A}_{max}$. Therefore, we can suppose that $i \geq 1$.
        
        Therefore, 
        
        \begin{align*}
            \sum_{A \in \mathcal{A}_{max}} |A| & = \sum_{k=0}^{i-1} \sum_{A \in \mathcal{A}_k} |A| \\
            & = \sum_{A \in \mathcal{A}_0} |A| + \sum_{k=
        1}^{i-1} \sum_{A \in \mathcal{A}_k} |A| \\
            & \leq 4\sum_{k=1}^{i-1} f(g,k) +\sum_{k=1}^{i-1} (f(g,i-k)-1) \times (f(g,k)+1) \\
            & < 4\sum_{k=1}^{i-1} f(g,k) +\sum_{k=1}^{i-1} f(g,i-k) \times f(g,k)
        \end{align*}

        As, for $1 \leq k \leq i-1$, 
        
        \begin{align*}
             f(g,i-k) \times f(g,k) &=  \left(4\sqrt{2A(g) (2\tilde{m}+1)^4 \tilde{m}^3}\right)^{(i-k)^2+k^2} \\
            & = \left(4\sqrt{2A(g)(2\tilde{m}+1)^4\tilde{m}^3}\right)^{i^2+2k(k-i)} \\
            & \leq \left(4\sqrt{2A(g)(2\tilde{m}+1)^4\tilde{m}^3}\right)^{i^2-2} = \frac{f(g,i)}{32A(g)(2\tilde{m}+1)^4\tilde{m}^3}
        \end{align*}

        and

        \begin{align*}
            4\sum_{k=1}^{i-1} f(g,k) &\leq 4 \times (i-1) \times f(g,i-1) \\
            & \leq (i-1)\times 4 \times \left(4\sqrt{2A(g)(2\tilde{m}+1)^4\tilde{m}^3}\right)^{i^2-2} = (i-1)\times \frac{4f(g,i)}{32A(g)(2\tilde{m}+1)^4 \tilde{m}^3}
        \end{align*}
        
        Then 

        \begin{align*}
            \sum_{A \in \mathcal{A}_{max}} |A| &< (i-1) \times \left(\frac{4f(g,i)}{32A(g) (2\tilde{m}+1)^4\tilde{m}^3} + \frac{f(g,i)}{32A(g) (2\tilde{m}+1)^4 \tilde{m}^3} \right) \\
            & \leq (i-1) \frac{f(g,i)-2}{4A(g) (2\tilde{m}+1)^4 \tilde{m}^3} \\
            & < \frac{f(g,i)-2}{4A(g)(2\tilde{m}+1)^4\tilde{m}^2}
        \end{align*}

        However, as the maximal arches of $H$ cover all the horizontal path of $H$, $\sum_{A \in \mathcal{A}_{max}} |A| \geq \frac{f(g,i)-2}{4A(g) (2\tilde{m}+1)^4 \tilde{m}^2}$, a contradiction.
    \end{proof}
    
    \begin{sublem}
    \label{big_degree_nested_cycles}
        If there exists a piece $p$ of $(G_1, \Pi) $ such that $d_{G_1}(p) \geq \Delta(g) = \left(4\sqrt{2 A(g) (2\tilde{m}+1)^4 \tilde{m}^3}\right)^{\tilde{m}^2}$, then $G_1$ contains $\tilde{m}+1$ cycles that are $\Pi$-contractible well nested pinched on $p$.
    \end{sublem}

    \begin{proof}
        Let $p$ be a piece of $(G_1, \Pi)$ of degree or size at least $\Delta(g)$.

        Let's prove simultaneously the two following properties by induction on $i$ with $0 \leq i \leq m$:
        \begin{enumerate}
            \item If there is a fan $H$ from $p$ of size $f(g,i)$ in $G_1$, then $G_1$ contains $i$ $\Pi$-well-nested cycles with a column of $H$ strictly contained in the innermost nested cycle.
            \item If there is a fan with an arch $H$ from $p$ of size $f(g,i)$ in $G_1$, then there are $i$ $\Pi$-well-nested cycles in $\text{int}(H, \Pi)$ with a column of $H$ strictly contained in the innermost nested cycle. 
        \end{enumerate}

        First, for $i=0$, both properties are trivial.

        Then, let $1 \leq i \leq m$ and suppose that both properties are valid for every $k < i$. Let's prove each property for $i$.

    \begin{enumerate}
        \item \underline{Suppose that there is a fan $H$ from $p$ of size $f(g,i)$ in $(G_1,\Pi)$.}
        
        \underline{Let's prove property (1) for $i$.} 
        
        Let $P$ be the horizontal path of $H$.  

        Remove the first and last column of the fan $H$; therefore, no face in $\mathcal{F}(H)$ touches the first or last vertical path of $H$ (except potentially in $p$).

        By \cref{touch_piece}, there exists at least $\frac{f(g,i)-2}{4A(g) (2\tilde{m}+1)^4 \tilde{m}^2}$ consecutive columns of $H$ (let's call the associated fan $H'$, and let $P'$ be its horizontal path) so that the arches in $\mathcal{A}(H')$ do not touch either $p$ or $C_1$.

        By \cref{new_fan}, there exists $1 \leq k \leq i$ so that there are at least $f(g,i-k)$ maximal arches from $\mathcal{A}_{max}(H')$ of size at least $f(g,k-1)+2$. Remark then that we can construct a fan $H''$ of size at least $f(g,i-k)$ so that each maximal arch of size at least $f(g,k-1)+2$ is contained in one of its columns. Remark also that, even though a maximal arch $A$ has size at least $f(g,k-1)+2$, the column of a fan $H''$ containing $A$ may not contain all the columns of $H'$ that $A$ touches. It, however, contains at least $f(g,k-1)$ of them.

        By induction hypothesis on the property (1) for $H''$, $G_1$ contains $i-k$ $\Pi$-well-nested cycles with a column of $H''$ strictly contained in the innermost nested cycle. Furthermore, by construction, this column contains an arch $A$ from $\mathcal{A}_{max}(H')$ (hence empty below) of size at least $f(g,k-1)+2$. Let $H_A$ be the fan with an arch induced by $H'$ whose arch path is $A$. 
        The fan $H_A$ has a size of at least $f(g,k-1)$. Therefore, by induction hypothesis on the property (2) for $H_A$, $G_1$ contains $k-1$ $\Pi$-well-nested cycles in $\text{int}(H_A, \Pi)$ with a column of $H_A$ (which corresponds to a column of $H$) strictly contained in the innermost nested cycle.

        Finally, $G_1$ contains $i$ $\Pi$-well-nested cycles with a column of $H$ strictly contained in the innermost nested cycle: the $i-k$ nested cycles found in $H''$, the cycle associated to $A$, and the $k-1$ nested cycles found in $H_A$.

        \item \underline{Suppose that there is a fan with an arch $H$ from $p$ of size $f(g,i)$ in $(G_1,\Pi)$.}
        
        \underline{Let's prove property (2) for $i$.} 
        
        Let $P$ be the horizontal path of $H$ and $P_{Arch}$ be the arch path of $H$.
        

        By \cref{touch_arch}, we can restrict $\mathcal{F}(H)$ to a subset of consecutive faces from $\mathcal{F}(H)$ that spans at least $\frac{f(g,i)}{(2\tilde{m}+1)^2} \geq \frac{f(g,i)-2}{4A(g)(2\tilde{m}+1)^4 \tilde{m}^2}$ columns so that this subset of faces does not touch $P_{arch}$ (let's call the associated fan $H'$, let $P'$ be its horizontal path). 

        By \cref{new_fan}, there exists $1 \leq k \leq i$ so that there are at least $f(g,i-k)$ maximal arches from $\mathcal{A}_{max}(H')$ of size at least $f(g,k-1)+2$. Remark then that we can construct a fan with an arch $H''$, with arch path $P_{Arch}$, of size at least $f(g,i-k)$ so that each maximal arch of size at least $f(g,k-1)+2$ is contained in one of its columns. Remark that, even though each maximal arch $A$ has size at least $f(g,k-1)+2$, the column of a fan $H''$ containing $A$ may not contain all the columns of $H'$ that $A$ touches. It, however, contains $f(g,k-1)$ of them.

        By induction hypothesis on the property (2) for $H''$, $G_1$ contains $i-k$ $\Pi$-well-nested cycles in $\text{int}(H, \Pi)$ with a column of $H''$ strictly contained in the innermost nested cycle. Furthermore, by construction, this column contains an arch $A$ from $\mathcal{A}_{max}(H')$ (hence empty below) of size at least $f(g,k-1)+2$. Let $H_A$ be the fan with an arch induced by $H'$ whose arch path is $A$. It has a size of at least $f(g,k-1)$. Therefore, by induction hypothesis on the property (2) for $H_A$, $G_1$ contains $k-1$ $\Pi$-well-nested cycles in $\text{int}(H_A, \Pi)$ with a column of $H_A$ (which corresponds to a column of $H$) strictly contained in the innermost nested cycle.

        Finally, $G_1$ contains $i$ $\Pi$-well-nested cycles in $\text{int}(H, \Pi)$ with a column of $H$ strictly contained in the innermost nested cycle: the $i-k$ nested cycles found in $H''$, the cycle associated to $A$, and the $k-1$ nested cycles found in $H_A$.
    \end{enumerate}
   
    This concludes the induction.

    \vspace{1em}

    By hypothesis, $d_{G_1}(p) > \Delta(g)$. Therefore, $H = p$ is a fan of size $f(g,m)$ in $G_1$.

    Let's apply the property (1) to $H$ with $i=\tilde{m}$. We obtain that $G_1$ contains $\tilde{m}$ $\Pi$-well-nested cycles pinched on $p$ with a column of $H$ strictly contained in the innermost nested cycle in $(G_1,\Pi)$. Finally, by taking the cycle on the boundary of this column as the last nested cycle, $G_1$ contains $\tilde{m}+1$ $\Pi$-well-nested cycles pinched on $p$ in $(G_1, \Pi)$. 
\end{proof}
    
    Suppose either $\Delta(G_1) > \Delta(g)$ or $\Delta_F(G_1, \Pi) > \Delta(g)$. Then there exists a piece $p$ of $(G_1, \Pi)$ with $d_{G_1}(p) > \Delta(g)$. 
    
    Then, by \cref{big_degree_nested_cycles}, $(G_1, \Pi)$ contains $\tilde{m}+1$ $\Pi$-well-nested cycles pinched on $p$. However, this contradicts \cref{loglog_bound_G0}. Finally, $d_{G_1}(p) \leq \Delta(g)$ and therefore $\Delta(G_1) \leq \Delta(g)$ and $\Delta_F(G_1, \Pi) \leq \Delta(g)$.
\end{proof}

\subsection{Bound on the order of \texorpdfstring{$G_1$}{G1}} \label{subsec:bound_on_planar_subgraph_logarithmic_separator}

We show that the order of $G_1$ is bounded by the sub-polynomial function \[P(g) = \frac{\Delta(g)(\Delta(g)^{2\tilde{m}}-1)}{\Delta(g)-1}\times A(g)\] with the function $\Delta$ and $A$ being defined in \cref{subsec:max_degree}.

\vspace{1em}

\begin{defi}[Radius]
    Let $H$ be a graph with embedding $\Pi_H$ in a surface.
    Let $C$ be a $\Pi_H$-contractible cycle of $H$ and let $G_C = \text{Int}(C, \Pi_H)$. We define the \textit{radius} $rad(f)$ of a face $f$ in $\mathcal{F}(C, \Pi_H)$ inductively as follow:
    The faces that share a vertex with $V(C)$ have radius $1$. For $i \geq 2$, suppose we have already determined the faces of radius $1$ to $i$. Then, the faces of radius $i+1$ are those that are not of radius at most $i$ but share a vertex with the faces of radius $i$.
    We say that $G_C$ has \textit{radius $k$} if $\max\{ rad(f) | f \in \mathcal{F}(C, \Pi_H)\} = k$.

    An illustration can be found in \cref{fig:radius_graph}.


\end{defi}

\begin{figure}[h!]
    \centering
    \input{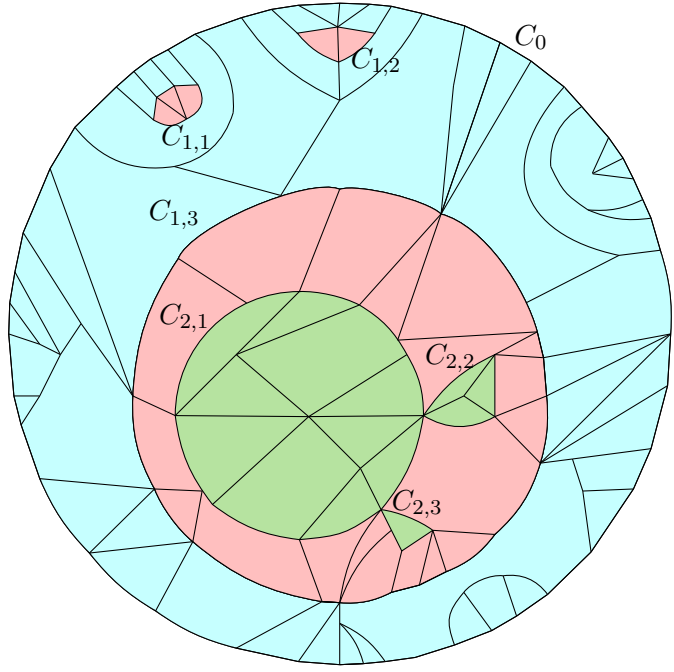}
    \caption{Radius of a contractible graph. The three colors indicate the radius of the faces with respect to the outer cycle $C_0$: faces in blue, pink, and green are respectively of radius $1$, $2$, and $3$.}
    \label{fig:radius_graph}
\end{figure}

Let $C$ be a $\Pi$-contractible cycle. We define $\mathcal{B}(C)$ to be the faces in $\text{Int}(C,\Pi)$ that touch $C$.

\begin{prop}
    \label{order_planar_subgraph_logarithmic_separator}
    \[|V(G_1)| \leq P(g) \]
\end{prop}

\begin{proof}
    For a cycle $C$ in $G_1$, let $E(C)^+$ be the set of edges in $\text{int}(C,\Pi)$ incident to $C$. By \cref{max_degree}, every vertex of $C$ is adjacent to at most $\Delta(g)$ edges in $G_1$, therefore $|E(C)^+| \leq \Delta(g) \times |V(C)|$. Let $C$ be a cycle in $G_1$ and let $\mathcal{C}(C)$ be the cycles so that $\mathcal{B}(C)$ contains exactly the faces inside $\text{Int}(C, \Pi) - \bigcup_{C' \in \mathcal{C}(C)} \text{Int}(C', \Pi)$.

    \setcounter{claim}{0}

    \begin{claim}
        \label{size_cycles}
         Let $C$ be a cycle in $G_1$. \[\sum_{C' \in \mathcal{C}(C)} |V(C')| \leq \Delta(g)^2 \times |V(C)|\]
    \end{claim}

    \begin{proof_claim}
        Let $H_C$ be the graph containing all the bridges inside $\text{int}(C, \Pi) - \bigcup_{C' \in \mathcal{C}(C)} \text{Int}(C', \Pi)$. Let $A(C)$ be the set of vertices in $H_C$ that have attachments on some cycle in $\mathcal{C}(C)$.

        Let's first prove that $A(C)$ is of order at most $\Delta(g) \times |V(C)|$. Remark that, by construction, there are fewer edges with an endpoint in $\bigcup_{C'\in \mathcal{C}(C)} C'$ in $H_C$ than to $C$ in $H_C$. Hence, we have $|A(C)| \leq |E(C)^+| \leq \Delta(g) \times |V(C)|$.
            
        Let $P$ be a path on a cycle $C'$ of $\mathcal{C}(C)$ between two consecutive attached vertices $a, a' \in A(C')$ in $H_C$. Then, a $\Pi$-facial walk contains $P$. Therefore, by \cref{max_degree}, $|P| \leq \Delta(g)$.

        Moreover, as $|A(C)| \leq \Delta(g) \times |V(C)|$, $\bigcup_{C' \in \mathcal{C}(C)} C'$ contains at most $\Delta(g) \times |V(C)|$ such paths. Finally, $\sum_{C' \in \mathcal{C}(C)} |V(C')| \leq \Delta(g)^2 \times |V(C)|$.
    \end{proof_claim}
	
	\vspace{0.3cm}

    By \cref{loglog_bound_G0}, the radius of $G_1$ with embedding $\Pi(G_1)$ is at most $\tilde{m}$.
    Let's partition the faces of $(G_1, \Pi(G_1))$ into their radius class with respect to the outer face boundary $C_1$: for $0 \leq i \leq \tilde{m}-1$, let $\mathcal{F}_i$ be the set of faces of $(G_1, \Pi(G_1))$ of radius $i+1$. Let's define inductively a set of cycles $\mathcal{C}_i$ so that $\bigcup_{C \in \mathcal{C}_i} \mathcal{B}(C) = \mathcal{F}_i$ and $\sum_{C \in \mathcal{C}_i} |V(C)| \leq |V(C_1)| \times \Delta(g)^{2i}$.

    We define $\mathcal{C}_0 = \{C_1\}$. We have indeed that $\mathcal{B}(C_1) = \mathcal{F}_0$ and $|V(C_1)| = |V(C_1)| \times 1$.
    Now, suppose that for $0 \leq i \leq \tilde{m}-1$, $\mathcal{C}_i$ is defined and is so that $\bigcup_{C \in \mathcal{C}_i} \mathcal{B}(C) = \mathcal{F}_i$ and $\sum_{C \in \mathcal{C}_i} |V(C)| \leq |V(C_1)| \times \Delta(g)^{2i}$. Let's define $\mathcal{C}_{i+1}$. For $C \in \mathcal{C}_i$, let $\mathcal{C}(C)$ be the set of cycles so that $\mathcal{B}(C)$ contains exactly the faces inside $\text{Int}(C, \Pi) - \bigcup_{C' \in \mathcal{C}(C)} \text{Int}(C', \Pi)$. Then, by the \cref{size_cycles}, $\sum_{C' \in \mathcal{C}(C)} |V(C')| \leq \Delta(g)^2 \times |V(C)|$. We define $\mathcal{C}_{i+1} = \bigcup_{C \in \mathcal{C}_i} \bigcup_{C' \in \mathcal{C}(C)} C'$. 
    
    Then, $\sum_{C \in \mathcal{C}_{i+1}} |V(C)| = \sum_{C \in \mathcal{C}_i} \sum_{C' \in \mathcal{C}(C)} |V(C')| \leq \sum_{C \in \mathcal{C}_i} \Delta(g)^2 \times |V(C)| \leq |V(C_1)| \times \Delta(g)^{2i+2}$.
    Moreover, it is easy to verify that $\bigcup_{C \in C_{i+1}} \mathcal{B}(C)$ is exactly $\mathcal{F}_{i+1}$.

    In \cref{fig:radius_graph}, the graph represented has radius $3$, and the faces of radius $1$, $2$, and $3$ are respectively in blue, pink, and green. Moreover, we have $\mathcal{C}_0 = \{C_0\}$, $\mathcal{C}_1 = \{C_{1,1}, C_{1,2}, C_{1,3}\}$ and $\mathcal{C}_2 = \{C_{2,1}, C_{2,2}, C_{2,3}\}$.

    \vspace{0.3cm}

    Let $0 \leq i \leq \tilde{m}-2$.
    Remark that, by construction, the set $\mathcal{B}(C) \subseteq \mathcal{F}_{i}$ for $C \in \mathcal{C}_i$ induces a set of bridges $\mathcal{H}_C$ on $C$ and the cycles of $\mathcal{C}_{i+1}$. By construction, each bridge in $\mathcal{H}_C$ is a tree whose leaves are on $C$ and whose roots are on some cycles in $\mathcal{C}_{i+1}$ (two distinct leaves or roots of these trees can lie on the same vertex of $C$).
    For $i = \tilde{m}-1$ and $C \in \mathcal{C}_{\tilde{m}-1}$, $\mathcal{B}(C)$ induces a set of bridges $\mathcal{H}_C$, and these bridges are trees whose leaves are on $C$ (two distinct leaves of these trees can lie on the same vertex of $C$).

    \begin{claim}
        \label{size_layers}
        For $0 \leq i \leq \tilde{m}-1$ and $C \in \mathcal{C}_i$, $V(C) \cup \bigcup_{B \in \mathcal{H}_C} V(B)$ has size at most $(\Delta(g)+1)|V(C)|$.
    \end{claim}

    \begin{proof_claim}
        The trees in $\mathcal{H}_C$ have in total at most $|E(C)^+|$ leaves. 
        Moreover, by minimality of $G$, there is no induced path of length $2$ in any tree in $\mathcal{H}_C$. Therefore, these trees have more leaves than internal vertices.
        
        Finally, \[\sum_{B \in \mathcal{H}_C} |V(B)| \leq |E(C)^+|\] and \[|V(C)| + \sum_{B \in  \mathcal{H}_C} |V(B)| \leq |V(C)| + |E(C)^+| \leq (\Delta(g)+1) \times |V(C)|\]
    \end{proof_claim}

	\vspace{0.3cm} 

    Remark that the set of vertices of $C_1$ on which edges of $\text{Ext}(C_1,\Pi)$ have an endpoint is exactly $A_1$. Let $P$ be a path on $C_1$ between two consecutive attached vertices $a, a' \in A_1$, then there is a $\Pi$-facial walk that contains $P$. Therefore, by \cref{max_degree}, $|P| \leq \Delta(g)$. Moreover, as $|A_1| \leq A(g)$, $C_1$ contains at most $A(g)$ such paths. Finally, $C_1$ has size at most $A(g) \times \Delta(g)$.

    By using \cref{size_layers} and the inequality $\sum_{C \in \mathcal{C}_i} |V(C)| \leq |V(C_1)| \times \Delta(g)^{2i}$ for $0 \leq i \leq \tilde{m}-1$, we obtain that the order of $G_1$ is bounded by

    \begin{align*}
        \sum_{i=0}^{\tilde{m}-1} \sum_{C \in \mathcal{C}_i} |V(C \cup \bigcup_{B \in \mathcal{H}_C} B)| &= \sum_{i=0}^{\tilde{m}-1} \sum_{C \in \mathcal{C}_i} (\Delta(g)+1)|V(C)| \\
        & = \sum_{i=0}^{\tilde{m}-1} (\Delta(g)+1)\Delta(g)^{2i} |V(C_1)| \\
        & \leq (\Delta(g)+1) \frac{\Delta(g)^{2\tilde{m}} - 1}{\Delta(g)^2-1}|V(C_1)| \\
        & \leq \frac{\Delta(g)^{2\tilde{m}}-1}{\Delta(g)-1}|V(C_1)| \\
        & \leq \frac{\Delta(g)(\Delta(g)^{2\tilde{m}}-1)}{\Delta(g)-1}\times A(g).
    \end{align*}

    Finally, $G_1$ is of order at most $\frac{\Delta(g)(\Delta(g)^{2\tilde{m}}-1)}{\Delta(g)-1}\times A(g) = P(g)$.
\end{proof}

\section{Main proof} \label{sec:main_proof}

In this section, we present the main result of this paper: the order of $G$ is bounded by an almost quadratic function $U(g) = ag^2 \times (\log g)^{b(\log \log g)^3}$ for some constant $a$ and $b$. 

\vspace{0.3cm}

The section is organized as follows:
\begin{itemize}
    \item First, we present a few results about balanced separator decompositions that directly build on the work on balanced separators from Robertson and Seymour \cite{GM2} and tailor this tool to our needs.
    \item Then, we present the main proof of the paper through \cref{almost_main,main_planar}: in \cref{almost_main}, we first show that there is a subgraph $G_0$ of $G$ that is separated from the rest of $G$ by a separator $A$ of size $O(g \log^2 g)$ and so that every bridge on $A$ in $G_0$ is $\Pi$-contractible. Moreover, the order of $G$ is bounded by $|V(G_0)| \times Q(g)$ with $Q$ some nearly linear function of $g$. In \cref{main_planar}, for a bridge $B$ on $A$ in $G_0$ with a separator $A_B \subseteq A$ of size between $a$ and $2a-1$, we then find a subgraph $G_1$ of $B$ that is separated from the rest of $B$ by a separator of size $O(\log g)$, and so that the order of $G_0$ is bounded by $|V(G_1)| \times f(g,a)$ with $f$ some function of $g$ and $a$. Finally, we conclude the proof by applying \cref{order_planar_subgraph_logarithmic_separator} to show that the order of $G_1$ is subpolynomial in $g$ and linear in $a$.
    \item Finally, we prove \cref{main_2_connected} by combining \cref{almost_main,main_planar}: we partition the set of bridges on $A$ in $G_0$ into $O(\log g)$ sets so that the $i$-th set contains the bridges with a separator of size between $a_i = 3 \times 2^i$ and $2a_i -1 = 3 \times 2^{i+1} -1$. We then find the size of each of the sets in the partition and use the bound we found on the size of the bridges in \cref{main_planar} to bound the size of $G_0$. Finally, by combining this result with the equation $|V(G)| \leq |V(G_0)| \times Q(g)$ obtained in \cref{almost_main}, we obtain an almost quadratic bound on the order of $G$.
    We finish by extending the bound on the order of $G$ to the order of any excluded minor for the surface $S$, even if they are not $2$-connected.
\end{itemize}

\subsection{Balanced separator decompositions} \label{subsec:balanced_tree_decompo}

Using the celebrated balanced separator theorem from Robertson and Seymour \cite{GM2}, we find a balanced separator decomposition of any graph, so that the intersection of a set of the decomposition with the other sets is small (at most the product of the treewidth of the graph and a function logarithmic in the size of the decomposition).

\vspace{0.3cm}

For a tree $T$, we define a \textit{$1$-separation sequence} of $T$ to be a collection $(T_1,..., T_m)$ of subtrees of $T$ ($m \geq 2$) such that $T_1 \cup ... \cup T_m = T$ and $|T_i \cap T_j| \leq 1$ for all $i \neq j$.
In \cite{GM2}, Robertson and Seymour proved the following theorem, known as the balanced separator theorem:

\begin{theo}[Balanced separator theorem, {\cite[(2.6)]{GM2}}]
    \label{balanced_separator_RS}
    Let $H$ be a graph and let $(T, (V_t)_{t \in T})$ be a tree decomposition of $H$. Then, there is a $1$-separation $(T_1, T_2)$ of $T$ such that 
    \[\frac{1}{3} |V(H) - V_{t_0}| \leq |\bigcup_{t \in T_1} V_t - V_{t_0}| \leq \frac{2}{3} |V(H) - V_{t_0}|\]
    and 
    \[\frac{1}{3} |V(H) - V_{t_0}| \leq |\bigcup_{t \in T_2} V_t - V_{t_0}| \leq \frac{2}{3} |V(H) - V_{t_0}|\]
    with $T_1 \cap T_2 = \{t_0\}$.
\end{theo}

We first prove the following variant of \cref{balanced_separator_RS}: 

\begin{cor}
    \label{balanced_separator_RS_modified}
    Let $H$ be a graph and let $(T, (V_t)_{t \in T})$ be a tree decomposition of $H$. Suppose that $\max_{t \in T} |V_t| \leq \frac{1}{A} |V(H)|$ for some constant $A > 1$. Then, there is a $1$-separation $(T_1, T_2)$ of $T$ such that 
    \[\frac{1}{3} |V(H)| \leq |\bigcup_{t \in T_1} V_t| \leq (\frac{2}{3}+\frac{1}{3A}) |V(H)|\]
    and 
    \[\frac{1}{3} |V(H)| \leq |\bigcup_{t \in T_2} V_t| \leq  (\frac{2}{3}+\frac{1}{3A}) |V(H)|\]
\end{cor}

\begin{proof}
    By \cref{balanced_separator_RS}, we easily get that $|\bigcup_{t \in T_i} V_t| \geq \frac{1}{3} |V(H)|$ for $i =1,2$. Indeed, for $i =1,2$ and $T_1 \cap T_2 = \{t_0\}$, we have 
    \begin{align*}
        \frac{1}{3} |V(H) - V_{t_0}| &\leq |\bigcup_{t \in T_i} V_t - V_{t_0}| \\
        \frac{1}{3} (|V(H)| - |V_{t_0}|) & \leq |\bigcup_{t \in T_i} V_t| - |V_{t_0}| \\
        \frac{1}{3} |V(H)| + \frac{2}{3} |V_{t_0}| & \leq |\bigcup_{t \in T_i} V_t| \\
        \frac{1}{3} |V(H)| & \leq |\bigcup_{t \in T_i} V_t|
    \end{align*}

    Moreover, by using again \cref{balanced_separator_RS}, we get that $|\bigcup_{t \in T_i} V_t| \leq (\frac{2}{3}+\frac{1}{3A}) |V(H)|$ for $i =1,2$. Indeed, for $i =1,2$ and $T_1 \cap T_2 = \{t_0\}$, we have 
    \begin{align*}
        |\bigcup_{t \in T_i} V_t - V_{t_0}| & \leq \frac{2}{3} |V(H) - V_{t_0}| \\
        |\bigcup_{t \in T_i} V_t| - |V_{t_0}| & \leq \frac{2}{3} (|V(H)| - |V_{t_0}|) \\
        |\bigcup_{t \in T_i} V_t| & \leq \frac{2}{3} |V(H)| + \frac{1}{3}|V_{t_0}| \\
        |\bigcup_{t \in T_i} V_t| & \leq \frac{2}{3} |V(H)| + \frac{1}{3A} |V(H)| \\
        |\bigcup_{t \in T_i} V_t| & \leq (\frac{2}{3}+ \frac{1}{3A}) |V(H)| \\
    \end{align*}
\end{proof}

Finally, we extend \cref{balanced_separator_RS_modified} to a $1$-separation of $T$ of arbitrary size:

\begin{lem}
    \label{balanced_separation_tree_decomposition}
    Let $H$ be a graph and let $(T, (V_t)_{t \in T})$ be a tree decomposition of $H$. Suppose that $\max_{t \in T} |V_t| \leq \frac{1}{4k} |V(H)|$. For $k \in \mathbb{N}^*$, there is a $1$-separation sequence $(T_1,...,T_k)$ of $T$ such that, for every $1 \leq i, j \leq k$, 
    \begin{enumerate}
        \item $|\bigcup_{t \in T_i} V_t| \leq 3 |\bigcup_{t \in T_j} V_t|$,
        \item $V(T_i) \cap (\bigcup_{1 \leq j \neq i \leq k} V(T_j))$ has size at most $\lfloor\log_{\frac{4}{3}} 3k\rfloor$.
    \end{enumerate}

    \vspace{1em}

    We call $(\bigcup_{t \in T_1} V_t, ..., \bigcup_{t \in T_k} V_t)$ a balanced separator decomposition of $H$.
\end{lem}

\begin{proof}
    We prove the first property by induction on $k$.

    For $k=1$, this is clear.

    Now, suppose that the property is true for some $k \in \mathbb{N}^*$, let's prove that it is true for $k+1$.
    Let $H$ be a graph and let $(T, (V_t)_{t \in T})$ be a tree decomposition of $H$. By induction hypothesis, there is a $1$-separation sequence $(T_1, ..., T_k)$ of $T$ such that, for every $1 \leq i, j \leq k$, 
    $|\bigcup_{t \in T_i} V_t| \leq 3 |\bigcup_{t \in T_j} V_t|$ and such that, for $1 \leq i \leq k$, $V(T_i) \cap (\bigcup_{1 \leq j \neq i \leq k} V(T_j))$ has size at most $\lfloor\log_{\frac{4}{3}} 3k\rfloor$. Suppose that the order $T_1, ..., T_k$ is chosen so that $|\bigcup_{t \in T_i} V_t| \leq |\bigcup_{t \in T_j} V_t|$ if $i \leq j$.
    
    Then, remark that $\max_{t \in T_k} |V_t| \leq \max_{t \in T} |V_t| \leq \frac{1}{4k} |V(H)| \leq \frac{1}{4} |\bigcup_{t \in T_k} V_t|$.
    Consequently, by \cref{balanced_separator_RS_modified}, there is a $1$-separation $(T_k', T''_k)$ of $T_k$ such that \[\frac{1}{3} |\bigcup_{t \in T_k} V_t| \leq |\bigcup_{t \in T_k'} V_t| \leq \frac{3}{4} |\bigcup_{t \in T_k} V_t| \text{\quad and \quad} \frac{1}{3} |\bigcup_{t \in T_k} V_t| \leq |\bigcup_{t \in T_k''} V_t| \leq \frac{3}{4} |\bigcup_{t \in T_k} V_t|\]
    Let's show that the $1$-separation sequence $(T_1, T_2, ..., T_{k-1}, T'_k, T''_k)$ of $T$ has the needed property. It suffices to verify it for $T_i$ ($1 \leq i \leq k-1$) and $T_k'$ without loss of generality.
    We indeed have that \[ |\bigcup_{t \in T_i} V_t| \leq |\bigcup_{t \in T_k} V_t | \leq 3 |\bigcup_{t \in T'_k} V_t| \text{\quad and \quad} |\bigcup_{t \in T'_k} V_t| \leq |\bigcup_{t \in T_k} V_t| \leq 3 |\bigcup_{t \in T_i} V_t |\]
    
     This concludes the proof of the first property.

    \vspace{1em}

     Finally, we prove the second property. Suppose that there exists $1 \leq i \leq k$ such that \[ V(T_i) \cap (\bigcup_{1 \leq j \neq i \leq k} V(T_j)) \text{ has size } n > \lfloor\log_{\frac{4}{3}} 3k\rfloor\]
     
     Remark that by construction, $T_i$ is obtained after a sequence of $n$ $1$-separations: $T, U_1, U_2, U_3, ..., U_n=T_i$ where, for all $1 \leq i \leq n-1$, $U_{i+1} \subseteq U_i$ and $|\bigcup_{t \in U_{i+1}}V_t| \leq \frac{3}{4} |\bigcup_{t \in U_i}V_t|$ (according to \cref{balanced_separator_RS_modified}).

     Therefore, we get that \[|\bigcup_{t \in T_i} V_t| \leq \left(\frac{3}{4}\right)^n |\bigcup_{t \in T} V_t|  = \left(\frac{3}{4}\right)^n |V(H)|.\] Moreover, by the first property, we have that \[|\bigcup_{t \in T_i} V_t| \geq \frac{1}{3} |\bigcup_{t \in T_k} V_t| \text{\quad and \quad} k \times |\bigcup_{t \in T_k} V_t| \geq \sum_{j=1}^k |\bigcup_{t \in T_j} V_t| \geq |V(H)|\] Therefore, \[|\bigcup_{t \in T_i} V_t| \geq \frac{|V(H)|}{3k}\]
     
     Finally, we have $(\frac{3}{4})^n \geq \frac{1}{3k}$, which is a contradiction because $n > \lfloor\log_{\frac{4}{3}} 3k\rfloor$, and therefore $(\frac{3}{4})^n < \frac{1}{3k}$. 
\end{proof}

\subsection{Bound on the order of \texorpdfstring{$G$}{G}} \label{subsec:main_proof}

We show that the order of $G$ is bounded by a function $U(g) = ag^2 \times (\log \log g)^{b(\log \log g)^3}$ of the genus $g$ for some constants $a$ and $b$.

\vspace{1em}

We first prove that there is a subgraph $G_0$ of $G$ that is separated from the rest of $G$ by a separator $A$ of size $O(g \log^2 g)$ and so that every bridge on $A$ in $G_0$ is $\Pi$-contractible. Moreover, the order of $G$ is bounded by $|V(G_0)| \times Q(g)$ with $Q$ some polynomial function of $g$.

\vspace{1em}

We will use the following classical result for graphs on surfaces:

\begin{lem}[{\cite[Proposition 4.3.1]{graphs_on_surfaces}}]
    \label{4.3.1}
    Let $H$ be a graph, $\Pi_H$-embedded in a surface $S_H$. Let $x,y \in V(H)$ and $P_1,P_2,P_3$ be three internally disjoint paths from $x$ to $y$. Let $C_{i,j} = P_i \cup P_j$ for $1 \leq i< j \leq 3$.

    If two of the three cycles $C_{i,j}$ ($1 \leq i< j \leq 3$) are $\Pi_H$-contractible, then so is the third cycle.
\end{lem}

Let $q$ and $m$ be the values defined in \cref{good_square_cor_general}, that is to say $q = \frac{9073}{9072}$ and $m = 2(\lfloor\log_{q}(3g+4)\rfloor + 2)$. Moreover, let $T(g) = 264(g+2)(m+1)-1$.
In analyzing tree decompositions of $G$, we use the following bound on its treewidth, established in \cite{HK2026}:

\begin{theo}[\cite{HK2026}]
    \label{treewidth}
    The treewidth of $G$ is bounded by the following function of $g$:
    \[ tw(G) \leq T(g)\]
\end{theo}

Note that the constant $q$ defined in this paper and the one defined in \cite{HK2026} are not the same: indeed, the constant $q$ defined in \cite{HK2026} was higher, equal to $\frac{1153}{1152}$. We state the result \cref{treewidth} with $q = \frac{9073}{9072}$ (which is still true because, with a smaller $q$, $T(g)$ is bigger) to avoid the multiplication of constants in the paper.

\vspace{1em}

Let $m' = \lfloor\log_{\frac{4}{3}} 3(4m^2 (3g+3)+1)\rfloor$.

\begin{prop}
    \label{almost_main}
    Suppose that $|V(G)| \geq 4 (4m^2(3g+3)+1)(tw(G)+1)$.
    Then, there exists a separation $(G_0,G_1)$ of $G$ such that $A = V(G_0) \cap V(G_1)$ has size at most $(tw(G)+1) \times m'$, every bridge on $A$ in $G_0$ is $\Pi$-contractible and
    \[|V(G)| \leq Q(g) \times |V(G_0)|\]

    with $Q(g) = 3 (4m^2(3g+3)+1)$.
\end{prop}

\begin{proof}
    Let $(T,(V_t)_{t \in T})$ be a tree decomposition of $G$ of width $tw(G)$.
    By \cref{balanced_separation_tree_decomposition}, as $tw(G)+1 \leq \frac{1}{4(4m^2(3g+3)+1)}|V(G)|$, there is a separation $(T_1, ..., T_k)$ of $T$ with:
    \begin{itemize}
        \item $k = 4m^2(3g+3)+1$
        \item For every $1 \leq i, j \leq k$, $|\bigcup_{t \in T_i} V_t| \leq 3 |\bigcup_{t \in T_j} V_t|$
        \item For every $1 \leq i \leq k$, $V(T_i) \cap (\bigcup_{1 \leq j \neq i \leq k} V(T_j))$ has size at most $\lfloor\log_{\frac{4}{3}} 3k\rfloor = m'$
    \end{itemize}

    For $1 \leq i \leq k$, let $G_i= \bigcup_{t \in T_i} V_t$ and let $A_i = \bigcup_{j \neq i} V(G_i) \cap V(G_j)$. Moreover, for $1 \leq i \leq k$, let $G_i'$ be the graph with vertex set $V(G_i)$ and edge set all the edges in $E(G_i)$ with at least one endpoint in $V(G_i) - A_i$. 

    Let $1 \leq i \leq k$. We define $(H_i^j)_{1 \leq j \leq \ell_i}$ to be the connected components of $G_i - A_i$. Moreover, for $1 \leq j \leq \ell_i$, we define $\overline{H}_i^j$ to be the graph induced by $H_i^j$ and the vertices in $A_i$ that have a neighbor in $H_i^j$.

    Let's first prove the following claim on $\overline{H}_i^j$:

    \begin{claim}
        \label{Gi'_nice_non_contractible_cycle}
        Let $1 \leq i \leq k$ and $1 \leq j \leq \ell_i$. If the graph $\overline{H}_i^j$ contains a $\Pi$-noncontractible cycle, then it contains a $\Pi$-noncontractible cycle with at most one vertex in $A_i$.
    \end{claim}

    \begin{proof_claim}
        Let $1 \leq i \leq k$ and $1 \leq j \leq \ell_i$. Suppose that $\overline{H}_i^j$ contains a $\Pi$-noncontractible cycle $C$. Suppose, moreover, that we selected $C$ so that it has the fewest vertices in $A_i$.

        If $C$ contains at most one vertex in $A_i$, then there is nothing to prove. 
        
        Suppose therefore that $C$ contains at least two distinct vertices $u,v$ in $A_i$. 
        Let $P$ and $P'$ be the two distinct subpaths of $C$ from $u$ to $v$.
        
        As $H_i^j$ is connected, there is a path $P''$ from the interior of $P$ to the interior of $P'$ in $H_i^j$. Then $P'' \cup C$ generates two cycles $C'$ and $C''$ distinct from $C$ and, by \cref{4.3.1}, at least one of them (say $C'$) is $\Pi$-noncontractible. The cycle $C'$ has fewer vertices in $A_i$ than $C$, a contradiction.
    \end{proof_claim}

    \vspace{1em}

    Let's prove that there exists $1 \leq i \leq k$ so that, for every $1 \leq j \leq \ell_i$, $\overline{H}_i^j$ is $\Pi$-contractible. By contradiction, suppose that, for $1 \leq i \leq k$, there exists $1 \leq j_i \leq \ell_i$ so that the graph $\overline{H}_i^{j_i}$ contains a $\Pi$-noncontractible cycle. Then, by \cref{Gi'_nice_non_contractible_cycle}, there exists a cycle $C_i$ of $\overline{H}_i^{j_i}$ that is $\Pi$-noncontractible and contains at most one vertex in $A_i$. Remark that the cycles $(C_i)_{1 \leq i \leq k}$ are almost disjoint.

    By \cref{nb_non_contractible_cycles}, there are at most $4m^2 (3g+3)$ almost disjoint $\Pi$-noncontractible cycles in $(G, \Pi)$, which leads to a contradiction.
    Therefore, there exists $1 \leq i_0 \leq k$ so that, for every $1 \leq j \leq \ell_{i_0}$, $\overline{H}_{i_0}^j$ is $\Pi$-contractible.

    Remark that, for $1 \leq j \leq \ell_{i_0}$, the subgraph $\overline{H}_{i_0}^j$ is separated from the rest of the graph by a subset of vertices in $A_{i_0}$ and that, by definition, $A_{i_0}$ has size at most $(tw(G) +1)\times m'$.

    Finally, we obtain
    \begin{align*}
        |V(G)| & \leq \sum_{j=1}^k |\bigcup_{t \in T_j} V_t| \\
        & \leq 3k |\bigcup_{t \in T_{i_0}} V_t| \\
        & = 3 (4m^2(3g+3)+1) \times |V(G_{i_0}')| \\
        & = Q(g) \times |V(G_{i_0}')|
    \end{align*}

    By choosing $G_0 = G_{i_0}'$ and $G_1 = G - G_{i_0}'$, we indeed get a separation $(G_0,G_1)$ of $G$ such that $A = A_i = V(G_0) \cap V(G_1)$ has size at most $(tw(G)+1) \times m'$, and every bridge on $A$ in $G_0$ is $\Pi$-contractible with the needed relationship between $|V(G)|$ and $|V(G_0)|$.
\end{proof}

\vspace{1em}

Let $G_0$ and $A$ be the subgraph of $G$ and its separator in $G$ which were defined in \cref{almost_main}.
Now, we prove a bound on the order of each bridge $B$ on $A$ in $G_0$ with respect to $g$ and the size of its separator $A_B \subseteq A$ in $G_0$.

\vspace{1em}

\begin{prop}
    \label{main_planar}
    Let $(G_0,G_1)$ be a separation of $G$ such that $A = V(G_0) \cap V(G_1)$ has size at least $a$ and at most $2a-1$ with $3 \leq a \leq (tw(G)+1)\times m'$, $G_0$ is $\Pi$-contractible and $|V(G_0)| \geq 4 ((T(g)+1)\times m'+1) (12m+8)$. Then,
    \[|V(G_0)| \leq R(g,a) \]
    with \[A(g) = 6\lfloor\log_{\frac{4}{3}} 3((T(g)+1)\times m' +1)\rfloor \times (12m+8)\] and \[R(g,a) = 3 \times 2a \times \left(\frac{5}{6}A(g) P(g) + \frac{1}{6}A(g)\right)\]
\end{prop}

\begin{proof}
    By hypothesis, $G_0$ is contained in a disk in $\Pi$.
    Let $C_0$ be the cycle bounding $G_0$ in $\Pi$. Remark that $G_0 \subseteq \text{Int}(C_0, \Pi)$ and let $G_0' = \text{Int}(C_0, \Pi)$. Remark that, like $G_0$, $G_0'$ is also separated from the rest of $G$ by $A$. We will show that $|V(G_0')| \leq R(g)$, therefore we will have $|V(G_0)| \leq |V(G_0')| \leq R(g)$.

    According to a result of Robertson, Seymour, and Thomas {\cite[(6.2)]{RST1994}}, every planar graph with no $k \times k$ grid minor has treewidth at most $6k - 5$.
    Remark that, by \cref{good_square_cor_general}, $G_0'$ does not contain $m+1$ disjoint $\Pi$-nested cycles. Therefore, $G_0'$ contains no $2(m+1) \times 2(m+1)$ grid minor, and it has treewidth at most $12m+7$ by the result stated above.
    
    Let $(T, (V_t)_{t \in T})$ be a tree decomposition of $G_0'$ of width $tw(G_0') \leq 12m+7$. 

    By \cref{balanced_separation_tree_decomposition}, as $tw(G_0')+1 \leq 12m+8 \leq \frac{1}{4((T(g)+1)\times m'+1)} |V(G_0)| \leq \frac{1}{4((T(g)+1)\times m'+1)} |V(G_0')|$ there is a separation $(T_1, ..., T_k)$ of $T$ with:
    \begin{itemize}
        \item $k = |A| + 1$
        \item For every $1 \leq i, j \leq k$, $|\bigcup_{t \in T_i} V_t| \leq 3 |\bigcup_{t \in T_j} V_t|$
        \item For every $1 \leq i \leq k$, $V(T_i) \cap (\bigcup_{1 \leq j \neq i \leq k} V(T_j))$ has size at most $\lfloor\log_{\frac{4}{3}} 3k\rfloor = \lfloor\log_{\frac{4}{3}} 3((tw(G)+1)\times m' +1)\rfloor$
    \end{itemize}

    For $1 \leq i \leq k$, let $G_i = \bigcup_{t \in T_i} V_t$ and let $A_i = \bigcup_{j \neq i} V(G_i) \cap V(G_j)$.
    For $1 \leq i \leq k$, let $G_i'= G_i - A_i$. Remark that the subgraphs $(G_i')_{1 \leq i \leq k}$ of $G'_0$ are pairwise disjoint. 

    Remark that, by the pigeon hole principle, there exists $1 \leq i_0 \leq k$ such that, for all $1 \leq j \leq k_i$, $G_i'$ contains no vertex from $A$.

    \vspace{1em}
    
    Now, we will bound the order of $G'_i$.
    Remark that $G'_i$ is separated from the rest of the graph by the separator $A_i$ of size at most $|A_i| \leq \lfloor\log_{\frac{4}{3}} 3((tw(G)+1)\times m' +1)\rfloor \times (12m+8) \leq A(g)/6$.

    \vspace{1em}

    Let $(D_j)_{1 \leq j \leq p}$ be the 2-connected components of $G'_i$. For $1 \leq j \leq p$, let $C_j$ be the cycle so that $D_j = \text{Int}(C_j, \Pi)$, we call it the exterior cycle of $D_j$.

    We first bound the number of $2$-connected components $(D_j)_{1 \leq j \leq p}$ and, for every $1 \leq j \leq p$, we establish a bound on the size of the separator of $D_j$ from the rest of the graph:

    \begin{claim}
        \label{cl:bound_2connected_components_Hi}
        $p \leq 5|A_i|$ and, for every $1 \leq j \leq p$, $D_j$ has a separator of size at most $6|A_i|$.
    \end{claim}

    \begin{proof_claim}
        Let $F_i$ be the block-cutvertex forest of $G'_i$, that is to say $V(F_i)$ consists of the cutvertices of $G_i'$ together with the $2$-connected components $(D_j)_{1 \leq j \leq p}$ of $G'_i$ and a connected component is incident to all the cutvertices it contains.
        
        For $1 \leq j \leq p$, $D_j$ is separated from the rest of $G$ by attachments in $A_i \cap D_j$ and by the cutvertices that separate it from the other subgraphs in $(D_j)_{1 \leq j \leq p}$.

        In particular, each subgraph $D_j$ corresponding to a leaf of $F_i$ is separated from the rest of $G$ by attachments in $A_i \cap D_j$ and by at most one cutvertex. Because $G$ is $2$-connected, $D_j$ contains at least one attachment of $A_i$ that is distinct from the cutvertex.

        Moreover, if there is a vertex $v$ in $F_i$ corresponding to a cutvertex of $G'_i$ that has degree $2$, then let $j$ and $j'$ be its neighbors in $F_i$ and let $D_j$ and $D_{j'}$ be the corresponding $2$-connected components. Then, $D_j \cup D_{j'}$ contains an attachment of $A_i$ distinct from the $2$-separator of $D_j \cup D_{j'}$. Indeed, if that is not the case, $D_j \cup D_{j'}$ is a subgraph with at least one internal vertex (the cutvertex corresponding to $v$ in $G'_i$) separated from the rest of the graph by a $2$-separator, which is in contradiction with \cref{2_separated_subgraph_is_edge}.

        For $k \geq 0$, let $v_k$ be the number of vertices of $F_i$ of degree $k$, and let $v_k'$ and $v_k''$ be the number of vertices of $F_i$ of degree $k$ which correspond respectively to cutvertices and $2$-connected components of $G'_i$. Remark that $v_k = v_k' + v_k''$. Moreover, remark that $v_0' = v_1' = 0$.

        Because the average degree of a forest is strictly less than $2$, we have 
        \begin{align*}
            &\sum_{k\geq 0}k v_k  < 2 \sum_{k \geq 0} v_k \\ 
            &\sum_{k\geq 3} (k-2) v_k < v_1 = v_1''
        \end{align*}

        Moreover, remark that $\sum_{k\geq 3} k v_k = \sum_{k\geq 3} (k-2) v_k < v_1''$ and $\sum_{k\geq 3} k v_k = \sum_{k\geq 3} (k-2) v_k + 2 \sum_{k \geq 3} v_k < 3v_1''$.
        
        By definition of the block-cutvertices forest $F_i$, we have $2v_2' + \sum_{k \geq 3} k v_k' = v_1'' + \sum_{k \geq 2} k v_k''$. Furthermore, because of the above, $v_0'' + v_1'' \leq |A_i|$ and $v_2' \leq |A_i|$.

        Finally, 
        
        \begin{align*}
            p & = v_0'' + v_1'' + \sum_{k\geq 2} v_k'' \\ 
            & \leq v_0'' + v_1'' + \sum_{k \geq 2} k v_k'' \\
            & = v_0'' + 2v_2' + \sum_{k \geq 3} k v_k' \\
            & \leq v_0'' + 2v_2' + \sum_{k \geq 3} k v_k \\
            & \leq v_0 + 2 |A_i| + 3 v_1'' \\
            & \leq 5 |A_i|
        \end{align*}

        \vspace{1em}

        Let $1 \leq j \leq p$, $D_j$ has a separator which is contained in the union of the at most $5|A_i|$ cutvertices of $D_j$ in $G'_i$ and $A_i$, which has therefore a total size of at most $6|A_i|$ vertices.
    \end{proof_claim}

    \vspace{1em}

    For $1 \leq j \leq p$, remark that $D_j$ is $\Pi$-contractible, $2$-connected and has a separator of size at most $6|A_i| \leq A(g)$.
    We can thus use \cref{order_planar_subgraph_logarithmic_separator} for $D_j$. We get $|V(D_j)|  \leq P(g)$. Therefore, \[V(G'_i) = \sum_{j=1}^p |V(D_j)| \leq 5 |A_i| \times P(g) \leq \frac{5}{6} A(g) P(g)\]

    Finally, \[|V(G_0')| = \sum_{j=1}^k |V(G_j)| \leq 3k \times |V(G_i)| = 3k \times (|V(G_i')| + |A_i|) \leq 3 \times 2a \times \left(\frac{5}{6}A(g) P(g) + \frac{1}{6}A(g)\right) \leq R(g,a)\] 
    Hence, $|V(G_0)| \leq |V(G_0')| \leq R(g,a)$ as announced.
\end{proof}

\vspace{1em}

We conclude by combining \cref{almost_main,main_planar} to find a bound on the order of $G$. Moreover, we extend the bound on the order of $G$ to every excluded minor for $S$ (even those which are not $2$-connected), by using \cref{G_2_connected}.

\vspace{1em}

\begin{prop}
    \label{main_2_connected}
    \[|V(G)| \leq U(g) = Q(g) \times S(g)\]

    with $Q$ defined in \cref{almost_main} and \[S(g) = 42 (m''+1) ((tw(G)+1)\times m' + g) \times \left(\frac{5}{6}A(g) P(g) + \frac{1}{6}A(g)\right)\].
\end{prop}

\begin{proof}
    By \cref{almost_main}, we obtain that either $|V(G)| \leq 4 (4m^2(3g+3)+1)(tw(G)+1) \leq 4 (4m^2(3g+3)+1)(T(g)+1) \leq U(g)$ or that there exists a subgraph $G_0$ of $G$ such that $G_0$ is separated from the rest of the graph by a separator $A$ of size $(tw(G)+1) \times m'$, every bridge on $A$ in $G_0$ is $\Pi$-contractible and such that $|V(G)| \leq Q(g) \times |V(G_0)|$.

    \vspace{1em}

    Let $\mathcal{H}$ be the set of bridges on $A$ in $G_0$ together with their attachments. We partition $\mathcal{H}$ into $m''= \lfloor \log_2 \left(\frac{1}{3}(tw(G)+1)\times m'+\frac{1}{3}\right) \rfloor$ sets $(\mathcal{H}_i)_{0 \leq i \leq m''}$ so that the subgraphs in $\mathcal{H}_i$ for some $0 \leq i \leq m''$ have between $a_i = 3 \times 2^i$ and $2a_i-1 = 3 \times 2^{i+1}-1$ attachments in $A$. Remark that $(\mathcal{H}_i)_{0 \leq i \leq m''}$ is indeed a partition of $\mathcal{H}$: every subgraph $H$ of $\mathcal{H}$ has at least $3$ attachments because $H$ is $\Pi$-contractible and therefore does not have a $2$-separator by \cref{2_separated_subgraph_is_edge}, and every subgraph $H$ of $\mathcal{H}$ has at most $3 \times 2^{m''+1} - 1 \geq 3 \times 2^{\log_2 \left(\frac{1}{3}(tw(G)+1)\times m'+\frac{1}{3}\right)} -1 = (tw(G)+1)\times m' \geq |A|$ attachments.

    \vspace{1em}

    Let's first determine the size $\ell_i$ of $\mathcal{H}_i$ for every $0 \leq i \leq m''$.

    \begin{claim}
        \label{bound_ell_i}
        For $1 \leq i \leq m''$,
        \[ \ell_i \leq \frac{2}{a_i-2}((tw(G)+1)\times m'+g)\]
    \end{claim}

    \begin{proof_claim}
        Let's associate to each $1 \leq j \leq \ell_i$ $a_i$ vertices $A_j$ in $A$ so that $A_j \subseteq H_j$, which is always possible because $H_j$ has at least $a_i$ attachments in $A$.
        
        Let's consider the bipartite graph $B$ with vertex set $A \cup [1, \ell_i]$ and so that, for $a \in A$ and $1 \leq j \leq \ell_i$, we put an edge between $a$ and $j$ if and only if $a \in A_j$.

        Remark that $B$ is a minor of $G_0 \subseteq G$, $B$ therefore has genus at most $g(G) \leq g+2$. 
        Let $\mathcal{F}$ be the set $F(B, \Pi(B))$ of faces of $B$ in $\Pi(B)$. Remark that $2|E(B)| = \sum_{f \in \mathcal{F}} |f| \geq 4 |\mathcal{F}|$ because $B$ is bipartite.
        
        By Euler's formula, we get:

        \begin{align*}
            |V(B)| - |E(B)| + |F| & \geq 2 - (g+2) \\
            |V(B)| - |E(B)| + \frac{1}{2}|E(B)| & \geq -g \\ 
            |E(B)| & \leq 2 (|V(B)| + g) \\
            a_i \ell_i & \leq 2 (|A| + \ell_i + g) \\
            (a_i-2)\ell_i & \leq 2 ((tw(G)+1)\times m' + g) \\
            \ell_i & \leq \frac{2}{a_i-2} ((tw(G)+1)\times m' + g)
        \end{align*}
    \end{proof_claim}

    \vspace{1em}

    For $0 \leq i \leq m''$, let $H_i \in \mathcal{H}_i$ be so that $|V(H_i)| = \max_{H \in \mathcal{H}_i} |V(H)|$.

    Then, we have
    \begin{align*}
        |V(G_0)| & \leq \sum_{i=0}^{m''} |\bigcup_{H \in \mathcal{H}_i} V(H)| \\
        & \leq \sum_{i=0}^{m''} \ell_i |V(H_i)|
    \end{align*}

    For $0 \leq i \leq m''$, we can apply \cref{main_planar} to the subgraph $H_i$ and we obtain that either $|V(H_i)| \leq 4 ((T(g)+1)\times m'+1)(12m+8)$ or $|V(H_i)| \leq R(g,a_i)$. In both cases, we get that $|V(H_i)| \leq R(g,a_i)$.

    \vspace{1em}

    We obtain 
    \begin{align*}
        |V(G_0)| & \leq \sum_{i=0}^{m''} \ell_i |V(H_i)| \\
        & \leq \sum_{i=0}^{m''} \frac{2}{a_i-2} ((tw(G)+1)\times m' + g) \times 3(2a_i+1) \times \left(\frac{5}{6}A(g) P(g) + \frac{1}{6}A(g)\right) \\
        & \leq \sum_{i=0}^{m''} 42 ((tw(G)+1)\times m' + g) \times \left(\frac{5}{6}A(g) P(g) + \frac{1}{6}A(g)\right) \\
        & \leq 42 (m''+1) ((tw(G)+1)\times m' + g) \times \left(\frac{5}{6}A(g) P(g) + \frac{1}{6}A(g)\right) \\
        & = S(g)
    \end{align*}

    Finally, $|V(G)| \leq Q(g) \times S(g) = U(g)$.
\end{proof}

\begin{reptheo}{main}
    Let $S$ be a given surface of Euler genus $g$. Every excluded minor for $S$ is of order at most $U(g) = g^2 \times (\log g)^{O\left((\log\log g)^3\right)} = g^{2+o(1)}$ for every $\varepsilon >0$.
\end{reptheo}

\begin{proof}
    By \cref{main_2_connected}, we get that any $2$-connected excluded minor for a surface $S$ of genus $g$ ($g \geq 0$) is of order at most $U(g)$.  Moreover, as $U$ is an increasing function with the property that $U(g_1+g_2) \geq U(g_1)+U(g_2)$, then by \cref{G_2_connected} we get that any excluded minor for a surface $S$ of genus $g$ is of order bounded by the same function $U(g)$.
\end{proof}

\section{Conclusion} \label{sec:conclusion}

We have shown in this paper that the order of $G$ is bounded by an almost quadratic function of $g$. However, determining the optimal bound remains an open problem. We conjecture that the optimal bound is almost linear.

\begin{prob}
Is there an almost linear bound on the order of $G$?
\end{prob}

The $O(g^2)$ factor arises from the need to identify a $\Pi$-contractible $2$-connected subgraph of $G$ that admits a separator of size logarithmic in $g$, and whose order constitutes a polynomial fraction of $|V(G)|$. We don't believe that a more refined analysis could reduce the exponent further and we think that achieving an exponent below $2$ will require the use of additional new techniques.

\bibliographystyle{alphaurl}
\bibliography{bibliography}

\end{document}